\documentclass[a4paper]{article}

\usepackage[latin1]{inputenc} 
\usepackage{graphicx}
\usepackage[pdftex,bookmarks]{hyperref}
\usepackage{amsmath}    
\usepackage{amssymb}
\usepackage{amsmath,amsthm}
\usepackage{textcomp}
\usepackage[all]{xy}
\usepackage{geometry}

\newtheorem{teo}{Theorem}[section]
\newtheorem{defi}{Definition}[section]

\newtheorem{cor}{Corollary}[section]
\newtheorem{prop}{Proposition}[section]

\newtheorem{rem} {Remark}[section]

\newcommand{\nn}{\nonumber}

\DeclareMathOperator{\im}{im}
\DeclareMathOperator{\ind}{ind}
\DeclareMathOperator{\dvol}{dvol}

\DeclareMathOperator{\supp}{supp}

\DeclareMathOperator{\id}{Id}
\DeclareMathOperator{\vol}{vol}
\DeclareMathOperator{\reg}{reg}
\DeclareMathOperator{\sing}{sing}

\DeclareMathOperator{\Tr}{Tr}

\DeclareMathOperator{\abs}{abs}
\DeclareMathOperator{\rel}{rel}

\title{\huge \bf  Degenerating Hermitian metrics and spectral geometry of the canonical bundle}
\author{Francesco Bei  \bigskip \\
Institut Camille Jordan, Universit\'e Lyon 1\\ E-mail addresses: \ bei@math.univ-lyon1.fr\     francescobei27@gmail.com }
\date{}

\begin{document}

\maketitle

\begin{abstract}
Let $(X,h)$ be  a compact and irreducible Hermitian complex space of complex dimension $m$. In this paper we are interested in the Dolbeault operator acting on the space of $L^2$ sections of the canonical bundle of $\reg(X)$, the regular part of $X$. More precisely let $\overline{\mathfrak{d}}_{m,0}:L^2\Omega^{m,0}(\reg(X),h)\rightarrow L^2\Omega^{m,1}(\reg(X),h)$ be an arbitrarily fixed closed extension of $\overline{\partial}_{m,0}:L^2\Omega^{m,0}(\reg(X),h)\rightarrow L^2\Omega^{m,1}(\reg(X),h)$ where the  domain of the latter operator is $\Omega_c^{m,0}(\reg(X))$.  We establish various properties such as closed range of $\overline{\mathfrak{d}}_{m,0}$, compactness of the  inclusion  $\mathcal{D}(\overline{\mathfrak{d}}_{m,0})\hookrightarrow L^2\Omega^{m,0}(\reg(X),h)$ where $\mathcal{D}(\overline{\mathfrak{d}}_{m,0})$, the domain of $\overline{\mathfrak{d}}_{m,0}$, is endowed with the corresponding graph norm, and   discreteness of the spectrum of the  associated Hodge-Kodaira Laplacian $\overline{\mathfrak{d}}_{m,0}^*\circ \overline{\mathfrak{d}}_{m,0}$ with an estimate for the growth of its eigenvalues. Several corollaries such as trace class property for   the heat operator associated to $\overline{\mathfrak{d}}_{m,0}^*\circ \overline{\mathfrak{d}}_{m,0}$, with an estimate for its  trace, are derived. Finally in the last part we provide several applications to the Hodge-Kodaira Laplacian  in the setting of both compact irreducible Hermitian complex spaces with isolated singularities and  complex  projective surfaces.
\end{abstract}

\noindent\textbf{Keywords}: Hermitian complex space,  Hermitian pseudometric, canonical bundle, $\overline{\partial}$-operator, Hodge-Kodaira Laplacian, parabolicity, complex  projective surface, Fubini-Study metric.\\

\noindent\textbf{Mathematics subject classification}:  32W05, 32W50, 35P15,  58J35.

\tableofcontents

\section*{Introduction}
Consider a complex projective variety $V\subset \mathbb{C}\mathbb{P}^n$. The regular part of $V$, $\reg(V)$, comes equipped with a natural K\"ahler metric $g$,  which is the one induced by  the Fubini-Study metric of $\mathbb{C}\mathbb{P}^n$. 
In particular, whenever $V$ has a nonempty singular set, we get  an incomplete K\"ahler manifold of finite volume. In the seminal papers \cite{CGM} and  \cite{Mac}, given a singular projective variety $V$, many questions with a rich interaction of topology and analysis, for instance  intersection cohomology, $L^2$-cohomology and  Hodge theory,  have been raised for the incomplete K\"ahler manifold $(\reg(V),g)$.   Some of the most important among them  are the Cheeger-Goresky-MacPherson's conjecture and  the MacPherson's conjecture. The former, which is still open,   says that the maximal $L^2$-de Rham cohomology groups  of $(\reg(V),g)$ are isomorphic to the middle perversity intersection cohomology groups of $V$ while the latter, proved in \cite{PS}, asks whether the $L^2$-$\overline{\partial}$-cohomology groups in bidegree $(0,q)$ of $(\reg(V),g)$ are  isomorphic to the $(0,q)$-Dolbeault cohomology groups of $\tilde{V}$, where $\tilde{V}$ is a resolution of $V$  \`a la Hironaka. Related to these problems there are many other interesting and deep analytic questions. We can mention for instance the $L^2$-Stokes theorem which asks whether  the maximal and minimal extension of the de Rham differential $d$ are the same, the existence of a $L^2$-Hodge decomposition for the $L^2$-de Rham cohomology of $(\reg(V),g)$ in terms of the $L^2$-${\overline{\partial}}$-cohomology of $(\reg(V),g)$, the existence of  self-adjoint extensions of the  Hodge-de Rham operator  $d+d^t$ and  the Hodge-Dolbeault operator $\overline{\partial}+\overline{\partial}^t$ with discrete spectrum, the properties of the heat operator associated to some self-adjoint extension of the Laplacian and so on. Moreover we point out that many of these problems  admit a natural extension in the more general setting of  Hermitian complex spaces. Several papers, during the last thirty years, have been devoted to these questions. Without any goal of completeness we can recall here  \cite{HS} and \cite{TOh} which concern the Cheeger-Goresky-MacPherson's conjecture, \cite{BPS}, \cite{PH},\cite{OVV},  \cite{OV},  \cite{PS} and \cite{JRu} devoted to the $L^2$-$\overline{\partial}$-cohomology, \cite{OvRu} and \cite{Ruppe} concerning the $\overline{\partial}$-operator on Hermitian complex spaces,  \cite{BLE}, \cite{GL}, \cite{PaS} and \cite{JRU} dealing with the $L^2$-Hodge decomposition and the $L^2$-Stokes theorem and finally   \cite{BRLU}, \cite{LT}, \cite{MasNag} and  \cite{Pa}  devoted to the heat operator.\\ Now, after this brief overview of the literature, we carry on by describing the aim of this paper. Given a compact and  irreducible Hermitian complex space $(X,h)$ of complex dimension $m$ we are interested in the Dolbeault operator $\overline{\partial}_{m,0}$ acting on the space of $L^2$-sections of the canonical bundle of $\reg(X)$, the regular part of $X$. More precisely our point of view is to consider $\overline{\partial}_{m,0}$ as an unbounded and densely defined operator 
\begin{equation}
\label{Amburgo}
\overline{\partial}_{m,0}:L^2\Omega^{m,0}(\reg(X),h)\rightarrow L^2\Omega^{m,1}(\reg(X),h)
\end{equation}
 with  domain $\Omega^{m,0}_c(\reg(X))$, the space of smooth sections with compact support of $\Lambda^{m,0}(\reg(X))$. Labeling by  $\overline{\mathfrak{d}}_{m,0}:L^2\Omega^{m,0}(\reg(X),h)\rightarrow L^2\Omega^{m,1}(\reg(X),h)$ any closed extension of \eqref{Amburgo}, we are interested in properties such as closed range of $\overline{\mathfrak{d}}_{m,0}$, compactness of the  inclusion  $\mathcal{D}(\overline{\mathfrak{d}}_{m,0})\hookrightarrow L^2\Omega^{m,0}(\reg(X),h)$ where $\mathcal{D}(\overline{\mathfrak{d}}_{m,0})$, the domain of $\overline{\mathfrak{d}}_{m,0}$, is endowed with the corresponding graph norm,  discreteness of the spectrum of  $\overline{\mathfrak{d}}_{m,0}^*\circ \overline{\mathfrak{d}}_{m,0}:L^2\Omega^{m,0}(\reg(X),h)\rightarrow L^2\Omega^{m,0}(\reg(X),h)$, estimates of the growth of the eigenvalues and so on. \\Let  us now go  more  into  the  details  by explaining  the  structure  of  this  paper.   The   first  section  is  devoted  to the background material. We collect some basic definitions and notions concerning differential operators with particular regard to the case of the $\overline{\partial}$-operator. The second section contains some abstract results  that will be  used later on in the paper. In the third section we recall the notion of parabolic Riemannian manifold $(M,g)$, see Def. \ref{para}, and then we proceed by studying  the Hodge-Dolbeault operator acting on an open, parabolic and dense subset of a compact Hermitian manifold. Furthermore the remaining part of the third section collects some useful propositions  in the realm of Hermitian manifolds. The forth section contains the main results of this paper whose applications will provide some satisfactory answers for the questions raised about the operator \eqref{Amburgo}. More precisely, the forth section starts with the notion of Hermitian pseudometric: given a complex manifold $M$, a Hermitian pseudometric $h$ on $M$ is nothing but a positive semidefinite Hermitian product on $M$ which is positive definite on an open and dense subset of $M$. As we will see later on, by virtue of Hironaka resolution, this is a convenient set to deal with many problems involving the $\overline{\partial}$-operator on Hermitian complex spaces. Within this framework the first theorem proved in the forth section is the following:
\begin{teo}
\label{mumford}
Let $(M,g)$ be a compact Hermitian manifold of complex dimension $m$. Let $h$ be a Hermitian pseudometric on $M$  and let $A_h:=M\setminus Z_h$ with $Z_h$ the degeneracy locus of $h$, see Def. \ref{pseudohermi}. Let  $(E,\rho)$ be a Hermitian holomorphic vector bundle over $M$. Assume that $(A_h,g|_{A_h})$ is parabolic.  Let 
\begin{equation}
\label{anycc}
\overline{\mathfrak{d}}_{E,m,0}:L^2\Omega^{m,0}(A_h,E|_{A_h},h|_{A_h})\rightarrow L^2\Omega^{m,1}(A_h,E|_{A_h},h|_{A_h})
\end{equation}
be any closed extension of $\overline{\partial}_{E,m,0}:L^2\Omega^{m,0}(A_h,E|_{A_h},h|_{A_h})\rightarrow L^2\Omega^{m,1}(A_h,E|_{A_h},h|_{A_h})$ where the  domain of the latter operator is  $\Omega_c^{m,0}(A_h,E|_{A_h})$.
Let  
\begin{equation}
\label{zitacc}
\overline{\partial}_{E,m,0}:L^2\Omega^{m,0}(M,E,g)\rightarrow L^2\Omega^{m,1}(M,E,g)
\end{equation}
be the unique closed extension of  $\overline{\partial}_{E,m,0}:\Omega^{m,0}(M,E)\rightarrow \Omega^{m,1}(M,E)$ where the latter operator is viewed as an unbounded and densely defined operator acting between $L^2\Omega^{m,0}(M,E,g)$ and $L^2\Omega^{m,1}(M,E,g)$. Finally let $\mathcal{D}(\overline{\mathfrak{d}}_{E,m,0})$  and $\mathcal{D}(\overline{\partial}_{E,m,0})$  be the domains of  \eqref{anycc} and \eqref{zitacc} respectively. Then the following properties hold true:
\begin{enumerate}
\item We have a continuous inclusion $\mathcal{D}(\overline{\mathfrak{d}}_{E,m,0})\hookrightarrow \mathcal{D}(\overline{\partial}_{E,m,0})$ where  each domain is endowed with the corresponding graph norm.  Moreover on $\mathcal{D}(\overline{\mathfrak{d}}_{E,m,0})$ the operator \eqref{zitacc}  coincides with  the operator \eqref{anycc}.
\item The inclusion $\mathcal{D}(\overline{\mathfrak{d}}_{E,m,0})\hookrightarrow L^2\Omega^{m,0}(M,E,g)$ is a compact operator where $\mathcal{D}(\overline{\mathfrak{d}}_{E,m,0})$ is endowed with the corresponding graph norm.
\item Let $\overline{\mathfrak{d}}_{E,m,0}^*:L^2\Omega^{m,1}(A_h,E|_{A_h},h|_{A_h})\rightarrow L^2\Omega^{m,0}(A_h,E|_{A_h},h|_{A_h})$ be the adjoint of  \eqref{anycc}. Then  the operator
 \begin{equation}
\label{anylapcc}
\overline{\mathfrak{d}}_{E,m,0}^*\circ\overline{\mathfrak{d}}_{E,m,0}:L^2\Omega^{m,0}(A_h,E|_{A_h},h|_{A_h})\rightarrow L^2\Omega^{m,0}(A_h,E|_{A_h},h|_{A_h})
\end{equation} whose domain is defined as $\{s\in \mathcal{D}(\overline{\mathfrak{d}}_{E,m,0}):\ \overline{\mathfrak{d}}_{E,m,0}s\in \mathcal{D}(\overline{\mathfrak{d}}_{E,m,0}^*)\}$, has discrete spectrum.
\end{enumerate}
\end{teo}
We point out explicitly that in the previous theorem the assumption concerning the parabolicity of $(A_h,g|_{A_h})$ does not depend on the particular Hermitian metric $g$ that we fix on $M$. Indeed if $g'$ is another Hermitian metric on $M$ then, since $g$ and $g'$ are quasi-isometric on $M$, we have that $(A_h,g|_{A_h})$ is parabolic if and only if $(A_h,g'|_{A_h})$ is parabolic.
Consider again the setting of Theorem \ref{mumford} and let $\overline{\partial}^t_{E,m,0}$ be the formal adjoint of $\overline{\partial}_{E,m,0}$ with respect to $g$. Let $\Delta_{\overline{\partial},E,m,0}:\Omega^{m,0}(M,E)\rightarrow \Omega^{m,0}(M,E)$, $\Delta_{\overline{\partial},E,m,0}=\overline{\partial}_{E,m,0}^t\circ\overline{\partial}_{E,m,0}$ be the Hodge-Kodaira Laplacian in bidegree $(m,0)$. Since $M$ is compact and  $\Delta_{\overline{\partial},E,m,0}$ is elliptic and formally self-adjoint we have that $\Delta_{\overline{\partial},E,m,0}$, acting on $L^2\Omega^{m,0}(M,E,g)$ with  domain $\Omega^{m,0}(M,E)$, is essentially self-adjoint. With  
\begin{equation}
\label{crematico}
\Delta_{\overline{\partial},E,m,0}:L^2\Omega^{m,0}(M,E,g)\rightarrow L^2\Omega^{m,0}(M,E,g)
\end{equation}
we mean its unique closed (and therefore self-adjoint) extension. We are  now in the position to recall the second theorem  proved in the forth section.

\begin{teo}
\label{prince}
In the  setting of Theorem \ref{mumford}. Let $$0\leq \mu_1\leq \mu_2\leq...\leq \mu_k\leq...$$ be the eigenvalues of \eqref{crematico} and let  $$0\leq \lambda_1\leq \lambda_2\leq...\leq \lambda_{k}\leq...$$ be the eigenvalues of  \eqref{anylapcc}
Then there exists a constant $\gamma>0$ such that  for every $k\in \mathbb{N}$ we have the following inequality:
\begin{equation}
\label{inesa}
\gamma\lambda_k\geq \mu_k.
\end{equation}
Moreover we have the following asymptotic inequality:
\begin{equation}
\label{asinello}
\lim \inf \lambda_k k^{-\frac{1}{m}}>0
\end{equation} 
as $k\rightarrow \infty$.
\end{teo}

Finally the remaining part of the forth section contains various corollaries and remarks. In particular we show that the heat operator associated to $\overline{\mathfrak{d}}_{E,m,0}^*\circ\overline{\mathfrak{d}}_{E,m,0}$ is a trace class operator and moreover we provide an estimate for its trace.    In the fifth and last section of this paper we collect various applications of Th. \ref{mumford} and Th. \ref{prince}. Its first part is devoted to the $\overline{\partial}$-operator acting on the space of $L^2$-sections of the canonical bundle of $\reg(X)$, see \eqref{Amburgo}. In particular we prove  the following result that, for the sake of brevity,  here is formulated only in the version where the canonical bundle is untwisted.  For the more general version we refer to Th. \ref{canonical}.
\begin{teo}
\label{canonicale}
Let $(X,h)$ be a compact and irreducible Hermitian complex space of complex dimension $m$. Consider the Dolbeault operator  $\overline{\partial}_{m,0}:L^2\Omega^{m,0}(\reg(X),h)\rightarrow L^2\Omega^{m,1}(\reg(X),h)$ with  domain $\Omega^{m,0}_c(\reg(X))$ and let 
\begin{equation}
\label{Amburgoe}
 \overline{\mathfrak{d}}_{m,0}:L^2\Omega^{m,0}(\reg(X),h)\rightarrow L^2\Omega^{m,1}(\reg(X),h)
\end{equation}
 be any of its closed extensions. The following properties hold true:
\begin{enumerate}
\item The inclusion $\mathcal{D}(\overline{\mathfrak{d}}_{m,0})\hookrightarrow L^2\Omega^{m,0}(\reg(X),h)$ is a compact operator where $\mathcal{D}(\overline{\mathfrak{d}}_{m,0})$ is endowed with the corresponding graph norm.
\item Let $\overline{\mathfrak{d}}_{m,0}^*:L^2\Omega^{m,1}(\reg(X),h)\rightarrow L^2\Omega^{m,0}(\reg(X),h)$ be the adjoint of  \eqref{Amburgoe}. Then  the operator 
\begin{equation}
\label{anylapze}
\overline{\mathfrak{d}}_{m,0}^*\circ\overline{\mathfrak{d}}_{m,0}:L^2\Omega^{m,0}(\reg(X),h)\rightarrow L^2\Omega^{m,0}(\reg(X),h)
\end{equation} whose domain is defined as $\{s\in \mathcal{D}(\overline{\mathfrak{d}}_{m,0}):\ \overline{\mathfrak{d}}_{m,0}s\in \mathcal{D}(\overline{\mathfrak{d}}_{m,0}^*)\}$,  has discrete spectrum.
\end{enumerate}
Let now $$0\leq \lambda_1\leq \lambda_2\leq \lambda_3\leq...$$ be the eigenvalues of \eqref{anylapze}. Then we have the following asymptotic inequality 
\begin{equation}
\label{rospoe}
\lim \inf \lambda_k k^{-\frac{1}{m}}>0
\end{equation} 
as $k\rightarrow \infty$.\\
Finally consider the heat operator $$e^{-t\overline{\mathfrak{d}}_{m,0}^*\circ\overline{\mathfrak{d}}_{m,0}}:L^2\Omega^{m,0}(\reg(X),h)\rightarrow L^2\Omega^{m,0}(\reg(X),h)$$ associated to \eqref{anylapze}. We have the following properties:
\begin{enumerate}
\item $e^{-t\overline{\mathfrak{d}}_{m,0}^*\circ\overline{\mathfrak{d}}_{m,0}}:L^2\Omega^{m,0}(\reg(X),h)\rightarrow L^2\Omega^{m,0}(\reg(X),h)$ is a trace class operator.
\item $\Tr(e^{-t\overline{\mathfrak{d}}_{m,0}^*\circ\overline{\mathfrak{d}}_{m,0}})\leq C t^{-m}$ for $t\in (0,1]$ and  for some constant $C>0$.
\end{enumerate}
\end{teo}

We stress on the fact that Th. \ref{canonicale} does not require assumptions on $\sing(X)$ nor on the dimension of $X$. 
In the second part of the fifth section, combining  Th. \ref{canonicale} with other theorems already available in the literature, we show the existence of self-adjoint extensions with discrete spectrum for the Hodge-Kodaira Laplacian in the framework of  compact and irreducible Hermitian complex spaces with isolated singularities.  For the definition of  Friedrich extension and  absolute extension we refer to  Prop. \ref{fall} and \eqref{asdf}.
\begin{teo}
\label{tommasodacquino}
Let $(X,h)$ be a compact and irreducible Hermitian complex space of complex dimension $m$. Assume that $\sing(X)$ is made of isolated singularities. Then  we have the following properties:
\begin{enumerate}
\item $\Delta_{\overline{\partial},m,q,\abs}:L^2\Omega^{m,q}(\reg(X),h)\rightarrow L^2\Omega^{m,q}(\reg(X),h)$  has discrete spectrum  for each $q=0,...,m$.
\item $\overline{\partial}_{m,\max}+\overline{\partial}^t_{m,\min}:L^2\Omega^{m,\bullet}(\reg(X),h)\rightarrow L^2\Omega^{m,\bullet}(\reg(X),h)$ has  discrete spectrum.
\item $\Delta_{\overline{\partial},m,q}^{\mathcal{F}}:L^2\Omega^{m,q}(\reg(X),h)\rightarrow L^2\Omega^{m,q}(\reg(X),h)$ has discrete spectrum for each $q=0,...,m$.
\end{enumerate}
\end{teo}
Finally in the last part of the fifth section, joining again our results with others already available in the literature, we provide a quite accurate study   of the Hodge-Kodaira Laplacian  on complex projective surfaces. We conclude this introduction by summarizing some of these results in the next two theorems.  
\begin{teo}
\label{lillottinab}
Let $V\subset \mathbb{C}\mathbb{P}^n$ be a complex projective surface and let $h$ be the K\"ahler metric on $\reg(V)$ induced by the Fubini-Study metric of $\mathbb{C}\mathbb{P}^n$. Then for each $q=0,1,2$  the operator
\begin{equation}
\label{resilevb}
\Delta_{\overline{\partial},2,q,\abs}:L^2\Omega^{2,q}(\reg(V),h)\rightarrow L^2\Omega^{2,q}(\reg(V),h)
\end{equation}
has discrete spectrum.
Let now $$0\leq \lambda_1\leq \lambda_2\leq...\leq \lambda_k\leq...$$ be the eigenvalues of \eqref{resilevb}. Then we have the following asymptotic inequality
\begin{equation}
\label{resilesvb}
\lim \inf \lambda_k k^{-\frac{1}{2}}>0
\end{equation}
as $k\rightarrow \infty$.\\ Finally consider  the heat operator associated to \eqref{resilevb}
\begin{equation}
\label{resilexvb}
e^{-t\Delta_{\overline{\partial},2,q,\abs}}:L^2\Omega^{2,q}(\reg(V),h)\rightarrow L^2\Omega^{2,q}(\reg(V),h).
\end{equation}
Then \eqref{resilexvb} is a trace class operator and its trace satisfies the following estimate
\begin{equation}
\label{resilezvb}
\Tr(e^{-t\Delta_{\overline{\partial},2,q,\abs}})\leq C_qt^{-2}
\end{equation}
for $t\in (0,1]$ and some constant $C_q>0$.
\end{teo}

We point out that  Th. \ref{lillottinab} does not require assumptions on $\sing(V)$. On the other hand, assuming moreover that $\sing(V)$ is made of isolated singularities, we have also the following result.

\begin{teo}
\label{tempoditurbavb}
Let $V\subset \mathbb{C}\mathbb{P}^n$ be a complex projective surface and let $h$ be the K\"ahler metric on $\reg(V)$ induced by the Fubini-Study metric of $\mathbb{C}\mathbb{P}^n$. Assume that $V$ has only isolated singularities. Then for each $q=0,1,2$   the operator 
\begin{equation}
\label{miguelvb}
\Delta_{\overline{\partial},0,q,\abs}:L^2\Omega^{0,q}(\reg(V),h)\rightarrow L^2\Omega^{0,q}(\reg(V),h)
\end{equation}
has discrete spectrum.
Let  $$0\leq \lambda_1\leq \lambda_2\leq...\leq \lambda_k\leq...$$ be the eigenvalues of \eqref{miguelvb}. Then we have the following asymptotic inequality
\begin{equation}
\label{miguelsvb}
\lim \inf \lambda_k k^{-\frac{1}{2}}>0
\end{equation}
as $k\rightarrow \infty$.\\ Finally consider the heat operator associated to \eqref{miguelvb}
\begin{equation}
\label{miguelzvb}
e^{-t\Delta_{\overline{\partial},0,q,\abs}}:L^2\Omega^{0,q}(\reg(V),h)\rightarrow L^2\Omega^{0,q}(\reg(V),h).
\end{equation}
Then \eqref{miguelzvb} is a trace class operator and its trace satisfies the following estimate
\begin{equation}
\label{miguelxvb}
\Tr(e^{-t\Delta_{\overline{\partial},0,q,\abs}})\leq C_qt^{-2}
\end{equation}
for $t\in (0,1]$ and some constant $C_q>0$.
\end{teo}

\vspace{1 cm}

\noindent\textbf{Acknowledgments.}  I wish to thank Paolo Piazza and Jochen Br\"uning  for interesting discussions. This research has been financially supported by the SFB 647 : Raum-Zeit-Materie.

\section{Background material}

We start by briefly recalling some basic notions about $L^p$-spaces  and differential operators. We refer for instance to \cite{FraBei} and the bibliography cited there. Let $(M,g)$  be an open and possibly incomplete Riemannian manifold of dimension $m$. Let $E$ be a vector bundle over $M$ of rank $k$ and let $\rho$ be a metric on $E$, Hermitian if $E$ is a complex vector bundle, Riemannian if $E$ is a real vector bundle. Let $\dvol_g$ be the one-density associated to $g$.   A section $s$ of $E$ is said measurable if, for any trivialization $(U,\phi)$ of $E$, $\phi(s|_U)$ is given by a $k$-tuple of measurable functions.  Given a measurable section $s$ let $|s|_{\rho}$ be defined as $|s|_{\rho}:= (\rho(s,s))^{1/2}$. Then for every $p$, $1\leq p< \infty$ we can define $L^{p}(M,E,g)$ as the  space  of measurable sections $s$ such that    $$\|s\|_{L^{p}(M,E,g)}:=\left(\int_{M}|s|_{\rho}^p\dvol_g\right)^{1/p}<\infty.$$
For each $p\in [1, \infty)$ we have a Banach space,  for each $p\in (1, \infty)$  we have a reflexive Banach space and  in the case $p=2$ we have a Hilbert space whose inner product is given by $$\langle s, t \rangle_{L^2(M,E,g)}:= \int_{M}\rho(s,t)\dvol_g.$$ Moreover $C^{\infty}_c(M,E)$,  the space of smooth sections with compact support,  is dense in $L^p(M,E,g)$ for $p\in [1,\infty).$ Finally $L^{\infty}(M,E,\rho)$ is defined as the space of measurable sections whose essential supp is bounded. 
Also in this case we get a Banach space. Clearly, when $p\in [1,\infty)$,  all the spaces we defined so far depend on $M$, $E$, $\rho$ and $g$ but in order to have a lighter notation we prefer to write $L^p(M,E,g)$ instead of $L^p(M,E,\rho,g)$.\\ 
Let now $F$ be another vector bundle over $M$ endowed with a metric $\tau$. Let $P: C^{\infty}_c(M,E)\longrightarrow  C^{\infty}_c(M,F)$ be a differential operator of order $d$. The formal adjoint of $P$ $$P^t: C^{\infty}_c(M,F)\longrightarrow C^{\infty}_c(M,E)$$ is the differential operator uniquely characterized by the following property: for each $u\in C^{\infty}_c(M,E)$ and for each $v\in C^{\infty}_c(M,F)$ we have  $$\int_{M}\rho(u, P^tv)\dvol_g=\int_M\tau(Pu,v)\dvol_g.$$ We can look at $P$ as an unbounded, densely defined and closable operator  acting between $L^2(M,E,g)$ and $L^2(M,F,g)$. In general $P$ admits several different closed extensions. We recall now the definitions of the maximal and the minimal one.
The domain of the  maximal extension of $P:L^2(M,E,g)\longrightarrow L^2(M,F,g)$ is defined as\\ 
\begin{align}
\label{ner}
& \mathcal{D}(P_{\max}):=\{s\in L^{2}(M,E,g): \text{there is}\ v\in L^2(M,F,g)\ \text{such that}\ \int_{M}\rho(s,P^t\phi)\dvol_g=\\
 \nn &=\int_{M}\tau(v,\phi)\dvol_g\ \text{for each}\ \phi\in C^{\infty}_c(M,F,g)\}.\ \text{In this case we put}\ P_{\max}s=v.
\end{align} In other words the maximal extension of $P$ is the one defined in the distributional sense.\\The domain of the minimal extension of $P:L^2(M,E,g)\longrightarrow L^2(M,F,g)$ is defined as\\ 
\begin{align}
\label{spinaci}
& \mathcal{D}(P_{\min}):=\{s\in L^{2}(M,E,g)\ \text{such that there is a sequence}\ \{s_i\}\in C_c^{\infty}(M,E)\ \text{with}\ s_i\rightarrow s\\ \nn & \text{in}\ L^{2}(M,E,g)\ \text{and}\ Ps_i\rightarrow w\ \text{in}\ L^2(M,F,g)\ \text{to some }\ w\in L^2(M,F,g)\}.\ \text{We put}\ P_{\min}s=w.
\end{align} Briefly the minimal extension of $P$ is the closure of $C^{\infty}_c(M,E)$ under the graph norm $\|s\|_{L^2(M,E,g)}+\|Ps\|_{L^2(M,F,g)}$.  It is immediate to check that 
\begin{equation}
\label{emendare}
P_{\max}^*=P^t_{\min}\ \text{and that}\  P_{\min}^*=P^t_{\max}
\end{equation}
 that is  $P^t_{\max/\min}:L^2(M,F,g)\rightarrow L^2(M,E,g)$ is the Hilbert space adjoint of $P_{\min/\max}$ respectively. Moreover we have the following two $L^2$-orthogonal decompositions for $L^2(M,E,g)$
\begin{equation}
\label{fantaghi}
L^2(M,E)=\ker(P_{\min/\max})\oplus \overline{\im(P^t_{\max/\min})}.
\end{equation}

Before to proceed by recalling some general properties we add the following remark.
\begin{rem}
\label{mogimo}
In this paper, when we will say that  a closed operator $\overline{P}:L^2(M,E,g)\rightarrow L^2(M,F,g)$ with domain $\mathcal{D}(\overline{P})$ is a closed extension of $P:C^{\infty}_c(M,E)\rightarrow C^{\infty}_c(M,F)$, we will always mean  that $P_{\max}$ is defined on $\mathcal{D}(\overline{P})$, and that $P_{\max}|_{\mathcal{D}(\overline{P})}=\overline{P}$. Note that $\overline{P}$ might be $P_{\min}$ or $P_{\max}$.
\end{rem}
We have now the   following propositions:

\begin{prop}
\label{partiro}
Let $(M,g)$,  $(E,\rho)$ and $(F,\tau)$ be as above.  Let $P:C^{\infty}_c(M,E)\rightarrow C^{\infty}_c(M,F)$ be a differential operator such that  $P^t\circ P:C^{\infty}_c(M,E)\rightarrow C^{\infty}_c(M,E)$ is elliptic. Let $\overline{P}:L^2(M,E,g)\rightarrow L^2(M,F,g)$ be a closed extension of $P$.  Let $\overline{P}^*$ be the Hilbert space adjoint of $\overline{P}$.  Then $C^{\infty}(M,E)\cap \mathcal{D}(\overline{P}^*\circ \overline{P})$ is dense in $\mathcal{D}(\overline{P})$ with respect to the graph norm of $\overline{P}$. In particular we have that  $C^{\infty}(M,E)\cap \mathcal{D}(P_{\max/\min})$ is dense in $\mathcal{D}(P_{\max/\min})$ with respect to the graph norm of $P_{\max/\min}$.
\end{prop}

\begin{proof}
See Prop. 2.1 in \cite{FraBei}.
\end{proof}

\begin{prop}
\label{carmina}
Let $(M,g)$,  $(E,\rho)$ and $(F,\tau)$ be as above.  Let  $ P:C^{\infty}_c(M,E)\rightarrow C^{\infty}_c(M,F)$ be  a first order differential operator. Let $s\in\mathcal{D}(P_{\max})$. Assume that there is  an open subset $U\subset M$  with compact closure such that $s|_{M\backslash \overline{U}}=0$. Then $s\in\mathcal{D}(P_{\min})$.

\end{prop}

\begin{proof}
The statement follows by   Lemma 2.1 in \cite{GL}.
\end{proof}

In the remaining part of this introductory section we specialize to the case of complex manifolds and to the natural differential operators appearing in this setting. Our aim here is to  introduce   some notations and to recall  some results from the general theory of Hilbert complexes applied to the Dolbeault complex.  We refer to \cite{BL} for the proofs. Assume  that $(M,g)$ is a complex manifold of real dimension $2m$.  As usual with $\Lambda^{p,q}(M)$ we denote the bundle $\Lambda^p(T^{1,0}M)^*\otimes \Lambda^q(T^{0,1}M)^*$ and by $\Omega^{p,q}(M)$, $\Omega^{p,q}_c(M)$ we denote respectively the space of sections, sections with compact support,  of $\Lambda^{p,q}(M)$. On the bundle $\Lambda^{p,q}(M)$ we consider the Hermitian metric induced by $g$ and with a little abuse of notation we still label it by $g$. With $L^2\Omega^{p,q}(M,g)$ we denote the Hilbert space of $L^2$-$(p,q)$-forms. The Dolbeault operator acting on $(p,q)$-forms is labeled by $\overline{\partial}_{p,q}:\Omega^{p,q}(M)\rightarrow \Omega^{p,q+1}(M)$ and similarly  we have the operator $\partial_{p,q}:\Omega^{p,q}(M)\rightarrow \Omega^{p+1,q}(M)$. When we look at $\overline{\partial}_{p,q}:L^2\Omega^{p,q}(M,g)\rightarrow L^2\Omega^{p,q+1}(M,g)$ as an unbounded and densely defined operator with domain $\Omega_c^{p,q}(M)$ we label by $\overline{\partial}_{p,q,\max/\min}:L^2\Omega^{p,q}(M,g)\rightarrow L^2\Omega^{p,q+1}(M,g)$ respectively its maximal and minimal extension. Analogous meaning has the notation $\partial_{p,q,\max/\min}:L^2\Omega^{p,q}(M,g)\rightarrow L^2\Omega^{p+1,q}(M,g)$. In the case of functions we will simply write $\overline{\partial}:C^{\infty}(M)\rightarrow \Omega^{0,1}(M)$, $\overline{\partial}_{\max/\min}:L^2(M,g)\rightarrow L^2\Omega^{0,1}(M,g)$ and analogously $\partial:C^{\infty}(M)\rightarrow \Omega^{1,0}(M)$ and  $\partial_{\max/\min}:L^2(M,g)\rightarrow L^2\Omega^{1,0}(M,g)$. With $\overline{\partial}_{p,q}^t:\Omega^{p,q+1}_c(M)\rightarrow \Omega^{p,q}_c(M)$ and $\partial_{p,q}^t:\Omega^{p+1,q}_c(M)\rightarrow \Omega^{p,q}_c(M)$ we mean the formal adjoint of $\overline{\partial}_{p,q}:\Omega^{p,q}_c(M)\rightarrow \Omega^{p,q+1}_c(M)$ and $\partial_{p,q}:\Omega^{p,q}_c(M)\rightarrow \Omega^{p+1,q}_c(M)$ respectively. For each bidegree $(p,q)$ we have  the Hodge-Kodaira Laplacian  defined as $$\Delta_{\overline{\partial},p,q}:\Omega^{p,q}_c(M)\rightarrow \Omega^{p,q}_c(M),\ \Delta_{\overline{\partial},p,q}:=\overline{\partial}_{p,q-1}\circ\overline{\partial}^t_{p,q-1}+\overline{\partial}_{p,q}^t\circ \overline{\partial}_{p,q}.$$
In the case of functions, that is $(p,q)=(0,0)$, we will simply write $\Delta_{\overline{\partial}}:C^{\infty}_c(M)\rightarrow C^{\infty}_c(M)$.
We recall now the definition of the  following two self-adjoint extensions of $\Delta_{\overline{\partial},p,q}$:
\begin{equation}
\label{asdf}
\overline{\partial}_{p,q-1,\max}\circ \overline{\partial}_{p,q-1,\min}^t+\overline{\partial}_{p,q,\min}^t\circ \overline{\partial}_{p,q,\max}:L^2\Omega^{p,q}(M,g)\rightarrow L^2\Omega^{p,q}(M,g)
\end{equation} 
and 
\begin{equation}
\label{buio}
\overline{\partial}_{p,q-1,\min}\circ \overline{\partial}_{p,q-1,\max}^t+\overline{\partial}_{p,q,\max}^t\circ \overline{\partial}_{p,q,\min}:L^2\Omega^{p,q}(M,g)\rightarrow L^2\Omega^{p,q}(M,g)
\end{equation}
called respectively the absolute and the relative extension. The operator \eqref{asdf}, the absolute extension, is labeled in general with  $\Delta_{\overline{\partial},p,q,\abs}$ and its domain is given by $$\mathcal{D}(\Delta_{\overline{\partial},p,q,\abs})=\left\{\omega\in \mathcal{D}(\overline{\partial}_{p,q,\max})\cap \mathcal{D}(\overline{\partial}_{p,q-1,\min}^t):\overline{\partial}_{p,q,\max}\omega \in \mathcal{D}(\overline{\partial}^t_{p,q,\min})\ \text{and}\ \overline{\partial}_{p,q-1,\min}^t\omega \in \mathcal{D}(\overline{\partial}_{p,q-1,\max})\right\}.$$
The operator \eqref{buio}, the relative extension,  is labeled in general with  $\Delta_{\overline{\partial},p,q,\rel}$ and its domain is given by $$\mathcal{D}(\Delta_{\overline{\partial},p,q,\rel})=\left\{\omega\in \mathcal{D}(\overline{\partial}_{p,q,\min})\cap \mathcal{D}(\overline{\partial}_{p,q-1,\max}^t):\overline{\partial}_{p,q,\min}\omega \in \mathcal{D}(\overline{\partial}^t_{p,q,\max})\ \text{and}\ \overline{\partial}_{p,q-1,\max}^t\omega \in \mathcal{D}(\overline{\partial}_{p,q-1,\min})\right\}.$$
 Using  $\Delta_{\overline{\partial},p,q,\abs}$ and $\Delta_{\overline{\partial},p,q,\rel}$ we obtain the following orthogonal  decompositions for $L^2\Omega^{p,q}(M,g)$:
\begin{equation}
\label{zazzac}
L^2\Omega^{p,q}(M,g)=\mathcal{H}^{p,q}_{\overline{\partial},\abs}(M,g)\oplus \overline{\im(\Delta_{\overline{\partial},p,q,\abs})}=\mathcal{H}^{p,q}_{\overline{\partial}, \abs}(M,g)\oplus \overline{\im(\overline{\partial}_{p,q-1,\max})}\oplus\overline{\im(\overline{\partial}^t_{p,q,\min})}
\end{equation}
and 
\begin{equation}
\label{cocapro}
L^2\Omega^{p,q}(M,g)=\mathcal{H}^{p,q}_{\overline{\partial},\rel}(M,g)\oplus \overline{\im(\Delta_{\overline{\partial},p,q,\rel})}=\mathcal{H}^{p,q}_{\overline{\partial}, \rel}(M,g)\oplus \overline{\im(\overline{\partial}_{p,q-1,\min})}\oplus\overline{\im(\overline{\partial}^t_{p,q,\max})}
\end{equation}
where 
\begin{equation}
\label{bubu}
\mathcal{H}^{p,q}_{\overline{\partial},\abs}(M,g):=\ker(\overline{\partial}_{p,q,\max})\cap \ker(\overline{\partial}^t_{p,q-1,\min})=\ker(\Delta_{\overline{\partial},p,q,\abs})
\end{equation} 
and  
\begin{equation}
\label{babirussa}
\mathcal{H}^{p,q}_{\overline{\partial},\rel}(M,g):=\ker(\overline{\partial}_{p,q,\min})\cap \ker(\overline{\partial}^t_{p,q-1,\max})=\ker(\Delta_{\overline{\partial},p,q,\rel}).
\end{equation}
 Consider now the Hodge-Dolbeault operator $\overline{\partial}_{p}+\overline{\partial}^t_{p}:\Omega_c^{p,\bullet}(M)\rightarrow \Omega^{p,\bullet}_c(M)$ where with $\Omega^{p,\bullet}_c(M)$ we mean $\bigoplus_{q=0}^m\Omega_c^{p,q}(M)$. We can define two self-adjoint extensions  of $\overline{\partial}_p+\overline{\partial}^t_p$ taking  
\begin{equation}
\label{twoself}
\overline{\partial}_{p,\max}+\overline{\partial}^t_{p,\min}:L^2\Omega^{p,\bullet}(M,g)\rightarrow L^2\Omega^{p,\bullet}(M,g)
\end{equation}
\begin{equation} 
\label{twoselfs}
\overline{\partial}_{p,\min}+\overline{\partial}^t_{p,\max}:L^2\Omega^{p,\bullet}(M,g)\rightarrow L^2\Omega^{p,\bullet}(M,g)
\end{equation}
where clearly $L^2\Omega^{p,\bullet}(M,g)=\bigoplus_{q=0}^mL^2\Omega^{p,q}(M,g)$. The domain of $\overline{\partial}_{p,\max}+\overline{\partial}^t_{p,\min}$ is given by $\mathcal{D}(\overline{\partial}_{p,\max})\cap\mathcal{D}(\overline{\partial}^t_{p,\min})$ where $\mathcal{D}(\overline{\partial}_{p,\max})=\bigoplus_{q=0}^m\mathcal{D}(\overline{\partial}_{p,q,\max})$ and $\mathcal{D}(\overline{\partial}^t_{p,\min})=\bigoplus_{q=0}^m\mathcal{D}(\overline{\partial}^t_{p,q,\min})$. Analogously the domain of $\overline{\partial}_{p,\min}+\overline{\partial}^t_{p,\max}$ is given by  $\mathcal{D}(\overline{\partial}_{p,\min})\cap\mathcal{D}(\overline{\partial}^t_{p,\max})$ where $\mathcal{D}(\overline{\partial}_{p,\min})=\bigoplus_{q=0}^m\mathcal{D}(\overline{\partial}_{p,q,\min})$ and $\mathcal{D}(\overline{\partial}^t_{p,\max})=\bigoplus_{q=0}^m\mathcal{D}(\overline{\partial}^t_{p,q,\max})$. In particular we have: 
\begin{align}
\label{coccodezio}
&\ker(\overline{\partial}_{p,\max/\min}+\overline{\partial}^t_{p,\min/\max})=\bigoplus_{q=0}^m\ker(\overline{\partial}_{p,q,\max/\min})\cap \ker (\overline{\partial}^t_{p,q-1,\min/\max})=\bigoplus_{q=0}^m\mathcal{H}^{p,q}_{\overline{\partial},\abs/\rel}(M,g)\\
&\nn\im(\overline{\partial}_{p,\max/\min}+\overline{\partial}^t_{p,\min/\max})=\bigoplus_{q=0}^m\left(\im(\overline{\partial}_{p,q-1,\max/\min})\oplus \im (\overline{\partial}^t_{p,q,\min/\max})\right).
\end{align}
Furthermore we recall that the maximal and the minimal $L^2$-$\overline{\partial}$-cohomology of $(M,g)$ are  defined respectively as 
\begin{equation}
\label{chimar}
H^{p,q}_{2,\overline{\partial}_{\max}}(M,g):=\frac{\ker(\overline{\partial}_{p,q,\max})}{\im(\overline{\partial}_{p,q-1,\max})}\ \text{and}\ H^{p,q}_{2,\overline{\partial}_{\min}}(M,g):=\frac{\ker(\overline{\partial}_{p,q,\min})}{\im(\overline{\partial}_{p,q-1,\min})}.
\end{equation}
In particular if $H^{p,q}_{2,\overline{\partial}_{\max}}(M,g)$ is finite dimensional then $\im(\overline{\partial}_{p,q-1,\max})$ is closed and analogously if $H^{p,q}_{2,\overline{\partial}_{\min}}(M,g)$ is finite dimensional then $\im(\overline{\partial}_{p,q-1,\min})$ is closed.
We have the following important properties:
\begin{prop}
\label{usipeti}
In the setting described above. The following properties are equivalent:
\begin{itemize}
\item $H^{p,q}_{2,\overline{\partial}_{\max}}(M,g)$ is finite dimensional for every $q=0,...,m$.
\item \eqref{twoself} is a Fredholm operator on its domain endowed with the graph norm.
\item $\Delta_{\overline{\partial},p,q,\abs}:L^2\Omega^{p,q}(M,g)\rightarrow L^2\Omega^{p,q}(M,g)$ is  a Fredholm operator on its domain endowed with the graph norm for each $q=0,...,m$.
\end{itemize}
Analogously the following properties are equivalent:
\begin{itemize}
\item $H^{p,q}_{2,\overline{\partial}_{\min}}(M,g)$ is finite dimensional for every $q=0,...,m$.
\item \eqref{twoselfs} is a Fredholm operator on its domain endowed with the graph norm.
\item $\Delta_{\overline{\partial},p,q,\rel}:L^2\Omega^{p,q}(M,g)\rightarrow L^2\Omega^{p,q}(M,g)$ is  a Fredholm operator on its domain endowed with the graph norm for each $q=0,...,m$.
\end{itemize}
\end{prop}
Clearly, if we replace the operator $\overline{\partial}_{p,q}$ with $\partial_{p,q}$, then we get the analogous definitions and properties for the operators $\partial_{p,q,\max/\min}$, $\Delta_{\partial,p,q}$, $\Delta_{\partial,p,q,\abs/\rel}$, $\partial_{q}+\partial^t_{q}$ and $\partial_{q,\max/\min}+\partial^t_{q,\min/\max}$. In particular  the corresponding version of Prop. \ref{usipeti} holds for $H^{p,q}_{2,\partial_{\max/\min}}(M,g)$, $\partial_{p,\max/\min}+\partial^t_{\min/\max}$ and $\Delta_{\partial,p,q,\abs/\rel}$. Moreover, according to \cite{Huy} pag. 116, we have 
\begin{equation}
\label{cicci}
\partial^t_{p,q}=-*\overline{\partial}_{m-q,m-p-1}*\ \text{and}\ \overline{\partial}^t_{p,q}=-*\partial_{m-q-1,m-p}*
\end{equation}
and from \eqref{cicci} we easily get that 
\begin{equation}
\label{cicuta}
\partial^t_{p,q,\max/\min}=-*\overline{\partial}_{m-q,m-p-1,\max/\min}*\ \text{and}\ \overline{\partial}^t_{p,q,\max/\min}=-*\partial_{m-q-1,m-p,\max/\min}*
\end{equation}
where $*:L^2\Omega^{p,q}(M,g)\rightarrow L^2\Omega^{m-q,m-p}(M,g)$ is the unitary operator induced by the  Hodge star operator. Now let us label by 
\begin{equation}
\label{conj}
c: T^{1,0}M\rightarrow T^{0,1}M
\end{equation}
the $\mathbb{C}$-antilinear map given by  complex conjugation. In particular, given $p\in M$ and $v\in T_p(M)$, so that  $v-iJv\in T^{1,0}_pM$, we have $c(v-iJv)=v+iJv$.  Let us label by 
\begin{equation}
\label{conju}
c_{p,q}:\Lambda^{p,q}(M)\rightarrow \Lambda^{q,p}(M)
\end{equation}
 the natural map induced by \eqref{conj}. With a little abuse of notation we still label by $c_{p,q}$ the induced map on $(p,q)$-forms, that is
\begin{equation}
\label{conjform}
 c_{p,q}:\Omega^{p,q}(M)\rightarrow \Omega^{q,p}(M).
\end{equation} 
Clearly both \eqref{conju} and \eqref{conjform} are  $\mathbb{C}$-antilinear isomorphisms such that $(c_{p,q})^{-1}=c_{q,p}$. Moreover \eqref{conjform} induces an isomorphism 
\begin{equation}
\label{conjformcs}
 c_{p,q}|_{\Omega^{p,q}_c(M)}:\Omega_c^{p,q}(M)\rightarrow \Omega^{q,p}_c(M).
\end{equation} 
 We have the following well known properties:
\begin{prop}
\label{comconj}
In the setting described above. On $\Omega^{p,q}(M)$ and $\Omega^{p,q}_c(M)$ the following properties hold true: 
\begin{enumerate}
\item $\overline{\partial}_{p,q}=c_{q+1,p}\circ \partial_{q,p}\circ c_{p,q}.$
\item $\overline{\partial}_{p,q}^t=c_{q,p}\circ \partial_{q,p}^t\circ c_{p,q+1}$
\end{enumerate}
\end{prop}
\begin{proof}
For the first point see for instance \cite{Wells} Prop. 3.6. The second point follows using the first point, \eqref{cicci} and the fact that the Hodge star operator commutes with the complex conjugation.
\end{proof}

Now consider again $M$ endowed with a Hermitian metric $g$.  For each $\omega\in \Omega_c^{p,q}(M)$ it is easy to check that 
$g(\omega,\omega)=g(c_{p,q}\omega, c_{p,q}\omega)$. Using this  equality and the other properties recalled above, we easily get that $c_{p,q}$ induces a $\mathbb{C}$-antilinear operator 
\begin{equation}
\label{conjhil}
c_{p,q}:L^2\Omega^{p,q}(M,g)\rightarrow L^2\Omega^{q,p}(M,g)
\end{equation} 
which is bijective, continuous, with continuous inverse given by $c_{q,p}$ and such that $\|\eta\|_{L^2\Omega^{p,q}(M,g)}=\|c_{p,q}\eta\|_{L^2\Omega^{q,p}(M,g)}$ for each $\eta\in L^2\Omega^{p,q}(M,g)$. 
Finally we close this introduction with the following proposition.
\begin{prop}
\label{occhiodibue}
In the setting described above. The following properties hold true:
\begin{enumerate}
\item $c_{p,q}\left(\mathcal{D}(\overline{\partial}_{p,q,\max})\right)=\mathcal{D}(\partial_{q,p\max})$ and\  $\overline{\partial}_{p,q,\max}=c_{q+1,p}\circ \partial_{q,p,\max}\circ c_{p,q}.$
\item  $c_{p,q}\left(\mathcal{D}(\overline{\partial}_{p,q,\min})\right)=\mathcal{D}(\partial_{q,p\min})$ and\  $\overline{\partial}_{p,q,\min}=c_{q+1,p}\circ \partial_{q,p,\min}\circ c_{p,q}.$
\item $c_{p,q+1}\left(\mathcal{D}(\overline{\partial}^t_{p,q,\max})\right)=\mathcal{D}(\partial_{q,p\max}^t)$ and\  $\overline{\partial}_{p,q,\max}^t=c_{q,p}\circ \partial_{q,p,\max}^t\circ c_{p,q+1}.$
\item $c_{p,q+1}\left(\mathcal{D}(\overline{\partial}^t_{p,q,\min})\right)=\mathcal{D}(\partial_{q,p\min}^t)$ and\  $\overline{\partial}_{p,q,\min}^t=c_{q,p}\circ \partial_{q,p,\min}^t\circ c_{p,q+1}.$
\item $*\left(\mathcal{D}(\Delta_{\overline{\partial},p,q,\abs})\right)=\mathcal{D}(\Delta_{\partial,m-q,m-p,\rel})$ and 
$*\circ  \Delta_{\overline{\partial},p,q,\abs}=\Delta_{\partial,m-q,m-p,\rel}\circ *.$
\item $*\left(c_{p,q}(\mathcal{D}(\Delta_{\overline{\partial},p,q,\abs}))\right)=\mathcal{D}(\Delta_{\overline{\partial},m-p,m-q,\rel})$ and 
$*\circ c_{p,q}\circ \Delta_{\overline{\partial},p,q,\abs}=\Delta_{\overline{\partial},m-p,m-q,\rel}\circ *\circ c_{p,q}.$
\end{enumerate}
\end{prop}
\begin{proof}
This follows immediately by \eqref{cicuta}, Prop. \ref{comconj} and the properties of \eqref{conjhil}.
\end{proof}

\section{Some abstract results}
This section contains some abstract results that will be  used later on in the paper. We start by recalling some well known facts about the Green operator.\\ Let $H_1$ and $H_2$ be  separable Hilbert spaces whose Hilbert products are labeled by $\langle\ ,\ \rangle_{H_1}$ and $\langle\ ,\ \rangle_{H_2}$. Let $T:H_1\rightarrow H_2$ be an unbounded, densely defined and closed operator  with domain $\mathcal{D}(T)$. Assume that $\im(T)$ is closed. Let  $T^*:H_2\rightarrow H_1$ be the adjoint of $T$. Then $\im(T^*)$ is closed as well and we have the following orthogonal decompositions: $H_1=\ker(T)\oplus \im(T^*)$ and $H_2=\ker(T^*)\oplus \im(T)$. The  Green operator of $T$, $G_T:H_2\rightarrow H_1$, is then the operator defined by the following assignments: if $u\in \ker(T^*)$ then $G_T(u)=0$,  if $u\in \im(T)$ then  $G_T(u)=v$ where $v$ is the unique element in $\mathcal{D}(T)\cap \im(T^*)$ such that $T(v)=u$. We have that $G_T:H_2\rightarrow H_1$ is  a bounded operator. Moreover, if $H_1=H_2$ and $T$ is self-adjoint then $G_T$ is self-adjoint too. If $H_1=H_2$ and $T$  is self-adjoint and non-negative, that is $\langle Tu,u\rangle_{H_1}\geq 0$ for each $u\in \mathcal{D}(T)$, then $G_T$ is self-adjoint and non negative as well. Furthermore we have $T\circ G_T=\id_2-P_{\ker(T^*)}$ and $G_T\circ T=\id_1-P_{\ker(T)}$ on $\mathcal{D}(T)$ where $\id_1:H_1\rightarrow H_1$, $\id_2:H_2\rightarrow H_2$  are the corresponding  identity maps and $P_{\ker(T)}:H_1\rightarrow \ker(T)$, $P_{\ker(T^*)}:H_2\rightarrow \ker(T^*)$ are the orthogonal projections on $\ker(T)$ and  $\ker(T^*)$ respectively.
Finally we recall that $G:H_2\rightarrow H_1$ is a compact operator if and only if the following inclusion $\mathcal{D}(T)\cap \im(T^*)\hookrightarrow H_1$, where  $\mathcal{D}(T)\cap \im(T^*)$ is endowed with the graph norm of $T$, is a compact operator.

\begin{prop}
\label{acacio}
Let $T:H_1\rightarrow H_2$ be an unbounded, densely defined and closed operator acting between two separable Hilbert spaces. Let $\mathcal{D}(T)$ be the domain of $T$ and let $T^*:H_2\rightarrow H_1$ be the adjoint of $T$. Assume that $\im(T)$ is closed. Consider the operator $T^*\circ T:H_1\rightarrow H_1$ with domain $\mathcal{D}(T^*\circ T)=\{u\in \mathcal{D}(T): Tu\in \mathcal{D}(T^*)\}$. Then we have the following properties:
\begin{enumerate}
\item $\im(T^*\circ T)=\im(T^*)$ and therefore it is closed in $H_1$.
\item $G_{T^*\circ T}=G_{T}\circ G_{T^*}$.
\end{enumerate}
\end{prop}

\begin{proof}
Clearly  $\im(T^*\circ T)\subset \im(T^*)$. Using the orthogonal decomposition $H_2=\ker(T^*)\oplus \im(T)$ we get that $\im(T^*)=\{T^*s\ \text{such that}\ s\in \im(T)\cap \mathcal{D}(T^*)\}$. Hence we can conclude  that $\im(T^*)=\im(T^*\circ T)$. In particular, by the fact that  $\im(T^*)$ is closed, we have  that $\im(T^*\circ T)$ is closed in $H_1$. Consider now the second point. Clearly if $v\in \ker(T^*\circ T)$ then $G_T(G_{T^*}(v))=0$. Let now $v\in \im(T^*\circ T)$ and let $u\in \mathcal{D}(T^*\circ T)\cap \im(T^*\circ T)$ be the unique element in $\mathcal{D}(T^*\circ T)\cap \im(T^*\circ T)$ such that $T^*(T(u))=v$. We have $G_{T^*}(v)=w$ where $w$ is the unique element in $\mathcal{D}(T^*)\cap \im(T)$ such that $T^*(w)=v$. Since $T^*$ is injective on $\mathcal{D}(T^*)\cap \im(T)$ we have $T(u)=w$ because $T^*(T(u))=v=T^*(w)$. Therefore we have $G_T(w)=G_T(T(u))=u$ because $u\in \mathcal{D}(T^*\circ T)\cap \im(T^*\circ T)\subset  \mathcal{D}( T)\cap \im(T^*)$. Summarizing  we have shown that if $v\in \ker(T^*\circ T)$ then $G_T(G_{T^*}(v))=0$ while if $v\in \im(T^*\circ T)$  then $G_T(G_{T^*}(v))=u$ where $u$ is the unique element in $\mathcal{D}(T^*\circ T)\cap \im(T^*\circ T)$ such that $T^*(T(u))=v$. We can thus conclude that $G_{T^*\circ T}=G_{T}\circ G_{T^*}$.
\end{proof}

\begin{prop}
\label{corniola}
In the setting of Prop. \ref{acacio}. Assume moreover that the following inclusion 
\begin{equation}
\label{miele}
\mathcal{D}(T^*\circ T)\cap \im(T^*\circ T)\hookrightarrow H_1
\end{equation}
where $\mathcal{D}(T^*\circ T)\cap \im(T^*\circ T)$ is endowed with the graph norm of $T^*\circ T$, is a compact operator. Then also the following  inclusion 
\begin{equation}
\label{melata}
\mathcal{D}(T\circ T^*)\cap \im(T\circ T^*)\hookrightarrow H_2
\end{equation}
 where $\mathcal{D}(T\circ T^*)\cap \im(T\circ T^*)$ is endowed with the graph norm of $T\circ T^*$, is a compact operator. 
\end{prop}

\begin{proof}
Let us define $D:=T\circ T^*$. Then $D:H_2\rightarrow H_2$ is an unbounded, densely defined, self-adjoint and non-negative operator such that $\im(D)$ is closed. Consider now $D^2:H_2\rightarrow H_2$ where $\mathcal{D}(D^2)=\{u\in \mathcal{D}(D): Du\in \mathcal{D}(D)\}$. Then, according to Prop. \ref{acacio}, $\im(D^2)$ is closed. Let $G_{D^2}:H_2\rightarrow H_2$ be the Green operator of $D^2$. Again by Prop. \ref{acacio} we know that $G_{D^2}=G_D\circ G_D= G_{T\circ T^*}\circ G_{T\circ T^*}$. Moreover, again by Prop. \ref{acacio}, we know that $G_{T\circ T^*}=G_{T^*}\circ G_T$. Therefore we have $G_{D^2}=G_{T^*}\circ G_{T}\circ G_{T^*}\circ G_{T}$. This tells us that $G_{D^2}:H_2\rightarrow H_2$ is a compact operator because, by \eqref{miele}, we know that   $G_T\circ G_{T^*}:H_1\rightarrow H_1$ is a compact operator.  On the other hand, since $D$ is self-adjoint, we have $G_{D^2}=G_{D}^2$. Therefore, by the fact that $G_{D}:H_2\rightarrow H_2$ is bounded, self-adjoint and non-negative, and by the fact that $G_D^2:H_2\rightarrow H_2$ is compact we can  conclude that $G_D:H_2\rightarrow H_2$ is a compact operator. As $D=T\circ T^*$ we can eventually conclude that   the  inclusion $\mathcal{D}(T\circ T^*)\cap \im(T\circ T^*)\hookrightarrow H_2$, where $\mathcal{D}(T\circ T^*)\cap \im(T\circ T^*)$ is endowed with the corresponding graph norm, is a compact operator. 
\end{proof}

We have now the following proposition. 
\begin{prop}
\label{coccadezio}
Let $H_1$, $H_2$ and $H_3$ be   separable Hilbert spaces and let $T_1:H_1\rightarrow H_{2}$ and $T_2:H_2\rightarrow H_{3}$ be   densely defined and closed operators such that, for each $n=1,2$, $\im(T_n)$ is closed in $H_{n+1}$ and  $\im(T_1)\subset \ker(T_{2})$. Consider the operator $\Delta_T:H_{2}\rightarrow H_2$, $\Delta_T:= T_1\circ T^*_1+T^*_2\circ T_2$, with domain given by $\left\{s\in \mathcal{D}(T_2)\cap\mathcal{D}(T^*_1): T_2s\in \mathcal{D}(T^*_2)\ \text{and}\ T^*_1s\in \mathcal{D}(T_1)\right\}$. Then the following inclusions hold true
\begin{equation}
\label{inclusion}
\left(\mathcal{D}(T_1\circ T^*_1)\cap \im(T_1\circ T^*_1)\right)\subset \mathcal{D}(\Delta_T),\ 
\left(\mathcal{D}(T^*_2\circ T_2)\cap \im(T_2^*\circ T_2)\right) \subset \mathcal{D}(\Delta_T)
\end{equation}
and we have the following orthogonal decomposition for $\mathcal{D}(\Delta_{T})$
\begin{equation}
\label{pezzi}
\mathcal{D}(\Delta_{T})=\left(\ker(T^*_1)\cap \ker(T_2)\right)\oplus \left(\mathcal{D}(T_1\circ T^*_1)\cap \im(T_1\circ T^*_1)\right)\oplus \left(\mathcal{D}(T^*_2\circ T_2)\cap \im(T_2^*\circ T_2))\right.
\end{equation}
where the addends on the right end side of \eqref{pezzi} are closed subspaces of $\mathcal{D}(\Delta_T)$ and are orthogonal to each other with respect to the graph product of $\mathcal{D}(\Delta_T)$.\\ Moreover on $\left(\mathcal{D}(T_1\circ T^*_1)\cap \im(T_1\circ T^*_1)\right)$ the graph product of $\Delta_T$ coincides with the graph product of $T_1\circ T^*_1$. Analogously on $\left(\mathcal{D}(T_2^*\circ T_2)\cap \im(T^*_2\circ T_2)\right)$ the graph product of $\Delta_T$ coincides with the graph product of $T^*_2\circ T_2$.
\end{prop}

\begin{proof}
The inclusions in \eqref{inclusion} follow immediately by the definition of $\mathcal{D}(\Delta_T)$ and by the fact that $T_2\circ T_1=0$ and $T^*_1\circ T^*_2=0$.  Moreover it is an easy check to verify that the spaces on the right hand side of \eqref{pezzi} are orthogonal to each other with respect to the graph product of $\Delta_T$. Therefore the right hand side of \eqref{pezzi} is contained in $\mathcal{D}(\Delta_T)$. In order to complete the proof we have to prove now the opposite inclusion. To this aim consider the orthogonal decomposition of $H_2$ given by $$H_2=\left(\ker(T^*_1)\cap \ker(T_2)\right)\oplus  \im(T_1)\oplus \im(T_2^*).$$
By the fact that $\im(T_n)$ is closed in $H_{n+1}$, $n=1,2$, we have that $\im(T^*_2\circ T_2)=\im(T^*_2)$ and that $\im(T_1)=\im(T_1\circ T^*_1)$, see Prop. \ref{acacio}. Therefore we can replace the above decomposition with 
\begin{equation}
\label{pezzi2}
H_2=\left(\ker(T^*_1)\cap \ker(T_2)\right)\oplus  \im(T_1\circ T^*_1)\oplus \im(T_2^*\circ T_2).
\end{equation}
Let now $s\in \mathcal{D}(\Delta_T)$. Then, according to \eqref{pezzi2}, we have $s=s_1+s_2+s_3$ with respectively $s_1\in \ker(T^*_1)\cap \ker(T_2)$, $s_2\in \im(T_1\circ T^*_1)$ and $s_3\in \im(T_2^*\circ T_2)$. By the fact that $s_1\in \ker(\Delta_T)$ we have $s_2+s_3\in \mathcal{D}(\Delta_T)$ that is $s_2+s_3\in \mathcal{D}(T^*_2\circ T_2)\cap\mathcal{D}(T_1\circ T^*_1)$. On the other hand $s_2\in \mathcal{D}(T^*_2\circ T_2)$ because $\im(T_1\circ T^*_1)\subset \ker(T^*_2\circ T_2)$ and $s_3\in \mathcal{D}(T_1\circ T^*_1)$ because $\im(T^*_2\circ T_2)\subset \ker(T_1\circ T^*_1)$. Hence this leads us to the conclusion that $s_2\in \im(T_1\circ T^*_1)\cap\mathcal{D}(T^*_2\circ T_2)\cap\mathcal{D}(T_1\circ T^*_1)$ and that 
$s_3\in \im(T_2^*\circ T_2)\cap\mathcal{D}(T^*_2\circ T_2)\cap\mathcal{D}(T_1\circ T^*_1)$,  that is, $$s_2\in \im(T_1\circ T^*_1)\cap\mathcal{D}(T_1\circ T^*_1)\ \text{and}\ s_3\in \im(T_2^*\circ T_2)\cap\mathcal{D}(T^*_2\circ T_2).$$ This completes the proof of \eqref{pezzi}. Now, by the fact that   \eqref{pezzi} is an orthogonal decomposition with respect to the graph product of $\Delta_T$, we can  conclude  that every addend on the right hand side of \eqref{pezzi} is a closed subspace of $\mathcal{D}(\Delta_T)$ with respect to its graph norm. Finally it is again an immediate check to verify that on $\left(\mathcal{D}(T_1\circ T^*_1)\cap \im(T_1\circ T^*_1)\right)$ the graph product of $\Delta_T$ coincides with the graph product of $T_1\circ T^*_1$ and that  analogously on $\left(\mathcal{D}(T_2^*\circ T_2)\cap \im(T^*_2\circ T_2)\right)$ the graph product of $\Delta_T$ coincides with the graph product of $T^*_2\circ T_2$.
\end{proof}

\begin{cor}
\label{calanchi}
In the setting of Prop. \ref{coccadezio}. The operator $\Delta_T:H_2\rightarrow H_2$ has closed range. Moreover the following properties are equivalent:
\begin{enumerate}
\item $G_{\Delta_T}:H_2\rightarrow H_2$ is a compact operator.
\item The inclusion $(\mathcal{D}(\Delta_T)\cap \im(\Delta_T))\hookrightarrow H_2$ is a compact operator where $\mathcal{D}(\Delta_T)\cap \im(\Delta_T)$ is endowed with the corresponding graph norm.
\item The inclusions $(\mathcal{D}(T_1^*\circ T_1)\cap \im(T^*_1\circ T_1))\hookrightarrow H_1$   and   $(\mathcal{D}(T_2\circ T_2^*)\cap \im(T_2\circ T_2^*))\hookrightarrow H_3$ are both compact operators where $\mathcal{D}(T_1^*\circ T_1)\cap \im(T^*_1\circ T_1)$ and $\mathcal{D}(T_2\circ T_2^*)\cap \im(T_2\circ T_2^*)$ are endowed with the corresponding graph norms.
\item $G_{T_1^*\circ T_1}:H_1\rightarrow H_1$ and $G_{T_2\circ T_2^*}:H_3\rightarrow H_3$  are both compact operators.
\end{enumerate}
\end{cor}

\begin{proof}
Consider the  operator $\Delta_T:H_2\rightarrow H_2$. We have the following chain of inclusions that follows immediately by \eqref{pezzi} and by the fact that $\im(T_n)$ is closed for $n=1,2$:
$$\im(\Delta_T)\subset \overline{\im(\Delta_T)}=\overline{\im(T_1)}\oplus \overline{\im(T_2^*)}=\im(T_1)\oplus \im(T_2^*)=\im(T_1\circ T^*_1)\oplus \im(T_2^*\circ T_2)=\im(\Delta_T).$$ Therefore $\im(\Delta_T)$ is closed.
Concerning the second part of the corollary  the equivalence between  the first two statements  follows by the elementary properties of the Green operator that we have recalled previously. For the same reasons it is clear the equivalence between the third and the forth statement. Finally  the equivalence between the second and the third statement follows immediately by Prop. \ref{corniola} and Prop. \ref{coccadezio}.
\end{proof}

We recall also the following property.

\begin{prop}
\label{same}
Let $T:H\rightarrow K$ be an unbounded, closed and densely defined operator between two separable Hilbert spaces. Assume that both $T^*\circ T:H\rightarrow H$ and $T\circ T^*:K\rightarrow K$ have discrete spectrum. Given $\lambda>0$ let us define $H_{\lambda}:=\{s\in \mathcal{D}(T^*\circ T): T^*(Ts)=\lambda s\}$ and analogously $K_{\lambda}:=\{u\in \mathcal{D}(T\circ T^*): T(T^*u)=\lambda u\}$. 
 Then, for every positive $\lambda$, we have $T(H_{\lambda})=K_{\lambda}$ and 
$$T|_{H_{\lambda}}:H_{\lambda}\rightarrow K_{\lambda}$$ is an isomorphism. 
\end{prop}
\begin{proof}
Let $s\in H_{\lambda}$. Since $H_{\lambda}\subset \mathcal{D}(T^*\circ T)$ we have $T(s)\in \mathcal{D}(T^*)$. Moreover $T^*(Ts)=\lambda s$ and therefore $Ts\in \mathcal{D}(T\circ T^*)$. Finally $T(T^*(Ts))=\lambda Ts$ and hence we can conclude that $Ts$ lies in $K_{\lambda}$. So we proved that $T(H_{\lambda})\subset K_{\lambda}$. Moreover, by the fact that $\lambda>0$, we get immediately that $T|_{H_{\lambda}}:H_{\lambda}\rightarrow K_{\lambda}$ is injective. Arguing in the same way with $T^*$ and $K_{\lambda}$ we have that $T^*(K_{\lambda})\subset H_{\lambda}$ and  that $T^*|_{K_{\lambda}}:K_{\lambda}\rightarrow H_{\lambda}$ is injective. Finally, by the fact that $T^*\circ T:H\rightarrow H$ and $T\circ T^*:K\rightarrow K$ have discrete spectrum, we know that $H_{\lambda}$ and  $K_{\lambda}$ are finite dimensional vector spaces. Using the fact that $T|_{H_{\lambda}}:H_{\lambda}\rightarrow K_{\lambda}$ is injective and that $T^*|_{K_{\lambda}}:K_{\lambda}\rightarrow H_{\lambda}$ is injective we  therefore get that  $\dim(H_{\lambda})=\dim(K_{\lambda})$.
Ultimately  this allows us to conclude that $T|_{H_{\lambda}}:H_{\lambda}\rightarrow K_{\lambda}$ is an isomorphism as desired.
\end{proof}

Finally we conclude this section with the following proposition. For the definition and the main properties of the Friedrich extension  we refer to \cite{FraBei} and to the bibliography cited there.  
\begin{prop}
\label{fall}
Let $E,F$ be two vector bundles over an open and  possibly incomplete  Riemannian manifold $(M,g).$ Let $\rho$ and $\tau$ be two metrics on $E$ and $F$ respectively. Let $D:C^{\infty}_{c}(M,E)\rightarrow C^{\infty}_{c}(M,F)$ be an unbounded and densely defined differential operator. Let $D^t:C^{\infty}_{c}(M,F)\rightarrow C^{\infty}_{c}(M,E)$ be its formal adjoint. Then for  $D^t\circ D:L^{2}(M,E,g)\rightarrow L^{2}(M,E,g)$ we have: 
 $$(D^t\circ D)^{\mathcal{F}}=D^t_{\max}\circ D_{\min}$$ where $(D^t\circ D)^{\mathcal{F}}$ is the Friedrich extension of $D^t\circ D$.
\end{prop}

\begin{proof}
This follows immediately by the definition of Friedrich extension. See  for instance \cite{BLE}, pag. 447.
\end{proof}

\section{Parabolic open subsets and Hodge-Dolbeault operator}
We start with the following definition.
\begin{defi}
\label{para}
Let $(M,g)$ be an open and possibly incomplete Riemannian manifold. Then $(M,g)$ is said parabolic if there exists a sequence of Lipschitz functions with compact support $\{\phi_n\}_{n\in \mathbb{N}}$ such that 
\begin{enumerate}
\item $0\leq \phi_n\leq 1$ for all $n$
\item $\lim\phi_n=1$ pointwise a.e. as $n\rightarrow \infty$
\item $\lim\|d_{\min} \phi_n\|_{L^2\Omega^1(M,g)}=0$ as $n\rightarrow \infty$.
\end{enumerate}
\end{defi}
We point out that as $\{\phi_n\}$ is a sequence of Lipschitz functions with compact support then $\phi_n\in \mathcal{D}(d_{\min})$ for each $n$ so that the third point in the previous definition makes sense, see \cite{FraBei}. Moreover the fact that $\phi_n\in \mathcal{D}(d_{\min})$  implies  that  $\phi_n\in \mathcal{D}(\overline{\partial}_{\min})$ and by the third point of Def. \ref{para} we can easily deduce  that  $\|\overline{\partial}_{\min}\phi_n\|_{L^2\Omega^{0,1}(M,g)}=0$ as $n\rightarrow \infty$.
Consider now  a compact complex Hermitian manifold $(M,g)$  of complex dimension $m$  and let $E$ be a holomorphic vector bundle over $M$. With $\Omega^{p,q}(M,E)$, $\Omega^{p,q}_c(M,E)$, we mean respectively the space of smooth sections, smooth sections with compact support, of $\Lambda^{p,q}(M)\otimes E$. With $\overline{\partial}_{E,p,q}:\Omega^{p,q}(M,E)\rightarrow \Omega^{p,q+1}(M,E)$ we label the corresponding Dolbeault operator. If $E$ is endowed with a Hermitian metric $\rho$ then with $g_{\rho}$ we label the natural Hermitian metric   induced on $\Lambda^{p,q}(M)\otimes E$. With $L^2\Omega^{p,q}(M,E,g)$ we mean the Hilbert space of $L^2$-$(p,q)$-forms with values in $E$. Analogously to \eqref{twoself} we have the following notations
$\Omega^{p,\bullet}_c(M,E):=\bigoplus_{q=0}^m\Omega_c^{p,q}(M,E)$ and $L^2\Omega^{p,\bullet}(M,E,g)=\bigoplus_{q=0}^mL^2\Omega^{p,q}(M,E,g)$. With $\overline{\partial}_{E,p}+\overline{\partial}_{E,p}^t:\Omega_c^{p,\bullet}(M,E)\rightarrow \Omega_c^{p,\bullet}(M,E)$ we mean the Hodge-Dolbeault operator acting on $\Omega_c^{p,\bullet}(M,E)$. We are now in the position to state the next proposition.

\begin{prop}
\label{gaur}
Let $(M,g)$ be a compact complex Hermitian manifold of complex dimension $m$. Let $(E,\rho)$ be a Hermitian holomorphic vector bundle on $M$. Let $A\subset M$ be an open and dense subset of $M$ such that $(A,g|_A)$ is parabolic. Then the Hodge-Dolbeault operator 
\begin{equation}
\label{pluf}
\overline{\partial}_{E,p}+\overline{\partial}_{E,p}^t:L^2\Omega^{p,\bullet}(A,E|_A,g|_A)\rightarrow L^2\Omega^{p,\bullet}(A,E|_A,g|_A)
\end{equation}
with  domain given by $\Omega^{p,\bullet}_c(A,E|_A)$ is essentially self-adjoint for each $p=0,...,m$. Moreover the unique closed extension of \eqref{pluf} coincides with the operator
\begin{equation}
\label{pluff}
\overline{\partial}_{E,p}+\overline{\partial}_{E,p}^t:L^2\Omega^{p,\bullet}(M,E,g)\rightarrow L^2\Omega^{p,\bullet}(M,E,g)
\end{equation}
where \eqref{pluff} is the unique closed extension of $\overline{\partial}_{E,p}+\overline{\partial}_{E,p}^t:\Omega^{p,\bullet}_c(M,E,g)\rightarrow \Omega^{p,\bullet}_c(M,E,g)$ viewed as an unbounded and densly defined operator acting on $L^2\Omega^{p,\bullet}(M,E,g)$.

\end{prop}

\begin{proof}
 First we observe that since $M\setminus A$ has measure zero in $M$ we have an equality of Hilbert spaces $L^2\Omega^{p,\bullet}(M,E,g)=L^2\Omega^{p,\bullet}(A,E|_A,g|_A)$. Let us label by $\mathcal{D}(\overline{\partial}_{E,p} + \overline{\partial}^t_{E,p})$, $\mathcal{D}((\overline{\partial}_{E,p} + \overline{\partial}^t_{E,p})_{\min})$ and $\mathcal{D}((\overline{\partial}_{E,p} + \overline{\partial}^t_{E,p})_{\max})$ respectively  the domain of \eqref{pluff}, the minimal domain of \eqref{pluf} and the maximal domain of \eqref{pluf}. 
As a first step we want to show that $\mathcal{D}(\overline{\partial}_{E,p} + \overline{\partial}^t_{E,p})=\mathcal{D}((\overline{\partial}_{E,p} + \overline{\partial}^t_{E,p})_{\min})$.
Since the inclusion $\mathcal{D}(\overline{\partial}_p + \overline{\partial}^t_p)\supset \mathcal{D}((\overline{\partial}_p + \overline{\partial}^t_p)_{\min})$
is clear we are left to prove the other inclusion $\mathcal{D}(\overline{\partial}_{E,p} + \overline{\partial}^t_{E,p})\subset
\mathcal{D}((\overline{\partial}_{E,p} + \overline{\partial}^t_{E,p})_{\min})$. According to Prop. \eqref{partiro} we know that $\Omega^{p,\bullet}(M,E)=\Omega^{p,\bullet}(M,E)\cap \mathcal{D}(\overline{\partial}_{E,p} + \overline{\partial}^t_{E,p})$ is dense in $\mathcal{D}(\overline{\partial}_{E,p} + \overline{\partial}^t_{E,p})$ with respect to the corresponding graph norm. Therefore it is enough to prove that $\Omega^{p,\bullet}(M,E)\subset \mathcal{D}((\overline{\partial}_{E,p} + \overline{\partial}^t_{E,p})_{\min})$.
As we assumed that $(A,g|_A)$ is parabolic there exists a sequence $\{\phi_i\}_{i\in \mathbb{N}}$ that satisfies the properties of 
Def. \ref{para}. Let $\omega\in \Omega^{p,\bullet}(M,E)$. By the fact that $\phi_i$ is Lipschitz we easily get  that  $\phi_i\omega\in \mathcal{D}((\overline{\partial}_{E,p} + \overline{\partial}^t_{E,p})_{\max})$ and that 
\begin{equation}
\label{lip-max}
(\overline{\partial}_{E,p} + \overline{\partial}^t_{E,p})_{\max}(\phi_i\omega)=\phi_i(\overline{\partial}_{E,p} + \overline{\partial}^t_{E,p})\omega+(\overline{\partial}_{\max}\phi_i)\wedge \omega-{\rm Int}(\overline{\partial}_{\max}\phi_i)\omega
\end{equation}
where the operator ${\rm Int}(\overline{\partial}_{\max}\phi_i):L^2\Omega^{p,\bullet}(A,E|_A,g|A)\rightarrow L^2\Omega^{p,\bullet}(A,E|_A,g|_A)$ is the adjoint of the bounded operator $(\overline{\partial}_{\max}\phi_i)\wedge  :L^2\Omega^{p,\bullet}(A,E|_A,g|_A)\rightarrow L^2\Omega^{p,\bullet}(A,E|_A,g|_A)$ given by exterior multiplication with $(\overline{\partial}_{\max}\phi_i)$.\\ That $(\overline{\partial}_{\max}\phi_i)\wedge  :L^2\Omega^{p,\bullet}(A,E|_A,g|_A)\rightarrow L^2\Omega^{p,\bullet}(A,E|_A,g|_A)$ is a bounded operator follows  by the fact that $\overline{\partial}_{\max}\phi_i\in L^{\infty}\Omega^{0,1}(A,g|_A)$ which is in turn a consequence of the fact that $\phi_i$ is Lipschitz.
Using now that $\phi_i$ has compact support in $A$ we obtain by Prop. \ref{carmina}  that 
$\phi_i\omega\in \mathcal{D}((\overline{\partial}_{E,p} + \overline{\partial}^t_{E,p})_{\min})$ and therefore, by \eqref{lip-max}, we have 
\begin{equation}
\label{lip}
(\overline{\partial}_{E,p} + \overline{\partial}^t_{E,p})_{\min}(\phi_i\omega)=\phi_i(\overline{\partial}_{E,p} + \overline{\partial}^t_{E,p})\omega+(\overline{\partial}_{\min}\phi_i)\wedge \omega-{\rm Int}(\overline{\partial}_{\min}\phi_i)\omega.
\end{equation}
Our aim now is  to show that when $i$ tends to $+\infty$ then $\phi_i\omega$ tends to $\omega$ in the graph norm of \eqref{pluff}. That $\phi_i \omega$ tends to $\omega$ in $L^2\Omega^{p,\bullet}(M,E,g)$ is a direct application of the Lebesgue dominated convergence theorem. Next we consider $(\overline{\partial}_{E,p} + \overline{\partial}^t_{E,p})_{\min}(\phi_i\omega)$; we want to show that this
 sequence converges to $(\overline{\partial}_{E,p} + \overline{\partial}^t_{E,p})(\omega)$ in $L^2\Omega^{p,\bullet}(M,E,g)$. We use \eqref{lip}: looking at the three summands
 on the right hand side we easily see, using again the Lebesgue dominate convergence theorem, that $\phi_i(\overline{\partial}_{E,p} + \overline{\partial}^t_{E,p})\omega$ converges to 
 $(\overline{\partial}_{E,p} + \overline{\partial}^t_{E,p})\omega$
 whereas for the last two terms we have the following inequality : 
 \begin{equation}
 \label{inequality}
 \| (\overline{\partial}_{\min}\phi_i)\wedge \omega\|_{L^2\Omega^{p,\bullet}(A, E|_A,g|_A)}\,\leq 
 \|  \overline{\partial}_{\min}\phi_i\|_{L^2 \Omega^{0,1}(A,g|_A)}\,
 \| \omega\|_{L^\infty\Omega^{p,\bullet}(A,E|_A,g_{\rho}|_A)}
 \end{equation}
 which in turn implies 
 \begin{equation}
 \label{estimate}
 \| {\rm Int}(\overline{\partial}_{\min}\phi_i)\omega \|_{L^2\Omega^{p,\bullet}(A,E|_A,g|_A)}\,\leq 
 \|  \overline{\partial}_{\min}\phi_i\|_{L^2 \Omega^{0,1}(A,g|_A)}\,
 \| \omega\|_{L^\infty \Omega^{p,\bullet}(A,E|_A,g_{\rho}|_A)}.
 \end{equation}
 Using the last property in Def. \ref{para} we get immediately that the left hand sides in \eqref{inequality} and \eqref{estimate} tend to $0$ when $i$ tends to $+\infty$. Summarizing we have shown that $\mathcal{D}(\overline{\partial}_{E,p} + \overline{\partial}^t_{E,p})\subset
\mathcal{D}((\overline{\partial}_{E,p} + \overline{\partial}^t_{E,p})_{\min})$ and thus $\mathcal{D}(\overline{\partial}_{E,p} + \overline{\partial}^t_{E,p})=
\mathcal{D}((\overline{\partial}_{E,p} + \overline{\partial}^t_{E,p})_{\min})$. Therefore the minimal extension of \eqref{pluf} coincides with \eqref{pluff}. Now, using the fact that $\overline{\partial}_{E,p} + \overline{\partial}^t_{E,p}:L^2\Omega^{p,\bullet}(M,E,g)\rightarrow L^2\Omega^{p,\bullet}(M,E,g)$
is self-adjoint, we   get that $((\overline{\partial}_{E,p} + \overline{\partial}^t_{E,p})_{\min})^*=(\overline{\partial}_{E,p} + \overline{\partial}^t_{E,p})_{\min}$. On the other hand, see \eqref{emendare}, we have $((\overline{\partial}_{E,p} + \overline{\partial}^t_{E,p})_{\min})^*=(\overline{\partial}_{E,p} + \overline{\partial}^t_{E,p})_{\max}$. Therefore we are lead to the conclusion that $(\overline{\partial}_{E,p} + \overline{\partial}^t_{E,p})_{\max}=(\overline{\partial}_{E,p} + \overline{\partial}^t_{E,p})_{\min}$  as desired.
\end{proof}

We have now the following application of Prop. \ref{gaur}. 

\begin{prop}
\label{lit}
Let $(M,g)$, $(E,\rho)$ and $A$ be as in Prop. \ref{gaur}. Then the following three operators coincide:
\begin{align}
\label{ttt}
& \overline{\partial}_{E,p,q,\max}:L^2\Omega^{p,q}(A,E|_A,g|_A)\rightarrow L^2\Omega^{p,q+1}(A,E|_A,g|_A),\\
\label{bilo}
& \overline{\partial}_{E,p,q,\min}:L^2\Omega^{p,q}(A,E|_A,g|_A)\rightarrow L^2\Omega^{p,q+1}(A,E|_A,g|_A),\\
\label{ses}
& \overline{\partial}_{E,p,q}:L^2\Omega^{p,q}(M,E,g)\rightarrow L^2\Omega^{p,q+1}(M,E,g),
\end{align}
where \eqref{ses} is the unique closed extension of  $\overline{\partial}_{E,p,q}:\Omega^{p,q}(M,E)\rightarrow \Omega^{p,q+1}(M,E)$ viewed as an unbounded and densely defined operator acting between $L^2\Omega^{p,q}(M,E,g)$ and  $L^2\Omega^{p,q+1}(M,E,g).$
\end{prop}

\begin{proof}
This follows immediately by Prop. \ref{gaur} and Lemma 2.3 in \cite{BL}. 
\end{proof}

\begin{cor}
\label{salmo}
Let $(M,g)$, $(E,\rho)$ and $A$ be as in Prop. \ref{gaur}. Then the following three operators coincide:
\begin{align}
\label{salmon}
& \overline{\partial}_{E,p,q,\max}^t:L^2\Omega^{p,q+1}(A,E|_A,g|_A)\rightarrow L^2\Omega^{p,q}(A,E|_A,g|_A),\\
\label{salmone}
& \overline{\partial}_{E,p,q,\min}^t:L^2\Omega^{p,q+1}(A,E|_A,g|_A)\rightarrow L^2\Omega^{p,q}(A,E|_A,g|_A),\\
\label{bottarga}
& \overline{\partial}_{E,p,q}^t:L^2\Omega^{p,q+1}(M,E,g)\rightarrow L^2\Omega^{p,q}(M,E,g),
\end{align}
where, similarly  to  \eqref{ses}, \eqref{bottarga} is the unique closed extension of  $\overline{\partial}_{E,p,q}^t:\Omega^{p,q+1}(M,E)\rightarrow \Omega^{p,q}(M,E)$ viewed as an unbounded and densely defined operator acting between $L^2\Omega^{p,q+1}(M,E,g)$ and  $L^2\Omega^{p,q}(M,E,g).$
\end{cor}

\begin{proof}
This is an immediate application of Prop. \ref{lit}.
\end{proof}

We conclude this section by recalling the following property.
\begin{prop}
\label{brick}
Let $M$ be a complex manifold of complex dimension $m$, let $(E,\rho)$ be a Hermitian  vector bundle on $M$  and let $g$ and $h$ be two Hermitian metrics on $M$. Then we have an equality of Hilbert spaces $$L^2\Omega^{m,0}(M,E,g)=L^2\Omega^{m,0}(M,E,h)$$ and $$L^2\Omega^{0,m}(M,E,g)=L^2\Omega^{0,m}(M,E,h).$$  Assume now that $cg\geq h$ for some $c>0$. Then for each $q=1,...,m$ there exists a constant $\xi_q>0$ such that for every $s\in \Omega^{m,q}_c(M,E)$  we have 
\begin{equation}
\label{gamma}
\|s\|^2_{L^2 \Omega^{m,q}(M,E,g)}\leq \xi_q \|s\|^2_{L^2 \Omega^{m,q}(M,E,h)}.
\end{equation}
Therefore the identity $\Omega^{m,q}_c(M,E)\rightarrow \Omega^{m,q}_c(M,E)$ induces a continuous inclusion $$L^2 \Omega^{m,q}(M,E,h)\hookrightarrow L^2 \Omega^{m,q}(M,E,g)$$ for each $q=1,...,m$. Analogously for each $p=1,...,m$ there exists a constant   $\xi'_p>0$ such that 
for every $s\in \Omega^{p,m}_c(M,E)$ we have 
\begin{equation}
\label{sgamma}
\|s\|^2_{L^2 \Omega^{p,m}(M,E,g)}\leq \xi'_p \|s\|^2_{L^2 \Omega^{p,m}(M,E,h)}.
\end{equation}
 Therefore the identity $\Omega^{p,m}_c(M,E)\rightarrow \Omega^{p,m}_c(M,E)$ induces a continuous inclusion $$L^2 \Omega^{p,m}(M,E,h)\hookrightarrow L^2 \Omega^{p,m}(M,E,g)$$ for each $p=1,...,m$.
\end{prop} 

\begin{proof}
The statement   follows  by the computations carried out in \cite{GMMI} pag. 145. See also \cite{JRu} pag. 2896.
\end{proof}

\section{Main theorems }

This section contains some of the main results of this paper. We start by giving the following definition.
\begin{defi}
\label{pseudohermi}
Let $N$ be a complex manifold of complex dimension $n$. A Hermitian pseudometric $h$ on $N$ is a positive semidefinite Hermitian product on $N$ which is positive definite on an open and dense subset of $N$.
\end{defi}
We will label by $Z_h$ the smallest closed subset of $M$ such that  $h$ is positive definite on $M\setminus Z_h$. Thus $(M\setminus Z_h, h|_{M\setminus Z_h})$ becomes an \emph{incomplete} Hermitian manifold. Following \cite{PS} we will call $Z_h$ the \emph{degeneracy locus} of $h$. Consider now a compact complex manifold $M$ of complex dimension $m$. Let $(E,\rho)$ be a Hermitian holomorphic vector bundle on $M$. Let $h$ be a Hermitian pseudometric on $M$ with degeneracy locus $Z_h$. Let us define $A_h:=M\setminus Z_h$. Our aim is to study the closed extensions of the following operator
\begin{equation}
\label{model}
\overline{\partial}_{E,m,0}:L^2\Omega^{m,0}(A_h,E|_{A_h},h|_{A_h})\rightarrow L^2\Omega^{m,1}(A_h,E|_{A_h},h|_{A_h})
\end{equation}
where the domain of \eqref{model} is  $\Omega^{m,0}_c(A_h,E|_{A_h})$. In order to achieve this goal we need to consider an auxiliary Hermitian metric. More precisely let us fix an arbitrary Hermitian metric $g$ on $M$  and, as in Prop. \ref{lit} and Cor. \ref{salmo}, let us label by
\begin{equation}
\label{zita}
\overline{\partial}_{E,p,q}:L^2\Omega^{p,q}(M,E,g)\rightarrow L^2\Omega^{p,q+1}(M,E,g)
\end{equation}
 the unique closed extension of  $\overline{\partial}_{E,p,q}:\Omega^{p,q}(M,E)\rightarrow \Omega^{p,q+1}(M,E)$ and by
\begin{equation}
\label{zitaw}
\overline{\partial}_{E,p,q}^t:L^2\Omega^{p,q+1}(M,E,g)\rightarrow L^2\Omega^{p,q}(M,E,g)
\end{equation}
the unique closed extension of  $\overline{\partial}_{E,p,q}^t:\Omega^{p,q+1}(M,E)\rightarrow \Omega^{p,q}(M,E)$.
Moreover we observe that, since $h$ is positive semidefinite on $M$ and $g$ is positive definite on $M$, there exists a constant  $c>0$ such that $h\leq cg$. Hence, by Prop. \ref{brick}, we know that $L^2\Omega^{m,1}(A_h,E|_{A_h},h|_{A_h})\subset L^2\Omega^{m,1}(A_h,E|_{A_h},g|_{A_h})$ and that there exists a constant $\gamma>0$ such that for each $\omega\in L^2\Omega^{m,1}(A_h,E|_{A_h},h|_{A_h})$
\begin{equation}
\label{drago}
\|\omega\|^2_{L^2\Omega^{m,1}(A_h,E|_{A_h},g|_{A_h})}\leq \gamma \|\omega\|^2_{L^2\Omega^{m,1}(A_h,E|_{A_h},h|_{A_h})}.
\end{equation}
We have now all the ingredients for the first result of this section.
\begin{teo}
\label{firstth}
Let $M$, $g$, $h$, $A_h$ and $(E,\rho)$ be as defined above. Assume that $(A_h,g|_{A_h})$ is parabolic \footnote{As remarked in the introduction it is clear that this property does not depend on the particular Hermitian metric $g$ that we fix on $M$. More precisely if $g'$ is another Hermitian metric on $M$ then, since $g$ and $g'$ are quasi-isometric on $M$, we have that $(A_h,g|_{A_h})$ is parabolic if and only if $(A_h,g'|_{A_h})$ is parabolic.}.  Let 
\begin{equation}
\label{any}
\overline{\mathfrak{d}}_{E,m,0}:L^2\Omega^{m,0}(A_h,E|_{A_h},h|_{A_h})\rightarrow L^2\Omega^{m,1}(A_h,E|_{A_h},h|_{A_h})
\end{equation}
be any closed extension of \eqref{model}. Let $\mathcal{D}(\overline{\partial}_{E,m,0})$ and $\mathcal{D}(\overline{\mathfrak{d}}_{E,m,0})$ be the domains of \eqref{zita} (in bidegree $(m,0)$) and \eqref{any} respectively. Then the following properties hold true:
\begin{enumerate}
\item We have a continuous inclusion $\mathcal{D}(\overline{\mathfrak{d}}_{E,m,0})\hookrightarrow \mathcal{D}(\overline{\partial}_{E,m,0})$ where  each domain is endowed with the corresponding graph norm.  Moreover on $\mathcal{D}(\overline{\mathfrak{d}}_{E,m,0})$ the operator \eqref{zita} (in bidegree $(m,0)$) coincides with  the operator \eqref{any}.
\item The inclusion $\mathcal{D}(\overline{\mathfrak{d}}_{E,m,0})\hookrightarrow L^2\Omega^{m,0}(M,E,g)$ is a compact operator where $\mathcal{D}(\overline{\mathfrak{d}}_{E,m,0})$ is endowed with the corresponding graph norm.
\item Let $\overline{\mathfrak{d}}_{E,m,0}^*:L^2\Omega^{m,1}(A_h,E|_{A_h},h|_{A_h})\rightarrow L^2\Omega^{m,0}(A_h,E|_{A_h},h|_{A_h})$ be the adjoint of  \eqref{any}. Then  the operator
 \begin{equation}
\label{anylap}
\overline{\mathfrak{d}}_{E,m,0}^*\circ\overline{\mathfrak{d}}_{E,m,0}:L^2\Omega^{m,0}(A_h,E|_{A_h},h|_{A_h})\rightarrow L^2\Omega^{m,0}(A_h,E|_{A_h},h|_{A_h})
\end{equation} whose domain is defined as $\{s\in \mathcal{D}(\overline{\mathfrak{d}}_{E,m,0}):\ \overline{\mathfrak{d}}_{E,m,0}s\in \mathcal{D}(\overline{\mathfrak{d}}_{E,m,0}^*)\}$, has discrete spectrum.
\end{enumerate}
\end{teo}

\begin{proof}
First of all we point out that we have an equality of Hilbert spaces 
\begin{equation}
\label{chiment}
L^2\Omega^{m,0}(A_h,E|_{A_h},h|_{A_h})=L^2\Omega^{m,0}(A_h,E|_{A_h},g|_{A_h})=L^2\Omega^{m,0}(M,E,g).
\end{equation}
This follows by Prop. \ref{brick} and by the fact that $A_h$ is open and dense in $M$. Now we address the \emph{first point}. Let $\mathcal{D}(\overline{\partial}_{E,m,0,\max})$ be the domain of $$\overline{\partial}_{E,m,0,\max}:L^2\Omega^{m,0}(A_h,E|_{A_h},g|_{A_h})\rightarrow L^2\Omega^{m,1}(A_h,E|_{A_h},g|_{A_h}).$$ According to Prop. \ref{lit}, in order to prove the first point,  it is enough to  show that we have a continuous inclusion
 \begin{equation}
\label{chi}
\mathcal{D}(\overline{\mathfrak{d}}_{E,m,0})\hookrightarrow \mathcal{D}(\overline{\partial}_{E,m,0,\max})
\end{equation}
and that, for each $s\in \mathcal{D}(\overline{\mathfrak{d}}_{E,m,0})$, we have $\overline{\mathfrak{d}}_{E,m,0}s=\overline{\partial}_{E,m,0,\max}s$.
According to Prop. \ref{partiro} we know that $\Omega^{m,0}(A_h,E|_{A_h})\cap \mathcal{D}(\overline{\mathfrak{d}}_{E,m,0})$ is dense in $\mathcal{D}(\overline{\mathfrak{d}}_{E,m,0})$ with respect to the corresponding graph norm. Hence, in order to establish the existence of the continuous inclusion  \eqref{chi}, it is sufficient   to  prove that we have a continuous inclusion 
\begin{equation}
\label{laviniac}
\left(\Omega^{m,0}(A_h,E|_{A_h})\cap\mathcal{D}(\overline{\mathfrak{d}}_{E,m,0})\right)\hookrightarrow \mathcal{D}(\overline{\partial}_{E,m,0,\max}).
\end{equation}
 To this aim  let $s\in \Omega^{m,0}(A_h,E|_{A_h})\cap\mathcal{D}(\overline{\mathfrak{d}}_{E,m,0})$. Then, since $s$ is smooth on $A_h$, we  have $\overline{\mathfrak{d}}_{E,m,0}s=\overline{\partial}_{E,m,0}s$ where, with the operator on the right hand side of the previous equality, we mean the Dolbeault operator 
\begin{equation}
\label{innu}
\overline{\partial}_{E,m,0}:\Omega^{m,0}(A_h,E|_{A_h})\rightarrow \Omega^{m,1}(A_h,E|_{A_h}).
\end{equation}
Moreover, by \eqref{chiment}, we know that  $s\in L^2\Omega^{m,0}(A_h,E|_{A_h},g|_{A_h})$ and that 
\begin{equation}
\label{chime}
\|s\|_{L^2\Omega^{m,0}(A_h,E|_{A_h},g|_{A_h})}=\|s\|_{L^2\Omega^{m,0}(A_h,E|_{A_h},h|_{A_h})}.
\end{equation}
On the other hand, by \eqref{drago}, we know that  
\begin{equation}
\label{chimen}
\|\overline{\partial}_{E,m,0}s\|^2_{L^2\Omega^{m,1}(A_h,E|_{A_h},g|_{A_h})}\leq \gamma \|\overline{\partial}_{E,m,0}s\|^2_{L^2\Omega^{m,1}(A_h,E|_{A_h},h|_{A_h})}.
\end{equation}
Therefore $s\in \mathcal{D}(\overline{\partial}_{E,m,0,\max})$ and  by \eqref{chime} and \eqref{chimen} we get that \eqref{laviniac} is a continuous inclusion.\\ Now let $s\in  \mathcal{D}(\overline{\mathfrak{d}}_{E,m,0})$ and let $\omega\in L^2\Omega^{m,1}(A_h,E|_{A_h},h|_{A_h})$ such that $\overline{\mathfrak{d}}_{E,m,0}s=\omega$. Then there exists a sequence $\{s_n\}_{n\in \mathbb{N}}\subset \Omega^{m,0}(A_h,E|_{A_h})\cap \mathcal{D}(\overline{\mathfrak{d}}_{E,m,0})$ such that $s_n\rightarrow s$ in  $L^2\Omega^{m,0}(A_h,E|_{A_h},h|_{A_h})$ and $\overline{\partial}_{E,m,0}s_n\rightarrow \omega$ in $L^2\Omega^{m,1}(A_h,E|_{A_h},h|_{A_h})$ as $n\rightarrow \infty$. By \eqref{chime} and \eqref{chimen} we get that $s_n\rightarrow s$ in  $L^2\Omega^{m,0}(A_h,E|_{A_h},g|_{A_h})$ and $\overline{\partial}_{E,m,0}s_n\rightarrow \omega$ in $L^2\Omega^{m,1}(A_h,E|_{A_h},g|_{A_h})$ as $n\rightarrow \infty$. Therefore  $\overline{\partial}_{E,m,0,\max}s=\omega$ as desired. This establishes the first point of the theorem.  Now we address the \emph{second point}. According to the first point it is enough to show that the inclusion $\mathcal{D}(\overline{\partial}_{E,0,m})\hookrightarrow L^2\Omega^{m,0}(M,E,g)$ is a compact operator. To this aim consider the Hodge-Dolbeault operator 
\begin{equation}
\label{meglio}
\overline{\partial}_{E,m}+\overline{\partial}_{E,m}^t:\Omega^{m,\bullet}(M,E)\rightarrow \Omega^{m,\bullet}(M,E).
\end{equation}
By the fact that $M$ is compact and that $\overline{\partial}_{E,m}+\overline{\partial}_{E,m}^t$ is elliptic, we have a unique closed extension of \eqref{meglio} as an unbounded and densely defined operator acting on $L^2\Omega^{m,\bullet}(M,E,g)$. We label this unique extension by 
\begin{equation}
\label{megliox}
\overline{\partial}_{E,m}+\overline{\partial}_{E,m}^t:L^2\Omega^{m,\bullet}(M,E,g)\rightarrow L^2\Omega^{m,\bullet}(M,E,g).
\end{equation}
Let $\mathcal{D}(\overline{\partial}_{E,m}+\overline{\partial}_{E,m}^t)$ be the domain of \eqref{megliox}. Using again the fact that $M$ is compact and that \eqref{meglio} is elliptic we get that the inclusion 
\begin{equation}
\label{chimento}
\mathcal{D}(\overline{\partial}_{E,m}+\overline{\partial}_{E,m}^t)\hookrightarrow L^2\Omega^{m,\bullet}(M,E,g)
\end{equation}
where $\mathcal{D}(\overline{\partial}_{E,m}+\overline{\partial}_{E,m}^t)$ is endowed with its graph norm,  is a compact operator. On the other hand $\mathcal{D}(\overline{\partial}_{E,m}+\overline{\partial}_{E,m}^t)$ satisfies the following decomposition
\begin{equation}
\label{ottobre}
\mathcal{D}(\overline{\partial}_{E,m}+\overline{\partial}_{E,m}^t)=\bigoplus_{q=0}^{m}\mathcal{D}({\overline{\partial}_{E,m,q}})\cap\mathcal{D}({\overline{\partial}^t_{E,m,q-1}})
\end{equation}
where $\mathcal{D}({\overline{\partial}_{E,m,q}})$  is the   the domain of \eqref{zita} (in bidegree $(m,q)$) and  $\mathcal{D}({\overline{\partial}^t_{E,m,q-1}})$ it is the domain of \eqref{zitaw} (in bidegree $(m,q-1)$). Thus, by \eqref{chimento} and \eqref{ottobre}, we get immediately  that  the inclusion $\mathcal{D}({\overline{\partial}_{E,m,q}})\cap\mathcal{D}({\overline{\partial}^t_{E,m,q-1}})\hookrightarrow L^2\Omega^{m,q}(M,E,g)$ is a compact operator, where $\mathcal{D}({\overline{\partial}_{E,m,q}})\cap\mathcal{D}({\overline{\partial}^t_{E,m,q-1}})$ is endowed with the graph norm of $\overline{\partial}_{E,m}+\overline{\partial}_{E,m}^t$. In particular, when $q=0$, we have $\mathcal{D}({\overline{\partial}_{E,m,0}})\cap\mathcal{D}({\overline{\partial}^t_{E,m,-1}})=\mathcal{D}({\overline{\partial}_{E,m,0}})$ and  $\overline{\partial}_{E,m}+\overline{\partial}_{E,m}^t$  acting on $\mathcal{D}({\overline{\partial}_{E,m,0}})$ is simply $\overline{\partial}_{E,m,0}$. In this way we can conclude  that the inclusion 
$\mathcal{D}(\overline{\partial}_{E,m,0})\hookrightarrow L^2\Omega^{m,0}(M,E,g)$ is a compact operator and this completes the proof of the second point. Finally we prove the \emph{third point}. By the fact that   $\overline{\mathfrak{d}}_{E,m,0}^*\circ\overline{\mathfrak{d}}_{E,m,0}:L^2\Omega^{m,0}(A_h,E|_{A_h},h|_{A_h})\rightarrow L^2\Omega^{m,0}(A_h,E|_{A_h},h|_{A_h})$ is self-adjoint the third point is equivalent to showing that  the inclusion $\mathcal{D}(\overline{\mathfrak{d}}_{E,m,0}^*\circ\overline{\mathfrak{d}}_{E,m,0})\hookrightarrow L^2\Omega^{m,0}(A_h,E|_{A_h},h|_{A_h})$ is a compact operator where  $\mathcal{D}(\overline{\mathfrak{d}}_{E,m,0}^*\circ \overline{\mathfrak{d}}_{E,m,0})$ is endowed with the corresponding graph norm, see \cite{MaMa} pag. 381.  For each $s\in \mathcal{D}(\overline{\mathfrak{d}}_{E,m,0}^*\circ \overline{\mathfrak{d}}_{E,m,0})$, we have 
\begin{align}
\nonumber & \|\overline{\mathfrak{d}}_{E,m,0}s\|^2_{L^2\Omega^{m,1}(A_h,E|_{A_h},h|_{A_h})}= \langle s,\overline{\mathfrak{d}}_{E,m,0}^*(\overline{\mathfrak{d}}_{E,m,0}s)\rangle_{L^2\Omega^{m,0}(A_h,E|_{A_h},h|_{A_h})}\leq \\
\nonumber &  \frac{1}{2}\left( \|s\|^2_{L^2\Omega^{m,0}(A_{h},E|_{A_h},h|_{A_h})}+\|\overline{\mathfrak{d}}_{E,m,0}^*(\overline{\mathfrak{d}}_{E,m,0}s)\|^2_{L^2\Omega^{m,0}(A_h,E|_{A_h},h|_{A_h})}\right)
\end{align}
and therefore
\begin{align}
\nonumber & \|s\|^2_{L^2\Omega^{m,0}(A_h,E|_{A_h},h|_{A_h})}+\|\overline{\mathfrak{d}}_{E,m,0}s\|^2_{L^2\Omega^{m,1}(A_h,E|_{A_h},h|_{A_h})}\leq \\
\nonumber & \frac{3}{2} \left(\|s\|^2_{L^2\Omega^{m,0}(A_{h},E|_{A_h},h|_{A_h})}+\|\overline{\mathfrak{d}}_{E,m,0}^*(\overline{\mathfrak{d}}_{E,m,0}s)\|^2_{L^2\Omega^{m,0}(A_h,E|_{A_h},h|_{A_h})}\right).
\end{align}
 The above inequality tells us that we have a continuous inclusion $\mathcal{D}(\overline{\mathfrak{d}}_{E,m,0}^*\circ\overline{\mathfrak{d}}_{E,m,0})\hookrightarrow \mathcal{D}(\overline{\mathfrak{d}}_{E,m,0})$ where each domain is endowed with corresponding graph norm. Now, using the first two points of this theorem, we finally get that the inclusion $\mathcal{D}(\overline{\mathfrak{d}}_{E,m,0}^*\circ\overline{\mathfrak{d}}_{E,m,0})\hookrightarrow L^2\Omega^{m,0}(A_h,E|_{A_h},h|_{A_h})$ is a compact operator as desired. The proof of theorem is thus complete. 
\end{proof}

We have the following immediate corollary.
\begin{cor}
\label{freddy}
In the setting of Theorem \ref{firstth}. The following properties hold true:
\begin{enumerate}
\item $\overline{\mathfrak{d}}_{E,m,0}^*\circ\overline{\mathfrak{d}}_{E,m,0}:L^2\Omega^{m,0}(A_h,E|_{A_h},h|_{A_h})\rightarrow L^2\Omega^{m,0}(A_h,E|_{A_h},h|_{A_h})$ is a Fredholm operator on its domain endowed with the graph norm.
\item $\im(\overline{\mathfrak{d}}_{E,m,0})$ is a closed subset of $L^2\Omega^{m,1}(A_h,E|_{A_h},h|_{A_h})$.
\item $\ker(\overline{\mathfrak{d}}_{E,m,0})$ is finite dimensional. 
\item We have the following $L^2$-orthogonal decomposition $$L^2\Omega^{m,0}(A_h,E|_{A_h},h|_{A_h})=\ker(\overline{\mathfrak{d}}_{E,m,0})\oplus \im(\overline{\mathfrak{d}}_{E,m,0}^*).$$
\end{enumerate} 
\end{cor}

\begin{proof}
By Th. \ref{firstth} we know  that $\overline{\mathfrak{d}}_{E,m,0}^*\circ\overline{\mathfrak{d}}_{E,m,0}:L^2\Omega^{m,0}(A_h,E|_{A_h},h|_{A_h})\rightarrow L^2\Omega^{m,0}(A_h,E|_{A_h},h|_{A_h})$  has discrete spectrum and this in turn implies  in particular  that it is a Fredholm operator on its domain endowed with the graph norm. Therefore we can conclude that  $\im(\overline{\mathfrak{d}}_{E,m,0}^*\circ\overline{\mathfrak{d}}_{E,m,0})$ is closed in $L^2\Omega^{m,0}(A_h,E|_{A_h},h|_{A_h})$. By \eqref{cocapro} we have  the following two orthogonal decompositions for $L^2\Omega^{m,0}(A_h,E|_{A_h},h|_{A_h})$ : $$L^2\Omega^{m,0}(A_h,E|_{A_h},h|_{A_h})=\ker(\overline{\mathfrak{d}}_{E,m,0})\oplus \overline{\im(\overline{\mathfrak{d}}_{E,m,0}^*)}=\ker(\overline{\mathfrak{d}}_{E,m,0}^*\circ\overline{\mathfrak{d}}_{E,m,0})\oplus \overline{\im(\overline{\mathfrak{d}}_{E,m,0}^*\circ\overline{\mathfrak{d}}_{E,m,0})}.$$ Clearly $\ker(\overline{\mathfrak{d}}_{E,m,0})=\ker(\overline{\mathfrak{d}}_{E,m,0}^*\circ\overline{\mathfrak{d}}_{E,m,0})$. Hence we have the following chain of inclusions: $$\im(\overline{\mathfrak{d}}_{E,m,0}^*\circ\overline{\mathfrak{d}}_{E,m,0})\subset \im(\overline{\mathfrak{d}}_{E,m,0}^*)\subset \overline{\im(\overline{\mathfrak{d}}_{E,m,0}^*)}= \overline{\im(\overline{\mathfrak{d}}_{E,m,0}^*\circ\overline{\mathfrak{d}}_{E,m,0})}=\im(\overline{\mathfrak{d}}_{E,m,0}^*\circ\overline{\mathfrak{d}}_{E,m,0})$$ which in particular implies that $\overline{\im(\overline{\mathfrak{d}}_{E,m,0}^*)}=\im(\overline{\mathfrak{d}}_{E,m,0}^*)$ and therefore, taking the adjoint, $\overline{\im(\overline{\mathfrak{d}}_{E,m,0})}=\im(\overline{\mathfrak{d}}_{E,m,0})$ as required. All the other points are immediate consequences of the  first two points of this corollary  and  Th. \ref{firstth}.
\end{proof}

Consider again the setting of Th. \ref{firstth}. Let  $\Delta_{\overline{\partial},E,m,0}:\Omega^{m,0}(M,E)\rightarrow \Omega^{m,0}(M,E)$, $\Delta_{\overline{\partial},E,m,0}=\overline{\partial}_{E,m,0}^t\circ\overline{\partial}_{E,m,0}$ be the Hodge-Kodaira Laplacian in bidegree $(m,0)$. Using again the fact that $M$ is compact and that $\Delta_{\overline{\partial},E,m,0}$ is elliptic and formally self-adjoint we can conclude that $\Delta_{\overline{\partial},E,m,0}$, acting on $L^2\Omega^{m,0}(M,E,g)$ with  domain $\Omega^{m,0}(M,E)$, is essentially self-adjoint; we label its unique (and therefore self-adjoint) extension by 
\begin{equation}
\label{crema}
\Delta_{\overline{\partial},E,m,0}:L^2\Omega^{m,0}(M,E,g)\rightarrow L^2\Omega^{m,0}(M,E,g).
\end{equation}
Clearly we can write  \eqref{crema}  as $\overline{\partial}_{E,m,0}^t\circ\overline{\partial}_{E,m,0}$ where 
\begin{equation}
\label{isabel}
\overline{\partial}_{E,m,0}:L^2\Omega^{m,0}(M,E,g)\rightarrow L^2\Omega^{m,1}(M,E,g)
\end{equation}
 is defined in \eqref{zita} and $\overline{\partial}_{E,m,0}^t:L^2\Omega^{m,1}(M,E,g)\rightarrow L^2\Omega^{m,0}(M,E,g)$ is defined in \eqref{zitaw}.  Furthermore it is another standard result from classical elliptic theory on closed manifolds that \eqref{crema}  has discrete spectrum. We are now in the position to state the other main result of this section.

\begin{teo}
\label{secondo}
In the  setting of Theorem \ref{firstth}. Let $$0\leq \mu_1\leq \mu_2\leq...\leq \mu_k\leq...$$ be the eigenvalues of \eqref{crema} and let  $$0\leq \lambda_1\leq \lambda_2\leq...\leq \lambda_{k}\leq...$$ be the eigenvalues of the operator 
\begin{equation}
\label{lachi}
\overline{\mathfrak{d}}_{E,m,0}^*\circ\overline{\mathfrak{d}}_{E,m,0}:L^2\Omega^{m,0}(A_h,E|_{A_h},h|_{A_h})\rightarrow L^2\Omega^{m,0}(A_h,E|_{A_h},h|_{A_h}).
\end{equation}
Then, for every $k\in \mathbb{N}$, we have the following inequality
\begin{equation}
\label{ine}
\gamma\lambda_k\geq \mu_k
\end{equation}
where $\gamma$ is the constant introduced  in \eqref{drago}. Moreover we have the following asymptotic inequality:
\begin{equation}
\label{asine}
\lim \inf \lambda_k k^{-\frac{1}{m}}>0
\end{equation} 
as $k\rightarrow \infty$.
\end{teo}

\begin{proof}
According to the min-max theorem, see for instance \cite{KSC}, we can characterize the eigenvalues of \eqref{crema} as 
$$\mu_k=\inf_{F\in \mathfrak{F}_k\cap\mathcal{D}(\Delta_{\overline{\partial},E,m,0})}\sup_{s\in F}\frac{\langle\Delta_{\overline{\partial},E,m,0}s,s\rangle_{L^2\Omega^{m,0}(M,E,g)}}{\|s\|^2_{L^2\Omega^{m,0}(M,E,g)}}=\inf_{F\in \mathfrak{F}_k\cap\mathcal{D}(\overline{\partial}_{E,m,0})}\sup_{s\in F}\frac{\langle \overline{\partial}_{E,m,0}s,\overline{\partial}_{E,m,0}s\rangle_{L^2\Omega^{m,1}(M,E,g)}}{\|s\|^2_{L^2\Omega^{m,0}(M,E,g)}}$$
where $\mathfrak{F}_k$ denotes the set of linear subspaces of $L^2\Omega^{m,0}(M,E,g)$ of dimension at most $k$, $\mathcal{D}(\Delta_{\overline{\partial},E,m,0})$ is the domain of \eqref{crema} and $\mathcal{D}(\overline{\partial}_{E,m,0})$ is the domain of \eqref{isabel}. In the same way for the eigenvalues of \eqref{lachi} we have 
\begin{align}
\label{mini}
& \lambda_k=\inf_{F\in \mathfrak{F}_k\cap\mathcal{D}(\overline{\mathfrak{d}}_{E,m,0}^*\circ\overline{\mathfrak{d}}_{E,m,0})}\sup_{s\in F}\frac{\langle \overline{\mathfrak{d}}_{E,m,0}^*\circ\overline{\mathfrak{d}}_{E,m,0}s,s\rangle_{L^2\Omega^{m,0}(A_h,E|_{A_h},h|_{A_h})}}{\|s\|^2_{L^2\Omega^{m,0}(A_h,E|_{A_h},h|_{A_h})}}=\\
\nonumber &\inf_{F\in \mathfrak{F}_k\cap\mathcal{D}(\overline{\mathfrak{d}}_{E,m,0})}\sup_{s\in F}\frac{\langle \overline{\mathfrak{d}}_{E,m,0}s,\overline{\mathfrak{d}}_{E,m,0}s\rangle_{L^2\Omega^{m,1}(A_h,E|_{A_h},h|_{A_h})}}{\|s\|^2_{L^2\Omega^{m,0}(A_h,E|_{h},h|_{A_h})}}.
\end{align}
Let now $\{\phi_n,\ n\in \mathbb{N}\}$ be an orthonormal basis of $L^2\Omega^{m,0}(A_h,E|_{A_h},h|_{A_h})$ made of eigensections of \eqref{lachi} such that $(\overline{\mathfrak{d}}_{E,m,0}^*\circ\overline{\mathfrak{d}}_{E,m,0})\phi_n=\lambda_n\phi_n$. Let us define  $F_k\in \mathfrak{F}_k$ as the $k$-dimensional subspace of $L^2\Omega^{m,0}(A_h,E|_{A_h},h|_{A_h})$ generated by $\{\phi_1,...,\phi_k\}$. Then, see for instance \cite{KSC} pag. 279, we have $$\lambda_k=\sup_{s\in F_k}\frac{\langle \overline{\mathfrak{d}}_{E,m,0}s,\overline{\mathfrak{d}}_{E,m,0}s\rangle_{L^2\Omega^{m,1}(A_h,E|_{A_h},h|_{A_h})}}{\|s\|^2_{L^2\Omega^{m,0}(A_h,E|_{A_h},h|_{A_h})}}.$$
Hence, using Theorem \ref{firstth} and \eqref{chimen}, we have 
\begin{align}
\nonumber & \gamma\lambda_k= \gamma\sup_{s\in F_k}\frac{\langle\overline{\mathfrak{d}}_{E,m,0}s,\overline{\mathfrak{d}}_{E,m,0}s\rangle_{L^2\Omega^{m,1}(A_h,E|_{A_h},h|_{A_h})}}{\|s\|^2_{L^2\Omega^{m,0}(A_h,E|_{A_h},h|_{A_h})}} \geq  \sup_{s\in F_k}\frac{\langle\overline{\partial}_{E,m,0}s,\overline{\partial}_{E,m,0}s\rangle_{L^2\Omega^{m,1}(M,E,g)}}{\|s\|^2_{L^2\Omega^{m,0}(M,E,g)}}\\
\nonumber & \geq  \inf_{F\in \mathfrak{F}_k\cap\mathcal{D}(\overline{\partial}_{E,m,0})}\sup_{s\in F}\frac{\langle \overline{\partial}_{E,m,0}s,\overline{\partial}_{E,m,0}s\rangle_{L^2\Omega^{m,1}(M,E,g)}}{\|s\|^2_{L^2\Omega^{m,0}(M,E,g)}}=\mu_k.
\end{align}
This establishes \eqref{ine}. Concerning \eqref{asine} it is enough to observe that according to  the Weyl law, see \cite{Jroe} pag.  115, we have $$\mu_k\sim 4\pi\left(\frac{1}{\Gamma(m+1)}\vol_g(M)\right)^{\frac{1}{m}}k^{\frac{1}{m}}$$ as $k\rightarrow \infty$. Applying now \eqref{ine} the desired conclusion follows.
\end{proof}
 As a consequence of  Theorem \ref{secondo} we have now the next corollary. 
\begin{cor}
\label{heat}
In the  setting of Theorem \ref{secondo}. Let $$e^{-t\overline{\mathfrak{d}}_{E,m,0}^*\circ\overline{\mathfrak{d}}_{E,m,0}}:L^2\Omega^{m,0}(A_h,E|_{A_h},h|_{A_h})\rightarrow L^2\Omega^{m,0}(A_h,E|_{A_h},h|_{A_h})$$ be  the heat operator associated to \eqref{lachi} and analogously let $$e^{-t\Delta_{\overline{\partial},E,m,0}}:L^2\Omega^{m,0}(M,E,g)\rightarrow L^2\Omega^{m,0}(M,E,g)$$ be  the heat operator associated to \eqref{crema}. We have the following properties.
\begin{enumerate}
\item $e^{-t\overline{\mathfrak{d}}_{E,m,0}^*\circ\overline{\mathfrak{d}}_{E,m,0}}:L^2\Omega^{m,0}(A_h,E|_{A_h},h|_{A_h})\rightarrow L^2\Omega^{m,0}(A_h,E|_{A_h},h|_{A_h})$ is a trace class operator.
\item $\Tr(e^{-t\overline{\mathfrak{d}}_{E,m,0}^*\circ \overline{\mathfrak{d}}_{E,m,0}})\leq\Tr(e^{-\frac{t}{\gamma}\Delta_{\overline{\partial},E,m,0}})$ for each $t>0$ and where $\gamma$ is the constant defined  in \eqref{drago}.
\item $\Tr(e^{-t\overline{\mathfrak{d}}_{E,m,0}^*\circ\overline{\mathfrak{d}}_{E,m,0}})\leq Ct^{-m}$ for $t\in (0,1]$ and  some constant  $C>0$.
\end{enumerate}
\end{cor}

\begin{proof}
As in the proof of Theorem \ref{secondo} let us label by $\{\phi_n,\ n\in \mathbb{N}\}$  an orthonormal basis of $L^2\Omega^{m,0}(A_h,E|_{A_h},h|_{A_h})$ made of eigensections of \eqref{lachi} such that $(\overline{\mathfrak{d}}_{E,m,0}^*\circ\overline{\mathfrak{d}}_{E,m,0})\phi_n=\lambda_n\phi_n$.
Then we have $$\Tr(e^{-t\overline{\mathfrak{d}}_{E,m,0}^*\circ\overline{\mathfrak{d}}_{E,m,0}})=\sum_{n\in \mathbb{N}}\langle e^{-t\overline{\mathfrak{d}}_{E,m,0}^*\circ\overline{\mathfrak{d}}_{E,m,0}}\phi_n,\phi_n\rangle_{L^2\Omega^{m,0}(A_h,E|_{A_h},h|_{A_h})}=\sum_{n\in \mathbb{N}}e^{-t\lambda_n}<\infty$$ for every fixed $t>0$ as a consequence of \eqref{asine}. This establishes the first point of the corollary. The second point follows by \eqref{ine}. Indeed we have $$\Tr(e^{-t\overline{\mathfrak{d}}_{E,m,0}^*\circ\overline{\mathfrak{d}}_{E,m,0}})=\sum_{n\in \mathbb{N}}e^{-t\lambda_n}\leq \sum_{n\in \mathbb{N}}e^{-t\frac{\mu_n}{\gamma}}=\sum_{n\in \mathbb{N}}e^{-\frac{t}{\gamma}\mu_n}=\Tr (e^{-\frac{t}{\gamma}\Delta_{\overline{\partial},E,m,0}})$$ for every fixed $t>0$. Finally for the third point we argue as follows. Let $b>0$ and let us define $g':=b^2g$. Clearly $g'$ is another Hermitian metric on $M$ and we have 
$$\|\ \|_{L^2\Omega^{m,q}(M,E,g')}=\frac{1}{b^q}\|\ \|_{L^2\Omega^{m,q}(M,E,g)}.$$
Therefore, choosing $b$ in such a way  that $b^{2}\geq \gamma$, we have $$b^{-2}\|\ \|^2_{L^2\Omega^{m,1}(A_h,E|_{A_h},g|_{A_h})}\leq \gamma^{-1}\|\ \|^2_{L^2\Omega^{m,1}(A_h,E|_{A_h},g|_{A_h})}\leq \|\ \|^2_{L^2\Omega^{m,1}(A_h,E|_{A_h},h|_{A_h})}$$ that is 
\begin{equation}
\label{drago2}
\|\ \|^2_{L^2\Omega^{m,1}(A_h,E|_{A_h},g'|_{A_h})}\leq \|\ \|^2_{L^2\Omega^{m,1}(A_h,E|_{A_h},h|_{A_h})}.
\end{equation}
Now, in the remaining part of this proof, let us replace the Hermitian metric $g$ with $g'$. Applying the second point of this corollary to $h$ and $g'$ and using  \eqref{drago2} instead of \eqref{drago}, we get  that 
\begin{equation}
\label{rsara}
\Tr(e^{-t\overline{\mathfrak{d}}_{E,m,0}^*\circ \overline{\mathfrak{d}}_{E,m,0}})\leq\Tr(e^{-t\Delta_{\overline{\partial},E,m,0}})
\end{equation}
for every $t>0$, where now,  on the right hand side of \eqref{rsara}, $\Delta_{\overline{\partial},E,m,0}:L^2\Omega^{m,0}(M,E,g')\rightarrow L^2\Omega^{m,0}(M,E,g')$ is the unique closed extension of the Hodge-Kodaira Laplacian in bidegree $(m,0)$ associated to $g'$.
Finally the conclusion  follows immediately  by the asymptotic expansion of the heat trace on the right hand side of \eqref{rsara}, see for instance \cite{BGV} or \cite{Jroe}. 
\end{proof}

We stress on the fact  that the results of this section hold for  $$\overline{\mathfrak{d}}_{E,m,0}:L^2\Omega^{m,0}(A_h,E|_{A_h},h)\rightarrow L^2\Omega^{m,1}(A_h,E|_{A_h},h)$$ which is defined as \emph{any closed extension} of    $\overline{\partial}_{E,m,0}:L^2\Omega^{m,0}(A_h,E|_{A_h},h|_{A_h})\rightarrow L^2\Omega^{m,1}(A_h,E|_{A_h},h|_{A_h})$ where the domain of the latter operator is  $\Omega^{m,0}_c(A_h,E|_{A_h})$. Therefore Th. \ref{firstth}, Th. \ref{secondo} and the corresponding corollaries apply in particular to $$\overline{\partial}_{E,m,0,\max/\min}:L^2\Omega^{m,0}(A_h,E|_{A_h},h|_{A_h})\rightarrow L^2\Omega^{m,1}(A_h,E|_{A_h},h|_{A_h}).$$
Now we have the following remark.

\begin{rem}
\label{bottchern}
Let $A'_h$ be an open and dense subset of $A_h$. Assume that $(A'_h,g|_{A'_h})$ is still parabolic. Then it is clear that we can reformulate Th. \ref{firstth}, Th. \ref{secondo}, Cor. \ref{freddy} and  Cor. \ref{heat}  by replacing $A'_h$ with $A_h$ in the corresponding statements.
\end{rem}

Finally we conclude this section with the following comment. Analogously to \eqref{model} let us consider the operator
\begin{equation}
\label{smodel}
\partial_{E,m,0}:L^2\Omega^{0,m}(A_h,E|_{A_h},h|_{A_h})\rightarrow L^2\Omega^{1,m}(A_h,E|_{A_h},h|_{A_h})
\end{equation}
with $\Omega^{0,m}_c(A_h,E|_{A_h})$ as  domain
and let us label by 
\begin{equation}
\label{sany}
\mathfrak{d}_{E,m,0}:L^2\Omega^{0,m}(A_h,E|_{A_h},h|_{A_h})\rightarrow L^2\Omega^{1,m}(A_h,E|_{A_h},h|_{A_h})
\end{equation}
any closed extension of \eqref{smodel}. Then by Prop. \ref{brick} and using analogous strategies to those employed in the previous proofs it is clear that  the corresponding versions   of Theorem \ref{firstth}, Cor. \ref{freddy}, Theorem \ref{secondo} and Cor. \ref{heat} hold true for \eqref{sany}.

\section{Applications}
This section contains various applications of the previous results. We start by applying Th. \ref{firstth}, Th. \ref{secondo} and their corollaries to the case of a compact and irreducible Hermitian complex space. The second part concerns the existence of self-adjoint extensions with discrete spectrum of the Hodge-Kodaira Laplacian  in the setting of  compact and irreducible Hermitian complex spaces with isolated singularities. Finally, in the last part, the Hodge-Kodaira Laplacian on  complex projective surfaces is carefully studied.

\subsection{Hermitian complex spaces}
We start with the following proposition which furnishes a sufficient condition for the  parabolicity of certain incomplete Riemannian manifolds.
\begin{prop}
\label{parasuf}
Consider a  compact Riemannian manifold $(M,g)$. Let $\Sigma\subset M$ be a subset made of a finite union of closed  submanifolds, $\Sigma=\cup_{i=1}^m S_i$ such that each submanifold $S_i$ has codimension greater than one, that is $\mathrm{cod}(S_i)\geq 2$. Let $A$ be defined as $M\setminus  \Sigma$ and consider the restriction of $g$ over $A$, $g|_{A}$.
Then $(A,g|_{A})$ is parabolic. 
\end{prop}

\begin{proof}
See  \cite{BeGu}.
\end{proof}

Let $M$ be a  complex manifold of complex dimension $m$. We recall that a divisor on $M$ is a locally finite combination $\sum_kc_kV_k$ where $V_k$ are irreducible analytic hypersurfaces of $M$ and $c_k\in \mathbb{Z}$. The support of $D$ is defined as  $\supp(D)=\cup \{V_k:c_k\neq 0\}$. A \emph{divisor with only normal crossings}  is a divisor of the form $D=\sum_kV_k$ where $V_k$ are distinct irreducible smooth hypersurfaces  of $M$ and, for any point $p\in \supp(D)$,  there is a local analytic coordinate system $(U,z_1,...,z_m)$ such that  $\supp(D)\cap U=\{z_1\cdot...\cdot z_k=0\}$ for some $1\leq k\leq m$. Sometimes in the rest of the paper, with a little abuse of notation, we will identify a divisor with only normal crossings with its support. Now we have the following immediate application of Prop. \ref{parasuf}.
\begin{prop}
\label{exam}
 Let $M$ be a compact complex manifold and let $g$ be a Hermitian metric on $M$. Let $D\subset M$ be a divisor with only normal crossings.  Then $(M\setminus D, g|_{M\setminus D})$ is parabolic.
\end{prop}

\begin{proof}
This follows immediately by Prop. \ref{parasuf} and the definition of divisor with only normal crossings.
\end{proof}

%

The next results deal with complex  spaces. This is a classical topic in complex geometry and we refer to \cite{GeFi} and \cite{GrRe} for definitions and properties. Here we recall that an irreducible complex space $X$ is a reduced complex space such that $\reg(X)$, the regular part of $X$, is connected. Moreover, in order to state the next results, we spend a few words concerning the resolution of singularities. We refer to the celebrated work of Hironaka \cite{Hiro},  to \cite{BiMi} and to \cite{HH} for a thorough discussion  on this subject. Furthermore we refer to \cite{GMMI} and to \cite{MaMa}  for a quick introduction. Below we simply recall what is strictly necessary for our purposes.\\ Let $X$ be a compact irreducible complex space. Then there exists a compact complex manifold $M$, a divisor with only normal crossings $D\subset M$ and a surjective holomorphic map $\pi:M\rightarrow X$ such that $\pi^{-1}(\sing(X))=D$ and 
\begin{equation}
\label{hiro}
\pi|_{M\setminus D}: M\setminus D\longrightarrow X\setminus \sing(X)
\end{equation}
is a biholomorphism. Furthermore, before to introduce the next results, we recall  that a paracompact  and reduced complex space $X$ is said \emph{Hermitian} if  the regular part of $X$ carries a Hermitian metric $h$  such that for every point $p\in X$ there exists an open neighborhood $U\ni p$ in $X$, a proper holomorphic embedding of $U$ into a polydisc $\phi: U \rightarrow \mathbb{D}^N\subset \mathbb{C}^N$ and a Hermitian metric $g$ on $\mathbb{D}^N$ such that $(\phi|_{\reg(U)})^*g=h$, see for instance \cite{Ohsa} or \cite{JRu}. In this case we will write $(X,h)$ and with a little abuse of language we will say that $h$ is a Hermitian metric on $X$. A natural example of Hermitian complex space is provided by an analytic sub-variety of a complex Hermitian manifold with the metric induced by the restriction of the  metric of the ambient space. In particular, within this class of examples, we have any complex projective variety $V\subset \mathbb{C}\mathbb{P}^n$ endowed  with the K\"ahler metric induced by the Fubini-Study metric of $\mathbb{C}\mathbb{P}^n$.\\ 
%
Consider now a  compact and irreducible  Hermitian complex space $(X,h)$ of complex dimension $m$. Let $\pi: M \longrightarrow X$ be a resolution of $X$ as described in \eqref{hiro}  and let $D\subset M$ be the divisor with only normal crossings such that  $\pi|_{M\setminus D}: M\setminus D\longrightarrow X\setminus \sing(X)$ is a biholomorphism. Let $(E,\rho)$ be a Hermitian holomorphic vector bundle on $\reg(X)$ such that there exists a Hermitian holomorphic vector bundle on $M$, $(F,\tau)$, which satisfies $(\pi|_{M\setminus D}^{-1})^*(F|_{M\setminus D})=E$ and $(\pi|_{M\setminus D}^{-1})^*(\tau|_{M\setminus D})=\rho$. Consider the operator 
\begin{equation}
\label{jiji}
\overline{\partial}_{E,m,0}:\Omega_c^{m,0}(\reg(X),E)\rightarrow \Omega_c^{m,1}(\reg(X),E)
\end{equation}
 and, similarly to  \eqref{any},  let 
\begin{equation}
\label{lola}
\overline{\mathfrak{d}}_{E,m,0}:L^2\Omega^{m,0}(\reg(X),E,h)\rightarrow L^2\Omega^{m,1}(\reg(X),E,h)
\end{equation}
be any closed extension of \eqref{jiji}. Let us label by $\mathcal{D}(\overline{\mathfrak{d}}_{E,m,0})$ the domain of \eqref{lola}. We are finally in the position to state the next theorem. It gathers the applications   of the results proved in the previous section to the case  of compact irreducible  Hermitian complex spaces.

\begin{teo}
\label{canonical}
In the setting described above. We have the following properties:
\begin{enumerate}
\item The inclusion $\mathcal{D}(\overline{\mathfrak{d}}_{E,m,0})\hookrightarrow L^2\Omega^{m,0}(\reg(X),E,h)$ is a compact operator where $\mathcal{D}(\overline{\mathfrak{d}}_{E,m,0})$ is endowed with the corresponding graph norm.
\item Let $\overline{\mathfrak{d}}_{E,m,0}^*:L^2\Omega^{m,1}(\reg(X),E,h)\rightarrow L^2\Omega^{m,0}(\reg(X),E,h)$ be the adjoint of  \eqref{lola}. Then  the operator 
\begin{equation}
\label{anylapz}
\overline{\mathfrak{d}}_{E,m,0}^*\circ\overline{\mathfrak{d}}_{E,m,0}:L^2\Omega^{m,0}(\reg(X),E,h)\rightarrow L^2\Omega^{m,0}(\reg(X),E,h)
\end{equation} whose domain is defined as $\{s\in \mathcal{D}(\overline{\mathfrak{d}}_{E,m,0}):\ \overline{\mathfrak{d}}_{E,m,0}s\in \mathcal{D}(\overline{\mathfrak{d}}_{E,m,0}^*)\}$, has discrete spectrum.
\end{enumerate}
Let $$0\leq \lambda_1\leq \lambda_2\leq \lambda_3\leq...$$ be the eigenvalues of \eqref{anylapz}. Then we have the following asymptotic inequality 
\begin{equation}
\label{rospo}
\lim \inf \lambda_k k^{-\frac{1}{m}}>0
\end{equation} 
as $k\rightarrow \infty$.\\
Finally consider  the heat operator $$e^{-t\overline{\mathfrak{d}}_{E,m,0}^*\circ\overline{\mathfrak{d}}_{E,m,0}}:L^2\Omega^{m,0}(\reg(X),E,h)\rightarrow L^2\Omega^{m,0}(\reg(X),E,h)$$ associated to \eqref{anylapz}. We have the following properties:
\begin{enumerate}
\item $e^{-t\overline{\mathfrak{d}}_{E,m,0}^*\circ\overline{\mathfrak{d}}_{E,m,0}}:L^2\Omega^{m,0}(\reg(X),E,h)\rightarrow L^2\Omega^{m,0}(\reg(X),E,h)$ is a trace class operator.
\item $\Tr(e^{-t\overline{\mathfrak{d}}_{E,m,0}^*\circ\overline{\mathfrak{d}}_{E,m,0}})\leq C t^{-m}$ for $t\in (0,1]$ and  for some constant $C>0$.
\end{enumerate}
\end{teo}

\begin{proof}
Let $M$, $\pi$ and $D$ be as in \eqref{hiro} such that, as required above, we have $(\pi|_{M\setminus D}^{-1})^*(F|_{M\setminus D})=E$ and $(\pi|_{M\setminus D}^{-1})^*(\tau|_{M\setminus D})=\rho$. Let $h':=(\pi|_{M\setminus D})^*h$.  Since $h$ is locally given by an embedding we  have that $h'$  extends as a Hermitian pseudometric on $M$ which is positive definite on $M\setminus D$. Moreover, according to Prop. \ref{exam}, we know that $(M\setminus D, g|_{M\setminus D})$ is parabolic where $g$ is any Hermitian metric on $M$. Furthermore, by the assumptions made on $(E,\rho)$, we know that 
\begin{equation}
\label{sonno}
(\pi|_{M\setminus D})^*:L^2\Omega^{m,q}(\reg(X),E,h)\rightarrow L^2\Omega^{m,q}(M\setminus D, F|_{M\setminus D},h'|_{M\setminus D})
\end{equation} is a unitary operator for each $q=0,...,m$. Now, in order to have a lighter notation let us label by $T: L^2\Omega^{m,q}(\reg(X),E,h)\rightarrow L^2\Omega^{m,q}(M\setminus D, F|_{M\setminus D},h'|_{M\setminus D})$ the operator \eqref{sonno}. Then the operator 
\begin{equation}
\label{risonno}
T\circ \overline{\mathfrak{d}}_{E,m,0}\circ T^{-1}:L^2\Omega^{m,0}(M\setminus D, F|_{M\setminus D},h'|_{M\setminus D})\rightarrow L^2\Omega^{m,1}(M\setminus D, F|_{M\setminus D},h'|_{M\setminus D})
\end{equation}
with domain given by $T(\mathcal{D}(\overline{\mathfrak{d}}_{E,m,0}))$, is a closed extension of $\overline{\partial}_{F,m,0}:\Omega^{m,0}_c(M\setminus D,F|_{M\setminus D})\rightarrow \Omega^{m,1}_c(M\setminus D,F|_{M\setminus D})$ unitarily equivalent to \eqref{lola}. Therefore all the statements of this theorem follow immediately by applying Remark \ref{bottchern}, Theorem \ref{firstth}, Theorem \ref{secondo} and Cor. \ref{heat} to \eqref{risonno}. Note that \eqref{risonno} obeys Remark \ref{mogimo}.
\end{proof}

\begin{cor}
\label{terzo}
In the setting of Theorem \ref{canonical}. We have the following properties.
\begin{enumerate}
\item $\im(\overline{\mathfrak{d}}_{E,m,0})$ is a closed subset of $L^2\Omega^{m,1}(\reg(X),E,h)$.
\item $\ker(\overline{\mathfrak{d}}_{E,m,0})$ is finite dimensional.
\item We have the following $L^2$-orthogonal decomposition: $$L^2\Omega^{m,0}(\reg(X),E,h)=\ker(\overline{\mathfrak{d}}_{E,m,0})\oplus \im(\overline{\mathfrak{d}}_{E,m,0}^*).$$
\item $\overline{\mathfrak{d}}_{E,m,0}^*\circ\overline{\mathfrak{d}}_{E,m,0}:L^2\Omega^{m,0}(\reg(X),E,h)\rightarrow L^2\Omega^{m,0}(\reg(X),E,h)$ is a Fredholm operator on its domain endowed with graph norm. 
\end{enumerate} 
\end{cor}

\begin{proof}
This follows applying Cor. \ref{freddy}.
\end{proof}
We conclude this section with the following remarks. All the results proved in this section hold in particular for $$\overline{\partial}_{E,m,0,\max/\min}:L^2\Omega^{m,0}(\reg(X),E,h)\rightarrow L^2\Omega^{m,1}(\reg(X),E,h).$$
Moreover  consider again the operator
\begin{equation}
\label{smodelzx}
\partial_{E,m,0}:L^2\Omega^{0,m}(\reg(X),E,h)\rightarrow L^2\Omega^{1,m}(\reg(X),E,h)
\end{equation}
with $\Omega^{0,m}_c(\reg(X),E)$ as  domain
and let us label by 
\begin{equation}
\label{sanyzx}
\mathfrak{d}_{E,m,0}:L^2\Omega^{0,m}(\reg(X),E,h)\rightarrow L^2\Omega^{1,m}(\reg(X),E,h)
\end{equation}
any closed extension of \eqref{smodelzx}. Then, according to  the remark  stated after the proof of Cor. \ref{heat}, we have also the  corresponding versions  of Theorem \ref{canonical} and  Cor. \ref{terzo}  for \eqref{sanyzx}.

\subsection{ Self-adjoint extensions with discrete spectrum  in the setting of isolated singularities}
In this subsection we prove the existence of self-adjoint extensions with discrete spectrum for the Hodge-Kodaira Laplacian in the framework of  compact and irreducible Hermitian complex spaces with isolated singularities. 
\begin{teo}
\label{pentecoste}
Let $(X,h)$ be a compact and irreducible Hermitian complex space of complex dimension $m$. Assume that $\sing(X)$ is made of isolated singularities. Then we have the following properties:
\begin{enumerate}
\item $\Delta_{\overline{\partial},m,q,\abs}:L^2\Omega^{m,q}(\reg(X),h)\rightarrow L^2\Omega^{m,q}(\reg(X),h)$ has discrete spectrum for each $q=0,...,m$.
\item $\overline{\partial}_{m,\max}+\overline{\partial}^t_{m,\min}:L^2\Omega^{m,\bullet}(\reg(X),h)\rightarrow L^2\Omega^{m,\bullet}(\reg(X),h)$ has discrete spectrum.
\item $\Delta_{\overline{\partial},m,q}^{\mathcal{F}}:L^2\Omega^{m,q}(\reg(X),h)\rightarrow L^2\Omega^{m,q}(\reg(X),h)$ has discrete spectrum for each $q=0,...,m$.
\end{enumerate}
\end{teo}

\begin{proof}
Consider the \emph{first point}. The case $(m,0)$ follows by Th. \ref{canonical}.  For the remaining cases we argue as follows.
According to \cite{MaMa} pag. 381  $\Delta_{\overline{\partial},m,q,\abs}:L^2\Omega^{m,q}(\reg(X),h)\rightarrow L^2\Omega^{m,q}(\reg(X),h)$ has discrete spectrum if and only if  the inclusion
\begin{equation}
\label{acacia}
\mathcal{D}(\Delta_{\overline{\partial},m,q,  \abs})\hookrightarrow L^2\Omega^{m,q}(\reg(X),h)
\end{equation}
 is a compact operator where $\mathcal{D}(\Delta_{\overline{\partial},m,q,\abs})$ is endowed with the corresponding graph norm. According to \cite{JRu}  we know that $H^{m,q}_{2,\overline{\partial}_{\max}}(\reg(X),h)$ is finite dimensional for each $q=0,...,m$. Therefore, using  Prop. \ref{usipeti}, we can conclude that, for each $q=0,...,m$, $\im(\overline{\partial}_{m,q,\max})$ is closed  and that $\Delta_{\overline{\partial},m,q,  \abs}:L^2\Omega^{m,q}(\reg(X),h)\rightarrow L^2\Omega^{m,q}(\reg(X),h)$ is a Fredholm operator on its domain endowed with the graph norm. Hence, by the fact that  $\Delta_{\overline{\partial},m,q, \abs}:L^2\Omega^{m,q}(\reg(X),h)\rightarrow L^2\Omega^{m,q}(\reg(X),h)$ is Fredholm and self-adjoint, we know now that \eqref{acacia} is a compact operator if and only if the following inclusion is a compact operator  
\begin{equation}
\label{cacio}
\left(\mathcal{D}(\Delta_{\overline{\partial},m,q, \abs})\cap \im(\Delta_{\overline{\partial},m,q,\abs})\right)\hookrightarrow L^2\Omega^{m,q}(\reg(X),h)
\end{equation}
 where  $\left(\mathcal{D}(\Delta_{\overline{\partial},m,q,\abs})\cap \im(\Delta_{\overline{\partial},m,q,\abs})\right)$ is endowed with the  graph norm of $\Delta_{\overline{\partial},m,q, \abs}$. Finally, by using \cite{Ruppe} Th. 1.1 or \cite{OvRu} Th. 1.2, we get that \eqref{cacio} is a compact inclusion for $q\geq 1$ and this completes the proof of the first point. Now we tackle the \emph{second point}. Consider the operator 
\begin{equation}
\label{frengogoh}
(\overline{\partial}_{m,\max}+\overline{\partial}^t_{m,\min})\circ (\overline{\partial}_{m,\max}+\overline{\partial}^t_{m,\min}):L^2\Omega^{m,\bullet}(\reg(X),h)\rightarrow L^2\Omega^{m,\bullet}(\reg(X),h)
\end{equation}
with domain given by 
\begin{align}
\nn &\mathcal{D}((\overline{\partial}_{m,\max}+\overline{\partial}^t_{m,\min})\circ (\overline{\partial}_{m,\max}+\overline{\partial}^t_{m,\min}))=\{\omega\in \mathcal{D}(\overline{\partial}_{m,\max}+\overline{\partial}^t_{m,\min})\ \text{such that}\\ \nn &(\overline{\partial}_{m,\max}+\overline{\partial}^t_{m,\min})\omega\in \mathcal{D}(\overline{\partial}_{m,\max}+\overline{\partial}^t_{m,\min}) \}.
\end{align}
We have $$(\overline{\partial}_{m,\max}+\overline{\partial}^t_{m,\min})\circ (\overline{\partial}_{m,\max}+\overline{\partial}^t_{m,\min})=\bigoplus_{q=0}^m\Delta_{\overline{\partial},m,q,\abs}$$ where the domain of the operator on the right hand side is $\bigoplus_{q=0}^m\mathcal{D}(\Delta_{\overline{\partial},m,q,\abs})$. By the first point of this theorem we can thus conclude that \eqref{frengogoh} has discrete spectrum and eventually this implies that 
$$\overline{\partial}_{m,\max}+\overline{\partial}^t_{m,\min}:L^2\Omega^{m,\bullet}(\reg(X),h)\rightarrow L^2\Omega^{m,\bullet}(\reg(X),h)$$  has discrete spectrum. Now we deal with the \emph{third point}.  Consider the operator $$(\overline{\partial}_m+\overline{\partial}^t_m)_{\min}:L^2\Omega^{m,\bullet}(\reg(X),h)\rightarrow L^2\Omega^{m,\bullet}(\reg(X),h)$$ that is  the minimal extension of $\overline{\partial}_{m}+\overline{\partial}^t_{m}:\Omega^{m,\bullet}_c(\reg(X))\rightarrow \Omega^{m,\bullet}_c(\reg(X))$. Clearly $\overline{\partial}_{m,\max}+\overline{\partial}^t_{m,\min}$ is a closed extension of $(\overline{\partial}_{m}+\overline{\partial}^t_{m})_{\min}$ and therefore, using the second point, we get that the inclusion $\mathcal{D}((\overline{\partial}_{m}+\overline{\partial}^t_{m})_{\min})\hookrightarrow L^2\Omega^{m,\bullet}(\reg(X),h)$ is a compact operator where $\mathcal{D}((\overline{\partial}_{m}+\overline{\partial}^t_{m})_{\min})$ is endowed with the corresponding graph norm. Let $(\overline{\partial}_{m}+\overline{\partial}^t_{m})_{\max}:L^2\Omega^{m,\bullet}(\reg(X),h)\rightarrow L^2\Omega^{m,\bullet}(\reg(X),h)$ be the maximal extension of $\overline{\partial}_{m}+\overline{\partial}^t_{m}$. Consider now the operator $$(\overline{\partial}_{m}+\overline{\partial}^t_{m})_{\max}\circ (\overline{\partial}_{m}+\overline{\partial}^t_{m})_{\min}:L^2\Omega^{m,\bullet}(\reg(X),h)\rightarrow L^2\Omega^{m,\bullet}(\reg(X),h)$$ with domain given by $$\mathcal{D}((\overline{\partial}_{m}+\overline{\partial}^t_{m})_{\max}\circ (\overline{\partial}_{m}+\overline{\partial}^t_{m})_{\min}):=\{\eta\in \mathcal{D}((\overline{\partial}_{m}+\overline{\partial}^t_{m})_{\min}): (\overline{\partial}_{m}+\overline{\partial}^t_{m})_{\min}\eta\in \mathcal{D}((\overline{\partial}_{m}+\overline{\partial}^t_{m})_{\max})\}.$$ By Prop. \ref{fall} we know that $(\overline{\partial}_{m}+\overline{\partial}^t_{m})_{\max}\circ (\overline{\partial}_{m}+\overline{\partial}^t_{m})_{\min}$ is the Friedrich extension of $(\overline{\partial}_{m}+\overline{\partial}^t_{m})\circ (\overline{\partial}_{m}+\overline{\partial}^t_{m})$ which in turn coincides with the direct sum $\bigoplus_{q=0}^m\Delta_{\overline{\partial},m,q}$. For each $\eta\in \mathcal{D}\left((\overline{\partial}_{m}+\overline{\partial}^t_{m})_{\max}\circ (\overline{\partial}_{m}+\overline{\partial}^t_{m})_{\min}\right)$ we have  the following inequality which is immediate to check 
\begin{equation}
\label{patrikqw}
\|(\overline{\partial}_{m}+\overline{\partial}^t_{m})_{\min}\eta\|^2_{L^2\Omega^{m,\bullet}(\reg(X),h)}\leq \frac{1}{2}\left(\|\eta\|^2_{L^2\Omega^{m,\bullet}(\reg(X),h)}+\|(\overline{\partial}_{m}+\overline{\partial}^t_{m})_{\max}\circ (\overline{\partial}_{m}+\overline{\partial}^t_{m})_{\min}\eta\|^2_{L^2\Omega^{m,\bullet}(\reg(X),h)}\right).
\end{equation}
The above inequality implies that we have a continuous inclusion 
\begin{equation}
\label{patrikdcqa}
\mathcal{D}\left((\overline{\partial}_{m}+\overline{\partial}^t_{m})_{\max}\circ (\overline{\partial}_{m}+\overline{\partial}^t_{m})_{\min}\right)\hookrightarrow \mathcal{D}\left((\overline{\partial}_{m}+\overline{\partial}^t_{m})_{\min}\right)
\end{equation}
 where each domain is endowed with the corresponding graph norm.
Therefore, using \eqref{patrikdcqa}, we have eventually shown  that the inclusion $$\mathcal{D}\left((\overline{\partial}_{m}+\overline{\partial}^t_{m})_{\max}\circ (\overline{\partial}_{m}+\overline{\partial}^t_{m})_{\min}\right)\hookrightarrow L^2\Omega^{m,\bullet}(\reg(X),h)$$ where $\mathcal{D}((\overline{\partial}_{m}+\overline{\partial}^t_{m})_{\max}\circ (\overline{\partial}_{m}+\overline{\partial}^t_{m})_{\min})$ is endowed with its graph norm,  is a compact operator. As remarked above this in turn implies that $(\overline{\partial}_{m}+\overline{\partial}^t_{m})_{\max}\circ (\overline{\partial}_{m}+\overline{\partial}^t_{m})_{\min}:L^2\Omega^{m,\bullet}(\reg(X),h)\rightarrow L^2\Omega^{m,\bullet}(\reg(X),h)$ has discrete spectrum. Finally, by the fact that $$(\overline{\partial}_{m}+\overline{\partial}^t_{m})_{\max}\circ (\overline{\partial}_{m}+\overline{\partial}^t_{m})_{\min}=\bigoplus_{q=0}^m\Delta_{\overline{\partial},m,q}^{\mathcal{F}}$$ see for instance \cite{FB} pag.  169, we get that, for each $q=0,...,m$, the operator $$\Delta_{\overline{\partial},m,q}^{\mathcal{F}}:L^2\Omega^{m,\bullet}(\reg(X),h)\rightarrow L^2\Omega^{m,\bullet}(\reg(X),h)$$  has discrete spectrum as desired. The proof of the third point is thus complete. 
\end{proof}

We conclude this subsection with the following corollary.

\begin{cor}
\label{aquino}
In the setting of Th. \ref{pentecoste}. We have the following properties:
\begin{enumerate}
\item $\Delta_{\overline{\partial},0,q,\rel}:L^2\Omega^{0,q}(\reg(X),h)\rightarrow L^2\Omega^{0,q}(\reg(X),h)$ has discrete spectrum for each $q=0,...,m$.
\item $\overline{\partial}_{0,\min}+\overline{\partial}^t_{0,\max}:L^2\Omega^{0,\bullet}(\reg(X),h)\rightarrow L^2\Omega^{0,\bullet}(\reg(X),h)$ has discrete spectrum.
\item $\Delta_{\overline{\partial},0,q}^{\mathcal{F}}:L^2\Omega^{0,q}(\reg(X),h)\rightarrow L^2\Omega^{0,q}(\reg(X),h)$ has discrete spectrum for each $q=0,...,m$.
\end{enumerate}
\end{cor}

\begin{proof}
It is enough to prove the first point. The second and the third point follow by the first one arguing as in the proof of Th. \ref{pentecoste}. Using \eqref{cicuta} and Prop. \ref{occhiodibue} we have that any form $\omega\in L^2\Omega^{m,q}(\reg(V),h)$ lies in $\mathcal{D}(\Delta_{\overline{\partial},m,q,\abs})$ if and only if $c_{m-q,0}(*\omega)\in \mathcal{D}(\Delta_{\overline{\partial},0,m-q,\rel})$ and if this is the case then we have $c_{m-q,0}(*(\Delta_{\overline{\partial},m,q,\abs}\omega))=\Delta_{\overline{\partial},0,m-q,\rel}(c_{m-q,0}(*\omega))$, see Prop. \ref{occhiodibue}. Since    $c_{m-q,0}\circ *:L^2\Omega^{m,q}(\reg(V),h)\rightarrow L^2\Omega^{0,m-q}(\reg(V),g)$ is a continuous and  bijective $\mathbb{C}$-antilinear isomorphism with continuous inverse the conclusion  follows now by Th. \ref{pentecoste}.
\end{proof}

\subsection{The Hodge-Kodaira Laplacian on  complex projective surfaces}
In this section we collect various applications to the Hodge-Kodaira Laplacian on   complex projective surfaces. We start with the following theorem.
\begin{teo}
\label{lillottina}
Let $V\subset \mathbb{C}\mathbb{P}^n$ be a complex projective surface. Let $h$ be the K\"ahler metric on $\reg(V)$ induced by the Fubini-Study metric of $\mathbb{C}\mathbb{P}^n$. We have the following properties:
\begin{enumerate}
\item $\Delta_{\overline{\partial},2,q,\abs}:L^2\Omega^{2,q}(\reg(V),h)\rightarrow L^2\Omega^{2,q}(\reg(V),h)$ has discrete spectrum for each $q=0,1,2$.
\item $\overline{\partial}_{2,\max}+\overline{\partial}^t_{2,\min}:L^2\Omega^{2,\bullet}(\reg(V),h)\rightarrow L^2\Omega^{2,\bullet}(\reg(V),h)$ has discrete spectrum.
\item $\Delta_{\overline{\partial},2,q}^{\mathcal{F}}:L^2\Omega^{2,q}(\reg(V),h)\rightarrow L^2\Omega^{2,q}(\reg(V),h)$ has discrete spectrum for each $q=0,1,2$.
\end{enumerate}
\end{teo}

\begin{proof}
We start by  considering  the operator  $\Delta_{\overline{\partial},2,0,\abs}:L^2\Omega^{2,0}(\reg(V),h)\rightarrow L^2\Omega^{2,0}(\reg(V),h)$. In this case the statement is  a particular case of  Th. \ref{canonical}. Now we deal with $\Delta_{\overline{\partial},2,2,\abs}:L^2\Omega^{2,2}(\reg(V),h)\rightarrow L^2\Omega^{2,2}(\reg(V),h)$. We observe that in this case $\Delta_{\overline{\partial},2,2,\abs}= \overline{\partial}_{2,1,\max}\circ \overline{\partial}_{2,1,\min}^t$. Applying the Hodge star operator $*:L^2\Omega^{2,2}(\reg(V),h)\rightarrow L^2(\reg(V),g)$ we have  $*(\mathcal{D}(\overline{\partial}_{2,1,\max}\circ \overline{\partial}_{2,1,\min}^t))=\mathcal{D}(\partial_{\max}^t\circ \partial_{\min})$ and  $$*\circ (\overline{\partial}_{2,1,\max}\circ \overline{\partial}_{2,1,\min}^t)=(\partial_{\max}^t\circ \partial_{\min})\circ *.$$  We are therefore left to prove that $\partial_{\max}^t\circ \partial_{\min}:L^2(\reg(V),h)\rightarrow L^2(\reg(V),g)$ has discrete spectrum. This is shown as follows. According to Prop. \ref{fall} we know that $\partial_{\max}^t\circ \partial_{\min}=\Delta^{\mathcal{F}}_{\partial}$, the Friedrich extension of $\Delta_{\partial}:C^{\infty}_c(\reg(V))\rightarrow C^{\infty}_c(\reg(V))$. On the other hand $(\reg(V),h)$ is a K\"ahler manifold. Therefore we have $\Delta_{\partial}=\Delta_{\overline{\partial}}$ on $C^{\infty}_c(\reg(V))$ and hence we can conclude that the corresponding Friedrich extensions, as operators acting on $L^2(\reg(V),h)$, coincide:
\begin{equation}
\label{nonso}
\Delta^{\mathcal{F}}_{\partial}=\Delta^{\mathcal{F}}_{\overline{\partial}}.
\end{equation}
Now, according to \cite{LT} we know that the right hand side of \eqref{nonso} has discrete spectrum. We can thus  conclude that also $\Delta^{\mathcal{F}}_{\partial}$  has discrete spectrum   and ultimately we have that $\Delta_{\overline{\partial},2,2,\abs}:L^2\Omega^{2,2}(\reg(V),h)\rightarrow L^2\Omega^{2,2}(\reg(V),h)$ has discrete spectrum as desired.  
As last step we are left   to prove  that $\Delta_{\overline{\partial},2,1,\abs}:L^2\Omega^{2,1}(\reg(V),h)\rightarrow L^2\Omega^{2,1}(\reg(V),h)$ has discrete spectrum. As we have already seen, this is equivalent to showing that the inclusion
\begin{equation}
\label{compactww}
\mathcal{D}(\Delta_{\overline{\partial},2,1,  \abs})\hookrightarrow L^2\Omega^{2,1}(\reg(V),h)
\end{equation}
 is a compact operator where $\mathcal{D}(\Delta_{\overline{\partial},2,1,\abs})$ is endowed with the corresponding graph norm. According to \cite{PS}  we know that $H^{2,q}_{2,\overline{\partial}_{\max}}(\reg(V),h)$ is finite dimensional for each $q$. Therefore, using  Prop. \ref{usipeti}, we can conclude that, for each $q$, $\im(\overline{\partial}_{2,q,\max})$ is closed  and that $\Delta_{\overline{\partial},2,q,  \abs}:L^2\Omega^{2,q}(\reg(V),h)\rightarrow L^2\Omega^{2,q}(\reg(V),h)$ is a Fredholm operator on its domain endowed with the graph norm. Hence, by the fact that  $\Delta_{\overline{\partial},2,1, \abs}:L^2\Omega^{2,1}(\reg(V),h)\rightarrow L^2\Omega^{2,1}(\reg(V),h)$ is Fredholm and self-adjoint, we know now that \eqref{compactww} is a compact operator if and only if the following inclusion is a compact operator  
\begin{equation}
\label{scompactIIIq}
\left(\mathcal{D}(\Delta_{\overline{\partial},2,1, \abs})\cap \im(\Delta_{\overline{\partial},2,1,\abs})\right)\hookrightarrow L^2\Omega^{2,1}(\reg(V),h)
\end{equation}
 where  $\left(\mathcal{D}(\Delta_{\overline{\partial},2,1,\abs})\cap \im(\Delta_{\overline{\partial},2,1,\abs})\right)$ is endowed with the  graph norm of $\Delta_{\overline{\partial},2,1, \abs}$. Since we have already seen  that both the operators  $\Delta_{\overline{\partial},2,0,  \abs}:L^2\Omega^{2,0}(\reg(V),h)\rightarrow L^2\Omega^{2,0}(\reg(V),h)$  and $\Delta_{\overline{\partial},2,2,  \abs}:L^2\Omega^{2,2}(\reg(V),h)\rightarrow L^2\Omega^{2,2}(\reg(V),h)$  have discrete spectrum,  we know in particular that both the inclusions $\mathcal{D}(\Delta_{\overline{\partial},2,0,\abs}) \hookrightarrow L^2\Omega^{2,0}(\reg(V),h)$ and $\mathcal{D}(\Delta_{\overline{\partial},2,2,\abs})\hookrightarrow L^2\Omega^{2,2}(\reg(V),h)$ are  compact operators. In particular we get  that  the following inclusions
\begin{equation}
\label{compactIV}
\left(\mathcal{D}(\overline{\partial}_{2,1,\max}\circ \overline{\partial}_{2,1,\min}^t)\cap \im(\overline{\partial}_{2,1,\max}\circ \overline{\partial}_{2,1,\min}^t)\right)\hookrightarrow L^2\Omega^{2,2}(\reg(V),h)
\end{equation}
\begin{equation}
\label{compactV}
\left(\mathcal{D}(\overline{\partial}_{2,0,\min}^t\circ \overline{\partial}_{2,0,\max})\cap \im(\overline{\partial}_{2,0,\min}^t\circ \overline{\partial}_{2,0,\max})\right)\hookrightarrow L^2\Omega^{2,0}(\reg(V),h)
\end{equation}
are both compact operators where each space is endowed with the corresponding graph norm. Therefore we are now in the position to use   Corollary \ref{calanchi} in order  to  conclude  that \eqref{scompactIIIq} is a compact operator. This completes the proof of the first point. 
Finally the second and the third point follow by using the same arguments used to show the second and the third point  of Th. \ref{pentecoste}. The proof is thus complete.
\end{proof}

\begin{teo}
\label{coccabelladezio}
In the  setting of Th. \ref{lillottina}. Let $q\in \{0,1,2\}$ and consider the operator
\begin{equation}
\label{resile}
\Delta_{\overline{\partial},2,q,\abs}:L^2\Omega^{2,q}(\reg(V),h)\rightarrow L^2\Omega^{2,q}(\reg(V),h).
\end{equation}
Let $$0\leq \lambda_1\leq \lambda_2\leq...\leq \lambda_k\leq...$$ be the eigenvalues of \eqref{resile}. Then we have the following asymptotic inequality
\begin{equation}
\label{resiles}
\lim \inf \lambda_k k^{-\frac{1}{2}}>0
\end{equation}
as $k\rightarrow \infty$.\\ Consider now the heat operator associated to \eqref{resile}
\begin{equation}
\label{resilex}
e^{-t\Delta_{\overline{\partial},2,q,\abs}}:L^2\Omega^{2,q}(\reg(V),h)\rightarrow L^2\Omega^{2,q}(\reg(V),h).
\end{equation}
Then \eqref{resilex} is a trace class operator and its trace satisfies the following estimates
\begin{equation}
\label{resilez}
\Tr(e^{-t\Delta_{\overline{\partial},2,q,\abs}})\leq C_qt^{-2}
\end{equation}
for $t\in (0,1]$ and some constant $C_q>0$.
\end{teo}

\begin{proof}
Let $q=0$. Then in this case the statement follows by Th. \ref{canonical}. Consider now the case $q=2$. Then, as pointed out in the proof of Th. \ref{lillottina}, we have  $*\Delta_{\overline{\partial},2,2, \abs}*=\partial_{\max}^t\circ \partial_{\min}=\Delta^{\mathcal{F}}_{\partial}=\Delta^{\mathcal{F}}_{\overline{\partial}}.$
Now the statement follows using the results proved  for $\Delta^{\mathcal{F}}_{\overline{\partial}}:L^2(\reg(V),h)\rightarrow L^2(\reg(V),h)$ in \cite{LT}. Finally we deal with the case $q=1$. Consider the operator
\begin{equation}
\label{isabella}
\overline{\partial}_{2,0,\max}+\overline{\partial}^t_{2,1,\min}:L^2\Omega^{2,0}(\reg(V),h)\oplus L^2\Omega^{2,2,}(\reg(V),h)\rightarrow L^2\Omega^{2,1}(\reg(V),h)
\end{equation}
 whose domain is $\mathcal{D}(\overline{\partial}_{2,0,\max})\oplus \mathcal{D}(\overline{\partial}_{2,1,\min}^t)\subset L^2\Omega^{2,0}(\reg(V),h)\oplus L^2\Omega^{2,2,}(\reg(V),h)$. Its adjoint is 
\begin{equation}
\label{rege}
\overline{\partial}_{2,1,\max}+\overline{\partial}^t_{2,0,\min}:L^2\Omega^{2,1}(\reg(V),h)\rightarrow L^2\Omega^{2,0}(\reg(V),h) \oplus L^2\Omega^{2,2,}(\reg(V),h)
\end{equation}
with domain given by   $\mathcal{D}(\overline{\partial}_{2,1,\max})\cap \mathcal{D}(\overline{\partial}_{2,0,\min}^t)\subset L^2\Omega^{2,1}(\reg(V),h)$.
Taking the composition of each operator with the corresponding adjoint we get   
\begin{align}
\label{isabella2}
& (\overline{\partial}_{2,0,\max}+\overline{\partial}^t_{2,1,\min})^*\circ (\overline{\partial}_{2,0,\max}+\overline{\partial}^t_{2,1,\min})=\\
\nonumber &\Delta_{\overline{\partial},2,0,\abs}\oplus\Delta_{\overline{\partial},2,2,\abs} :L^2\Omega^{2,0}(\reg(V),h)\oplus L^2\Omega^{2,2,}(\reg(V),h)\rightarrow L^2\Omega^{2,0}(\reg(V),h)\oplus L^2\Omega^{2,2,}(\reg(V),h)
\end{align}
and 
\begin{equation}
\label{rege2}
(\overline{\partial}_{2,1,\max}+\overline{\partial}^t_{2,0,\min})^*\circ (\overline{\partial}_{2,1,\max}+\overline{\partial}^t_{2,0,\min})=\Delta_{\overline{\partial},2,1,\abs}:L^2\Omega^{2,1}(\reg(V),h)\rightarrow L^2\Omega^{2,1}(\reg(V),h).
\end{equation}
By Prop. \ref{same} we know that a real number $\lambda>0$ is an eigenvalue for \eqref{rege2} if and only if is an eigenvalue for \eqref{isabella2} and, if this is the case, the corresponding egeinspaces have the same dimension. Hence, by the fact that \eqref{isabella2} is the direct sum of $\Delta_{\overline{\partial},2,0,\abs}:L^2\Omega^{2,0}(\reg(V),h)\rightarrow L^2\Omega^{2,0}(\reg(V),h)$ and 
$\Delta_{\overline{\partial},2,2,\abs} : L^2\Omega^{2,2,}(\reg(V),h)\rightarrow L^2\Omega^{2,2,}(\reg(V),h)$ and by the fact that we have already shown that the asymptotic inequality \eqref{resiles} holds true for $\Delta_{\overline{\partial},2,0,\abs}$ and $\Delta_{\overline{\partial},2,2,\abs}$ we are in the position to   conclude that  \eqref{resiles} holds true also for the eigenvalues of \eqref{rege2}. We can also conclude immediately that $e^{-t\Delta_{\overline{\partial},2,1,\abs}}:L^2\Omega^{2,1}(\reg(V),h)\rightarrow L^2\Omega^{2,1}(\reg(V),h)$ is a trace class operator because, thanks to \eqref{resiles}, we have $$\Tr(e^{-t\Delta_{\overline{\partial},2,1,\abs}})=\sum_ke^{-t\lambda_k}<\infty$$  where in the above formula $\{\lambda_k\}_{k\in \mathbb{N}}$ are the eigenvalues of \eqref{rege2}. Finally \eqref{resilez} follows observing that, again by Prop. \eqref{same}, we have 
\begin{equation}
\label{pescesega}
\Tr(e^{-t\Delta_{\overline{\partial},2,1,\abs}})-\ker(\Delta_{\overline{\partial},2,1,\abs})=\Tr(e^{-t\Delta_{\overline{\partial},2,0,\abs}})-\ker(\Delta_{\overline{\partial},2,0,\abs})+\Tr(e^{-t\Delta_{\overline{\partial},2,2,\abs}})-\ker(\Delta_{\overline{\partial},2,2,\abs})
\end{equation}
 and therefore for $t\in (0,1]$ we have $$\Tr(e^{-t\Delta_{\overline{\partial},2,1,\abs}})\leq \ker(\Delta_{\overline{\partial},2,1,\abs})+C_0t^{-2}-\ker(\Delta_{\overline{\partial},2,0,\abs})+C_2t^{-2}-\ker(\Delta_{\overline{\partial},2,2,\abs})\leq C_1t^{-2}$$ for some $C_1>0$. 
\end{proof}

As a consequence of the previous theorem we recover the McKean-Singer formula on complex projective surfaces concerning the complex $(L^2\Omega^{2,q}(\reg(V),h),\overline{\partial}_{2,q,\max})$.

\begin{cor}
\label{kean}
In the setting of Th. \ref{lillottina}. Let us label by $(\overline{\partial}_{2,\max}+\overline{\partial}_{2,\min}^t)^+$ the operator defined in \eqref{isabella}. Then $(\overline{\partial}_{2,\max}+\overline{\partial}_{2,\min}^t)^+$ is a Fredholm operator and its index satisfies
\begin{equation}
\label{jilm}
\ind((\overline{\partial}_{2,\max}+\overline{\partial}_{2,\min}^t)^+)=\sum_{q=0}^2(-1)^q\Tr(e^{-t\Delta_{\overline{\partial},2,q,\abs}}).
\end{equation}
In particular we have 
\begin{equation}
\label{jilmx}
\chi(\tilde{V},\mathcal{K}_{\tilde{V}})=\sum_{q=0}^2(-1)^q\Tr(e^{-t\Delta_{\overline{\partial},2,q,\abs}})
\end{equation}
where $\pi:\tilde{V}\rightarrow V$ is any resolution of $V$, $\mathcal{K}_{\tilde{V}}$ is the sheaf of holomorphic $(2,0)$-forms on $\tilde{V}$ and $\chi(\tilde{V},\mathcal{K}_{\tilde{V}})=\sum_{q=0}^2(-1)^q\dim(H^q(\tilde{V},\mathcal{K}_{\tilde{V}}))$.
\end{cor} 

\begin{proof}
That $(\overline{\partial}_{2,\max}+\overline{\partial}_{2,\min}^t)^+$ is a Fredholm operator it is clear from Th. \ref{lillottina}. The equality \eqref{jilm} follows by \eqref{pescesega}. Indeed we have $$\ind((\overline{\partial}_{2,\max}+\overline{\partial}_{2,\min}^t)^+)=\sum_{q=0}^2(-1)^q\ker(\Delta_{\overline{\partial},2,q,\abs})=\sum_{q=0}^2(-1)^q\Tr(e^{-t\Delta_{\overline{\partial},2,q,\abs}}).$$ The equality \eqref{jilmx} follows by \eqref{jilm} and the results established in \cite{PS}.
\end{proof}

\begin{teo}
\label{coccabelladezios}
In the  setting of Th. \ref{lillottina}. Let $q\in \{0,1,2\}$ and consider the operator
\begin{equation}
\label{resileq}
\Delta_{\overline{\partial},2,q,}^{\mathcal{F}}:L^2\Omega^{2,q}(\reg(V),h)\rightarrow L^2\Omega^{2,q}(\reg(V),h)
\end{equation}
that is the Friedrich extension of $\Delta_{\overline{\partial},2,q}:\Omega^{2,q}_c(\reg(V))\rightarrow \Omega^{2,q}_c(\reg(V))$.
Let  $$0\leq \mu_1\leq \mu_2\leq...\leq \mu_k\leq...$$ be the eigenvalues of \eqref{resileq} and 
let $$0\leq \lambda_1\leq \lambda_2\leq...\leq \lambda_k\leq...$$ be the eigenvalues of \eqref{resile}. Then we have the following  inequality for every $k\in \mathbb{N}$
\begin{equation}
\label{compare}
\lambda_{k}\leq \mu_k.
\end{equation}
In particular we have 
\begin{equation}
\label{resileh}
\lim \inf \mu_k k^{-\frac{1}{2}}>0
\end{equation}
as $k\rightarrow \infty$.\\ Consider now the heat operator associated to \eqref{resileq}
\begin{equation}
\label{resiley}
e^{-t\Delta_{\overline{\partial},2,q}^{\mathcal{F}}}:L^2\Omega^{2,q}(\reg(V),h)\rightarrow L^2\Omega^{2,q}(\reg(V),h).
\end{equation}
Then \eqref{resiley} is a trace class operator and 
\begin{equation}
\label{pololo}
\Tr(e^{-t\Delta_{\overline{\partial},2,q}^{\mathcal{F}}})\leq \Tr(e^{-t\Delta_{\overline{\partial},2,q,\abs}}).
\end{equation}
In particular we have the following estimate for  $\Tr(e^{-t\Delta_{\overline{\partial},2,q}^{\mathcal{F}}})$
\begin{equation}
\label{resilezk}
\Tr(e^{-t\Delta_{\overline{\partial},2,q}^{\mathcal{F}}})\leq  B_qt^{-2}
\end{equation}
for $t\in (0,1]$ and some constant $B_q>0$.
\end{teo}

\begin{proof}
Using again the min-max Theorem  as in the proof of Th. \ref{secondo} we have 
\begin{equation}
\label{dovizia}
\mu_k=\inf_{F\in \mathfrak{F}_k\cap\mathcal{D}(\Delta_{\overline{\partial},2,q}^{\mathcal{F}})}\sup_{s\in F}\frac{\langle\Delta_{\overline{\partial},2,q}^{\mathcal{F}}s,s\rangle_{L^2\Omega^{2,q}(\reg(V),h)}}{\|s\|^2_{L^2\Omega^{2,q}(\reg(V),h)}}
\end{equation}
where $\mathfrak{F}_k$ denotes the set of linear subspaces of $L^2\Omega^{2,q}(\reg(V),h)$ of dimension at most $k$. Analogously for the eigenvalues of \eqref{resile} we have 
\begin{equation}
\label{minuzia}
\lambda_k=\inf_{F\in \mathfrak{F}_k\cap\mathcal{D}(\Delta_{\overline{\partial},2,q,\abs})}\sup_{s\in F}\frac{\langle \Delta_{\overline{\partial},2,q,\abs}s,s\rangle_{L^2\Omega^{2,q}(\reg(V),h)}}{\|s\|^2_{L^2\Omega^{2,q}(\reg(V),h)}}.
\end{equation}
By Prop. \ref{fall} we know that $\Delta^{\mathcal{F}}_{\overline{\partial},2,q}=(\overline{\partial}_{2,q}+\overline{\partial}^t_{2,q-1})_{\max}\circ(\overline{\partial}_{2,q}+\overline{\partial}^t_{2,q-1})_{\min}$ and by \eqref{asdf} we know that $\Delta_{\overline{\partial},2,q,\abs}=(\overline{\partial}_{2,q,\max}+\overline{\partial}^t_{2,q-1,\min})\circ(\overline{\partial}_{2,q,\max}+\overline{\partial}^t_{2,q-1,\min})$. Therefore \eqref{dovizia} and \eqref{minuzia} become respectively
\begin{equation}
\label{arguzia}
\inf_{F\in \mathfrak{F}_k\cap\mathcal{D}((\overline{\partial}_{2,q}+\overline{\partial}^t_{2,q-1})_{\min})}\sup_{s\in F}\frac{\langle (\overline{\partial}_{2,q}+\overline{\partial}^t_{2,q-1})_{\min}s,(\overline{\partial}_{2,q}+\overline{\partial}^t_{2,q-1})_{\min}s\rangle_{L^2\Omega^{2,\bullet}(\reg(V),h)}}{\|s\|^2_{L^2\Omega^{2,q}(\reg(V),h)}}
\end{equation}

and

\begin{equation}
\label{letizia}
\inf_{F\in \mathfrak{F}_k\cap\mathcal{D}(\overline{\partial}_{2,q,\max}+\overline{\partial}^t_{2,q-1,\min})}\sup_{s\in F}\frac{\langle (\overline{\partial}_{2,q,\max}+\overline{\partial}^t_{2,q-1,\min})s,(\overline{\partial}_{2,q,\max}+\overline{\partial}^t_{2,q-1,\min})s\rangle_{L^2\Omega^{2,\bullet}(\reg(V),h)}}{\|s\|^2_{L^2\Omega^{2,q}(\reg(V),h)}}
\end{equation}
where $L^2\Omega^{2,\bullet}(\reg(V),h)=\oplus_{q=0}^2 L^2\Omega^{2,q}(\reg(V),h)$. Let now $\{\phi_n,\ n\in \mathbb{N}\}$ be an orthonormal basis of $L^2\Omega^{2,q}(\reg(V),h)$ made of eigensections of $\Delta_{\overline{\partial},2,q}^{\mathcal{F}}$ such that $\Delta^{\mathcal{F}}_{\overline{\partial},2,q}\phi_k=\mu_k\phi_k$. Let us define  $F_k\in \mathfrak{F}_k$ as the $k$-dimensional subspace of $L^2\Omega^{2,q}(\reg(V),h)$ generated by $\{\phi_1,...,\phi_k\}$. Then, see for instance \cite{KSC} pag. 279, we have $$\mu_k=\sup_{s\in F_k}\frac{\langle (\overline{\partial}_{2,q}+\overline{\partial}_{2,q-1}^t)_{\min}s,(\overline{\partial}_{2,q}+\overline{\partial}_{2,q-1}^t)_{\min}s\rangle_{L^2\Omega^{2,\bullet}(\reg(V),h)}}{\|s\|^2_{L^2\Omega^{2,q}(\reg(V),h)}}.$$
Since $\mathcal{D}((\overline{\partial}_{2,q}+\overline{\partial}_{2,q-1}^t)_{\min})\subset \mathcal{D}(\overline{\partial}_{2,q,\max}+\overline{\partial}_{2,q-1,\min}^t)$  we can deduce  that 
\begin{align}
\nonumber  \mu_k &= \sup_{s\in F_k}\frac{\langle (\overline{\partial}_{2,q}+\overline{\partial}_{2,q-1}^t)_{\min}s,(\overline{\partial}_{2,q}+\overline{\partial}_{2,q-1}^t)_{\min}s\rangle_{L^2\Omega^{2,\bullet}(\reg(V),h)}}{\|s\|^2_{L^2\Omega^{2,q}(\reg(V),h)}}\\
\nonumber & =\sup_{s\in F_k}\frac{\langle (\overline{\partial}_{2,q,\max}+\overline{\partial}_{2,q-1,\min}^t)s,(\overline{\partial}_{2,q,\max}+\overline{\partial}_{2,q-1,\min}^t)s\rangle_{L^2\Omega^{2,\bullet}(\reg(V),h)}}{\|s\|^2_{L^2\Omega^{2,q}(\reg(V),h)}}\\
\nonumber & \geq \inf_{F\in \mathfrak{F}_k\cap\mathcal{D}(\overline{\partial}_{2,q,\max}+\overline{\partial}^t_{2,q-1,\min})}\sup_{s\in F}\frac{\langle (\overline{\partial}_{2,q,\max}+\overline{\partial}^t_{2,q-1,\min})s,(\overline{\partial}_{2,q,\max}+\overline{\partial}^t_{2,q-1,\min})s\rangle_{L^2\Omega^{2,\bullet}(\reg(V),h)}}{\|s\|^2_{L^2\Omega^{2,q}(\reg(V),h)}}=\lambda_k.
\end{align}
This establishes \eqref{compare}. The remaining properties follow now immediately  using \eqref{compare} and Th. \ref{coccabelladezio}.
\end{proof}

Concerning the bidegree $(1,0)$ we have the following application. 

\begin{teo}
\label{supercoccabelladezio}
In the  setting of Th. \ref{lillottina}. Consider the operator
\begin{equation}
\label{sacripante}
\Delta_{\overline{\partial},1,0}^{\mathcal{F}}:L^2\Omega^{1,0}(\reg(V),h)\rightarrow L^2\Omega^{1,0}(\reg(V),h)
\end{equation}
that is the Friedrich extension of $\Delta_{\overline{\partial},1,0}:\Omega^{1,0}_c(\reg(V))\rightarrow \Omega^{1,0}_c(\reg(V))$.
Then \eqref{sacripante} has discrete spectrum. Let  $$0\leq \mu_1\leq \mu_2\leq...\leq \mu_k\leq...$$ be the eigenvalues of \eqref{sacripante}. We have the following asymptotic  inequality 
\begin{equation}
\label{zucca}
\lim \inf \mu_k k^{-\frac{1}{2}}>0
\end{equation}
as $k\rightarrow \infty$.\\ Finally consider  the heat operator associated to \eqref{sacripante}
\begin{equation}
\label{gargantuesco}
e^{-t\Delta_{\overline{\partial},1,0}^{\mathcal{F}}}:L^2\Omega^{1,0}(\reg(V),h)\rightarrow L^2\Omega^{1,0}(\reg(V),h).
\end{equation}
Then \eqref{gargantuesco} is a trace class operator and its trace satisfies the following estimate
\begin{equation}
\label{pantagruelico}
\Tr(e^{-t\Delta_{\overline{\partial},1,0}^{\mathcal{F}}})\leq  Ct^{-2}
\end{equation}
for $t\in (0,1]$ and some constant  $C>0$.
\end{teo}

\begin{proof}
Using \eqref{cicci} and the Hodge star operator we have $*\circ \Delta_{\overline{\partial},2,1}=\Delta_{\partial,1,0}\circ *$ on $\Omega^{2,1}_c(\reg(V))$. By Prop. \ref{fall} it is easy to check that the previous equality implies that $*(\mathcal{D}(\Delta_{\overline{\partial},2,1}^{\mathcal{F}}))=\mathcal{D}(\Delta_{\partial,1,0}^{\mathcal{F}})$ and that  $*\circ \Delta_{\overline{\partial},2,1}^{\mathcal{F}}=\Delta_{\partial,1,0}^{\mathcal{F}}\circ *$. Moreover, by the K\"ahler identities, we have $\Delta_{\partial,1,0}=\Delta_{\overline{\partial},1,0}$ on $\Omega^{1,0}_c(\reg(V))$ and therefore $\Delta_{\partial,1,0}^{\mathcal{F}}=\Delta_{\overline{\partial},1,0}^{\mathcal{F}}$ on $L^2\Omega^{1,0}(\reg(V),h)$ as unbounded self-adjoint operators. In conclusion we have shown that $*\circ \Delta_{\overline{\partial},2,1}^{\mathcal{F}}=\Delta_{\overline{\partial},1,0}^{\mathcal{F}}\circ *$, that is any form $\omega\in L^2\Omega^{2,1}(\reg(V),h)$ lies in $\mathcal{D}(\Delta_{\overline{\partial},2,1}^{\mathcal{F}})$ if and only if $*\omega \in \mathcal{D}(\Delta_{\overline{\partial},1,0}^{\mathcal{F}})$ and if this is the case then  $*(\Delta_{\overline{\partial},2,1,
}^{\mathcal{F}}\omega)=\Delta_{\overline{\partial},1,0}^{\mathcal{F}}(*\omega)$. Now all the statements of this theorem follows by Th. \ref{coccabelladezios} because $*:L^2\Omega^{1,0}(\reg(V),h)\rightarrow L^2\Omega^{2,1}(\reg(V),h)$ is a unitary operator.
\end{proof}

As an immediate application of Th. \ref{supercoccabelladezio} we have the following \emph{Hodge theorem}.
\begin{cor}
In the setting of Th. \ref{lillottina}. The following properties hold true:
\begin{enumerate}
\item $\im(\overline{\partial}_{1,0,\min})$ is a closed subset of $L^2\Omega^{1,1}(\reg(V),h)$.
\item $\ker(\Delta_{\overline{\partial},1,1,\rel})\cong H^{1,1}_{2,\overline{\partial}_{\min}}(\reg(V),h)$.
\item $H^{1,1}_{2,\overline{\partial}_{\min}}(\reg(V),h)$ is finite dimensional.
\end{enumerate}
\end{cor}

\begin{proof}
According to  Th. \ref{supercoccabelladezio} and to Prop. \ref{fall} we know  that $\overline{\partial}_{1,0,\max}^t\circ \overline{\partial}_{1,0,\min}: L^2\Omega^{1,0}(\reg(V),h)\rightarrow L^2\Omega^{1,0}(\reg(V),h)$ has discrete spectrum and this in turn implies  in particular  that it is a Fredholm operator on its domain endowed with the graph norm. Therefore we can conclude that  $\im(\overline{\partial}_{1,0,\max}^t\circ \overline{\partial}_{1,0,\min})$ is closed in $L^2\Omega^{1,0}(\reg(V),h)$. By \eqref{cocapro} we have  the following two orthogonal decompositions for  $L^2\Omega^{1,0}(\reg(V),h)$: $$L^2\Omega^{1,0}(\reg(V),h)=\ker(\overline{\partial}_{1,0,\min})\oplus \overline{\im(\overline{\partial}^t_{1,0,\max})}$$ $$L^2\Omega^{1,0}(\reg(V),h)=\ker(\overline{\partial}^t_{1,0,\max}\circ \overline{\partial}_{1,0,\min})\oplus \overline{\im(\overline{\partial}^t_{1,0,\max}\circ \overline{\partial}_{1,0,\min})}.$$ Clearly $\ker(\overline{\partial}^t_{1,0,\max}\circ \overline{\partial}_{1,0,\min})=\ker(\overline{\partial}_{1,0,\min})$. Therefore we have the following chain of inclusions: $$\im(\overline{\partial}^t_{1,0,\max}\circ \overline{\partial}_{1,0,\min})\subset \im(\overline{\partial}^t_{1,0,\max})\subset \overline{\im(\overline{\partial}^t_{1,0,\max})}= \overline{\im(\overline{\partial}^t_{1,0,\max}\circ \overline{\partial}_{1,0,\min})}=\im(\overline{\partial}^t_{1,0,\max}\circ \overline{\partial}_{1,0,\min})$$ which in particular implies that $\overline{\im(\overline{\partial}^t_{1,0,\max})}=\im(\overline{\partial}^t_{1,0,\max})$ and therefore, taking the adjoint, $\overline{\im(\overline{\partial}_{1,0,\min})}=\im(\overline{\partial}_{1,0,\min})$. Hence the first point is established. Using again \eqref{cocapro} we easily get that $$\ker(\Delta_{\overline{\partial},1,1,\rel})\cong \frac{\ker(\overline{\partial}_{1,1,\min})}{\overline{\im(\overline{\partial}_{1,0,\min})}}.$$ On the other hand, by  the first point of this corollary, we know that $\overline{\im(\overline{\partial}_{1,0,\min})}=\im(\overline{\partial}_{1,0,\min})$. Thus we have 
$$\ker(\Delta_{\overline{\partial},1,1,\rel})\cong \frac{\ker(\overline{\partial}_{1,1,\min})}{\overline{\im(\overline{\partial}_{1,0,\min})}}=\frac{\ker(\overline{\partial}_{1,1,\min})}{\im(\overline{\partial}_{1,0,\min})}=H^{1,1}_{2,\overline{\partial}_{\min}}(\reg(V),h).$$ Finally, according to  \cite{OV}, we know that $H^{1,1}_{2,\overline{\partial}_{\max}}(\reg(V),h)$ is finite dimensional. By  virtue of the $L^2$-Serre duality, see Th. 2.3 in \cite{JRu}, and using the second point of this corollary, we can thus conclude that $H^{1,1}_{2,\overline{\partial}_{\min}}(\reg(V),h)$ is finite dimensional too.
\end{proof}

Assuming that $\sing(V)$ is made of \emph{isolated singularities} we can also deal  with the $L^2$-Dolbeault complex $(L^2\Omega^{0,q}(\reg(V),h), \overline{\partial}_{0,q,\max})$ and its associated Laplacians.

\begin{teo}
\label{nonnapapera}
Let $V\subset \mathbb{C}\mathbb{P}^n$ be a complex projective surface with only isolated singularities. For each $q=0,1,2$  the operator
\begin{equation}
\label{miguel}
\Delta_{\overline{\partial},0,q,\abs}:L^2\Omega^{0,q}(\reg(V),h)\rightarrow L^2\Omega^{0,q}(\reg(V),h)
\end{equation}
 has discrete spectrum. Let $$0\leq \lambda_1\leq \lambda_2\leq...\leq \lambda_k\leq...$$ be the eigenvalues of \eqref{miguel}. Then we have the following asymptotic inequality
\begin{equation}
\label{miguels}
\lim \inf \lambda_k k^{-\frac{1}{2}}>0
\end{equation}
as $k\rightarrow \infty$.\\ Consider now the heat operator associated to \eqref{miguel}
\begin{equation}
\label{miguelz}
e^{-t\Delta_{\overline{\partial},0,q,\abs}}:L^2\Omega^{0,q}(\reg(V),h)\rightarrow L^2\Omega^{0,q}(\reg(V),h).
\end{equation}
Then \eqref{miguelz} is a trace class operator and its trace satisfies the following estimates
\begin{equation}
\label{miguelx}
\Tr(e^{-t\Delta_{\overline{\partial},0,q,\abs}})\leq C_qt^{-2}
\end{equation}
for $t\in (0,1]$ and some constant $C_q>0$.
\end{teo}

\begin{proof}
According to \cite{OvRu} Th. 1.2 we know that $\left(\mathcal{D}(\Delta_{\overline{\partial},0,q,\abs})\cap \im(\mathcal{D}(\Delta_{\overline{\partial},0,q,\abs})\right)\hookrightarrow L^2\Omega^{0,q}(\reg(V),h)$ is a compact inclusion for each $q$. Moreover, by \cite{JRu}, we know that $H^{0,q}_{2,\overline{\partial}_{\max}}(\reg(V),h)$ is finite dimensional for each $q$ and therefore, using Prop. \ref{usipeti}, we get that  $\ker(\Delta_{\overline{\partial},0,q,\abs})$ is finite dimensional. In conclusion we have just shown  that $\left(\mathcal{D}(\Delta_{\overline{\partial},0,q,\abs})\right)\hookrightarrow L^2\Omega^{0,q}(\reg(V),h)$ is a compact inclusion for each $q$ and therefore we can conclude that \eqref{miguel} has discrete spectrum. Now, according to \cite{GL} Th 1.2, we know that $\overline{\partial}_{\min}=\overline{\partial}_{\max}$. In particular this implies that $\Delta^{\mathcal{F}}_{\overline{\partial}}=\Delta_{\overline{\partial},\abs}$, that is the absolute extension and the Friedrich extension of $\Delta_{\overline{\partial}}:C^{\infty}_c(\reg(V))\rightarrow C^{\infty}_c(\reg(V))$ coincide. Hence the statement of this theorem in the case $q=0$ follows by \cite{LT}. By \eqref{cicuta} and Prop. \ref{occhiodibue} we know that a form $\omega\in L^2\Omega^{2,0}(\reg(V),h)$ lies in $\mathcal{D}(\Delta_{\overline{\partial},2,0,\rel})$ if and only if $c_{2,0}(*\omega)\in \mathcal{D}(\Delta_{\overline{\partial},0,2,\abs})$ and if this is the case then we have $c_{2,0}(*(\Delta_{\overline{\partial},2,0,\rel}\omega)=\Delta_{\overline{\partial},0,2,\abs}(c_{2,0}(*\omega))$. Since  $c_{2,0}\circ *:L^2\Omega^{2,0}(\reg(V),h)\rightarrow L^2\Omega^{0,2}(\reg(V),g)$ is a continuous and  bijective $\mathbb{C}$-antilinear isomorphism  with continuous inverse the conclusion for the case $q=2$  follows by Th. \ref{canonical}. 
Finally the conclusion in the case $q=1$ follows by repeating the arguments, with the obvious modifications, used in the proof of Th. \ref{coccabelladezio} to prove the case $(2,1)$.
\end{proof}

An immediate application of the above theorem is the following  McKean-Singer formula for the complex $(L^2\Omega^{0,q}(\reg(V),h),\overline{\partial}_{0,q,\max})$. To this aim consider the operator
\begin{equation}
\label{cora}
\overline{\partial}_{\max}+\overline{\partial}^t_{0,1,\min}:L^2(\reg(V),h)\oplus L^2\Omega^{0,2}(\reg(V),h)\rightarrow L^2\Omega^{0,1}(\reg(V),h)
\end{equation}
 whose domain is $\mathcal{D}(\overline{\partial}_{\max})\oplus \mathcal{D}(\overline{\partial}_{0,1,\min}^t)\subset L^2(\reg(V),h)\oplus L^2\Omega^{0,2}(\reg(V),h)$. Its adjoint is 
\begin{equation}
\label{della}
\overline{\partial}_{0,1,\max}+\overline{\partial}^t_{\min}:L^2\Omega^{0,1}(\reg(V),h)\rightarrow L^2(\reg(V),h) \oplus L^2\Omega^{0,2}(\reg(V),h)
\end{equation}
with domain given by   $\mathcal{D}(\overline{\partial}_{0,1,\max})\cap \mathcal{D}(\overline{\partial}_{\min}^t)\subset L^2\Omega^{0,1}(\reg(V),h)$.

\begin{cor}
\label{mckean}
In the setting of Th. \ref{nonnapapera}. Let us label by $(\overline{\partial}_{0,\max}+\overline{\partial}_{0,\min}^t)^+$ the operator defined in \eqref{cora}. Then $(\overline{\partial}_{0,\max}+\overline{\partial}_{0,\min}^t)^+$ is a Fredholm operator and 
\begin{equation}
\label{stocca}
\ind((\overline{\partial}_{0,\max}+\overline{\partial}_{0,\min}^t)^+)=\sum_{q=0}^2(-1)^q\Tr(e^{-t\Delta_{\overline{\partial},0,q,\abs}}).
\end{equation}
In particular we have 
\begin{equation}
\label{fisso}
\chi(\tilde{V},\mathcal{O}(L))=\sum_{q=0}^2(-1)^q\Tr(e^{-t\Delta_{\overline{\partial},0,q,\abs}})
\end{equation}
where $\pi:\tilde{V}\rightarrow V$ is any resolution of $V$, $L$ is a suitable holomorphic line bundle on $\tilde{V}$ \footnote{We refer to \cite{PS} and to \cite{JRu} for the definition of $L$.} and $\chi(\tilde{V},\mathcal{O}(L))= \sum_{q=0}^{2}(-1)^q\dim(H^q(\tilde{V},\mathcal{O}(L)))$.
\end{cor} 

\begin{proof}
The equality \eqref{stocca} can be proved in the same way we proved \eqref{jilm}. The equality \eqref{fisso} follows by \eqref{stocca} and the results established in \cite{PS} and \cite{JRu}.
\end{proof}

In the last part of this section we collect various corollaries that arise, through  \eqref{cicuta} and Prop. \ref{occhiodibue}, as immediate consequences of the results proved so far. 
\begin{cor}
\label{sigari}
In the  setting of Th. \ref{lillottina}. For each $q=0,1,2$ the operator
\begin{equation}
\label{salicornia}
\Delta_{\overline{\partial},0,q,\rel}:L^2\Omega^{0,q}(\reg(V),h)\rightarrow L^2\Omega^{0,q}(\reg(V),h)
\end{equation}
has discrete spectrum. Let  $$0\leq \lambda_1\leq \lambda_2\leq...\leq \lambda_k\leq...$$ be the eigenvalues of \eqref{salicornia}; we have the following asymptotic  inequality 
\begin{equation}
\label{drive}
\lim \inf \lambda_k k^{-\frac{1}{2}}>0
\end{equation}
as $k\rightarrow \infty$.\\ Finally consider  the heat operator associated to \eqref{salicornia}
\begin{equation}
\label{ottenebrare}
e^{-t\Delta_{\overline{\partial},0,q,\rel}}:L^2\Omega^{0,q}(\reg(V),h)\rightarrow L^2\Omega^{0,q}(\reg(V),h).
\end{equation}
Then \eqref{ottenebrare} is a trace class operator and  we have the following estimate for its trace
\begin{equation}
\label{obnubilare}
\Tr(e^{-t\Delta_{\overline{\partial},0,q,\rel}})\leq  C_qt^{-2}
\end{equation}
for $t\in (0,1]$ and some constant $C_q>0$.
\end{cor}

\begin{proof}
Using \eqref{cicuta} and Prop. \ref{occhiodibue} we have that any form $\omega\in L^2\Omega^{2,q}(\reg(V),h)$ lies in $\mathcal{D}(\Delta_{\overline{\partial},2,q,\abs})$ if and only if $c_{2-q,0}(*\omega)\in \mathcal{D}(\Delta_{\overline{\partial},0,2-q,\rel})$ and if this is the case then we have $c_{2-q,0}(*(\Delta_{\overline{\partial},2,q,\abs}\omega))=\Delta_{\overline{\partial},0,2-q,\rel}(c_{2-q,0}(*\omega))$. Since  $c_{2-q,0}\circ *:L^2\Omega^{2,q}(\reg(V),h)\rightarrow L^2\Omega^{0,2-q}(\reg(V),g)$ is a continuous and  bijective $\mathbb{C}$-antilinear isomorphism with continuous inverse the conclusion  follows now by Th. \ref{lillottina} and Th. \ref{coccabelladezio}.
\end{proof}

\begin{cor}
\label{supercoccabelladeziow}
In the  setting of Th. \ref{lillottina}. For each $q=0,1,2$ the operator
\begin{equation}
\label{sacripantez}
\Delta_{\overline{\partial},0,q}^{\mathcal{F}}:L^2\Omega^{0,q}(\reg(V),h)\rightarrow L^2\Omega^{0,q}(\reg(V),h)
\end{equation}
has discrete spectrum. Let  $$0\leq \mu_1\leq \mu_2\leq...\leq \mu_k\leq...$$ be the eigenvalues of \eqref{sacripantez}; we have the following inequality $$\mu_k\geq \lambda_k$$ where $0\leq \lambda_1\leq...\leq\lambda_k\leq...$ are the eigenvalues of \eqref{salicornia}. Moreover we have the following asymptotic  inequality 
\begin{equation}
\label{zuccaz}
\lim \inf \mu_k k^{-\frac{1}{2}}>0
\end{equation}
as $k\rightarrow \infty$.\\ Consider now the heat operator associated to \eqref{sacripantez}
\begin{equation}
\label{gargantuescoz}
e^{-t\Delta_{\overline{\partial},0,q}^{\mathcal{F}}}:L^2\Omega^{0,q}(\reg(V),h)\rightarrow L^2\Omega^{0,q}(\reg(V),h).
\end{equation}
Then \eqref{gargantuescoz} is a trace class operator.  We have the following inequality $$\Tr(e^{-t\Delta_{\overline{\partial},0,q}^{\mathcal{F}}})\leq \Tr(e^{-t\Delta_{\overline{\partial},0,q,\rel}})$$ for every $t>0$ and 
furthermore   $\Tr(e^{-t\Delta_{\overline{\partial},0,q}^{\mathcal{F}}})$ satisfies the following estimates
\begin{equation}
\label{pantagruelicoz}
\Tr(e^{-t\Delta_{\overline{\partial},0,q}^{\mathcal{F}}})\leq  C_qt^{-2}
\end{equation}
for $t\in (0,1]$ and some constant $C_q>0$.
\end{cor}

\begin{proof}
This corollary follows by Th. \ref{coccabelladezios} using \eqref{cicuta} and Prop. \ref{occhiodibue} as in the  proof of Cor. \ref{sigari}.
\end{proof}

\begin{cor}
\label{sigariq}
In the  setting of Th. \ref{nonnapapera}.  For each $q=0,1,2$ the operator
\begin{equation}
\label{salicorniaq}
\Delta_{\overline{\partial},2,q,\rel}:L^2\Omega^{2,q}(\reg(V),h)\rightarrow L^2\Omega^{2,q}(\reg(V),h)
\end{equation}
has discrete spectrum. Let  $$0\leq \lambda_1\leq \lambda_2\leq...\leq \lambda_k\leq...$$ be the eigenvalues of \eqref{salicorniaq}; we have the following asymptotic  inequality 
\begin{equation}
\label{driveq}
\lim \inf \lambda_k k^{-\frac{1}{2}}>0
\end{equation}
as $k\rightarrow \infty$.\\ Consider now the heat operator associated to \eqref{salicorniaq}
\begin{equation}
\label{ottenebrareq}
e^{-t\Delta_{\overline{\partial},2,q,\rel}}:L^2\Omega^{2,q}(\reg(V),h)\rightarrow L^2\Omega^{2,q}(\reg(V),h).
\end{equation}
Then \eqref{ottenebrareq} is a trace class operator.  Furthermore we have the following estimate for  $\Tr(e^{-t\Delta_{\overline{\partial},2,q,\rel}})$
\begin{equation}
\label{obnubilareq}
\Tr(e^{-t\Delta_{\overline{\partial},2,q,\rel}})\leq  C_qt^{-2}
\end{equation}
for $t\in (0,1]$ and some constant $C_q>0$.
\end{cor}

\begin{proof}
This corollary follows by Th. \ref{nonnapapera} using \eqref{cicuta} and Prop. \ref{occhiodibue} as in the  proof of Cor. \ref{sigari}.
\end{proof}

\begin{cor}
\label{supercoccabelladeziof}
In the setting of Th. \ref{lillottina}. The operator
\begin{equation}
\label{sacripantef}
\Delta_{\overline{\partial},1,2}^{\mathcal{F}}:L^2\Omega^{1,2}(\reg(V),h)\rightarrow L^2\Omega^{1,2}(\reg(V),h)
\end{equation}
 has discrete spectrum. Let  $$0\leq \mu_1\leq \mu_2\leq...\leq \mu_k\leq...$$ be the eigenvalues of \eqref{sacripantef}; we have the following asymptotic  inequality 
\begin{equation}
\label{zuccaf}
\lim \inf \mu_k k^{-\frac{1}{2}}>0
\end{equation}
as $k\rightarrow \infty$.\\ Finally consider  the heat operator associated to \eqref{sacripantef}
\begin{equation}
\label{gargantuescof}
e^{-t\Delta_{\overline{\partial},1,2}^{\mathcal{F}}}:L^2\Omega^{1,2}(\reg(V),h)\rightarrow L^2\Omega^{1,2}(\reg(V),h).
\end{equation}
Then \eqref{gargantuescof} is a trace class operator and its trace satisfies the following estimate
\begin{equation}
\label{pantagruelicof}
\Tr(e^{-t\Delta_{\overline{\partial},1,2}^{\mathcal{F}}})\leq  Ct^{-2}
\end{equation}
for $t\in (0,1]$ and some constant $C>0$.
\end{cor}

\begin{proof}
The statements of this corollary follow by Th. \ref{supercoccabelladezio} using \eqref{cicuta} and Prop. \ref{occhiodibue} as we did in the  proof of Cor. \ref{sigari}.
\end{proof}

We conclude the paper with the following  McKean-Singer formula concerning  $(L^2\Omega^{0,q}(\reg(V),h),\overline{\partial}_{0,q,\min})$. Let $V$ and $h$ be as in Th. \ref{lillottina}. Consider the operator
\begin{equation}
\label{isabellaf}
\overline{\partial}_{\min}+\overline{\partial}^t_{0,1,\max}:L^2(\reg(V),h)\oplus L^2\Omega^{0,2}(\reg(V),h)\rightarrow L^2\Omega^{0,1}(\reg(V),h)
\end{equation}
 whose domain is $\mathcal{D}(\overline{\partial}_{\min})\oplus \mathcal{D}(\overline{\partial}_{0,1,\max}^t)\subset L^2(\reg(V),h)\oplus L^2\Omega^{0,2}(\reg(V),h)$. Its adjoint is 
\begin{equation}
\label{regef}
\overline{\partial}_{0,1,\min}+\overline{\partial}^t_{\max}:L^2\Omega^{0,1}(\reg(V),h)\rightarrow L^2(\reg(V),h) \oplus L^2\Omega^{0,2}(\reg(V),h)
\end{equation}
with domain given by   $\mathcal{D}(\overline{\partial}_{0,1,\min})\cap \mathcal{D}(\overline{\partial}_{\max}^t)\subset L^2\Omega^{0,1}(\reg(V),h)$.

\begin{cor}
\label{mckeanwx}
In the setting of Th. \ref{lillottina}. Let us label by $(\overline{\partial}_{0,\min}+\overline{\partial}_{0,\max}^t)^+$ the operator defined in \eqref{isabellaf}. Then $(\overline{\partial}_{0,\min}+\overline{\partial}_{0,\max}^t)^+$ is a Fredholm operator and 
\begin{equation}
\label{jilmqz}
\ind((\overline{\partial}_{0,\min}+\overline{\partial}_{0,\max}^t)^+)=\sum_{q=0}^2(-1)^q\Tr(e^{-t\Delta_{\overline{\partial},0,q,\rel}}).
\end{equation}
In particular we have 
\begin{equation}
\label{jilmxqz}
\chi(\tilde{V},\mathcal{O}_{\tilde{V}})=\sum_{q=0}^2(-1)^q\Tr(e^{-t\Delta_{\overline{\partial},0,q,\rel}})
\end{equation}
where $\pi:\tilde{V}\rightarrow V$ is any resolution of $V$ and $\chi(\tilde{V},\mathcal{O}_{\tilde{V}})=\sum_{q=0}^2(-1)^q\dim(H^{0,q}_{\overline{\partial}}(\tilde{V}))$. 
\end{cor} 

\begin{proof}
The equality \eqref{jilmqz} can be proved arguing as in the proof of  \eqref{jilm}. The equality \eqref{jilmxqz} follows by \eqref{jilmqz} and the results established in \cite{PS}.
\end{proof}

\end{document}